\documentclass[a4paper,11pt,twoside]{amsart}
\usepackage[english]{babel}
\usepackage[utf8]{inputenc}

\usepackage[a4paper,inner=3cm,outer=3cm,top=4cm,bottom=4cm,pdftex]{geometry}
\usepackage{fancyhdr}
\usepackage{fancyhdr}
\usepackage{enumerate}
\usepackage{color}
\usepackage{bold-extra}

\usepackage{color}
\usepackage{bold-extra}
\usepackage{ mathrsfs }

\usepackage{comment}
\usepackage{graphics}
\usepackage{aliascnt}
\usepackage[pdftex,citecolor=green,linkcolor=red]{hyperref}

\usepackage{amsmath}
\usepackage{amsfonts}
\usepackage{amssymb}
\usepackage{amsthm}
\usepackage{comment}
\usepackage{mathtools}

\newtheorem{innercustomgeneric}{\customgenericname}
\providecommand{\customgenericname}{}
\newcommand{\newcustomtheorem}[2]{%
  \newenvironment{#1}[1]
  {%
   \renewcommand\customgenericname{#2}%
   \renewcommand\theinnercustomgeneric{##1}%
   \innercustomgeneric
  }
  {\endinnercustomgeneric}
}

\newcustomtheorem{customthm}{Theorem}
\newtheorem{theorem}{Theorem}[section]
\newtheorem{corollary}[theorem]{Corollary}

\newtheorem{lemma}[theorem]{Lemma}
\newtheorem{proposition}[theorem]{Proposition}
\theoremstyle{definition}
\newtheorem{definition}[theorem]{Definition}
\newtheorem{remark}[theorem]{Remark}
\newtheorem{conjecture}[theorem]{Conjecture}

\numberwithin{equation}{section}

\renewcommand\d{\textnormal{d}}

\newcommand{\Mod}[1]{\ (\mathrm{mod}\ #1)}

\newcommand{\ra}{\rightarrow}
\newcommand{\bk}{\backslash}
\newcommand{\mc}{\mathcal}

\newcommand{\mb}{\mathbb}

\renewcommand{\ss}{\substack}

\newcommand{\e}{\varepsilon}

\newcommand{\mbf}{\boldsymbol}

\begin{document}

\title{On Elliott's conjecture and applications}

\author
{Oleksiy Klurman}
\address{School of Mathematics,
University of Bristol, Woodland Road, Bristol,  UK}
\email{lklurman@gmail.com}

\author{Alexander P. Mangerel}
\address{Department of Mathematical Sciences, Durham University, Upper Mountjoy Campus, Stockton Road, Durham, UK}
\email{smangerel@gmail.com}

\author{Joni Ter\"{a}v\"{a}inen}
\address{Department of Mathematics and Statistics, University of Turku, 20014 Turku, Finland}
\email{joni.p.teravainen@gmail.com}

\begin{abstract} 
   Let $f:\mathbb{N}\to \mathbb{D}$ be a multiplicative function. Under the merely  necessary assumption that $f$ is non-pretentious (in the sense of Granville and Soundararajan), we show that for any pair of distinct integer shifts $h_1,h_2$ the two-point correlation
         $$\frac{1}{x}\sum_{n\leq x}{f(n+h_1)\overline{f}(n+h_2)}$$
         tends to $0$ along a set of $x\in\mathbb{N}$ of full upper logarithmic density.
We also show that the same result holds for the $k$-point correlations
$$\frac{1}{x}\sum_{n\leq x}{f(n+h_1)\cdots f(n+h_k)}$$
if $k$ is odd and  $f$ is a real-valued non-pretentious function. Previously, the vanishing of correlations was known only under stronger non-pretentiousness hypotheses on $f$ by the works of Tao, and Tao and the third author.
      We derive several applications, including:
       \begin{itemize}
       \item A classification of $\pm 1$-valued completely multiplicative functions that omit a length four sign pattern, solving a 1974 conjecture of R.H. Hudson. 
       \item A proof that a class of ``Liouville-like" functions satisfies the unweighted Elliott conjecture of all orders, solving a problem of de la Rue.
       \item Constructing examples of multiplicative $f:\mathbb{N}\to \{-1,0,1\}$ with a given (unique) Furstenberg system,  answering a question of Lema\'nczyk.
       \item A density version of the Erd\H{o}s discrepancy theorem of Tao.
       \end{itemize}
\end{abstract}

\maketitle

\section{Introduction}

Let $f_1,\ldots, f_k:\mathbb{N}\to \mathbb{D}$ be multiplicative functions, where $\mathbb{D}$ denotes the closed unit disc of the complex plane. We shall study the problem of determining when the correlation averages
\begin{align}\label{eq0}
 \frac{1}{x}\sum_{n\leq x}f_1(n+h_1)\cdots f_k(n+h_k) \end{align}
converge to $0$ along a  subsequence (which is \emph{dense} with respect to some notion of density), with $h_1,\ldots, h_k\in \mathbb{N}$ being fixed, distinct integer shifts. It is well known that the behavior of the correlation averages~\eqref{eq0} depends crucially on whether the functions $f_j$ are \emph{pretentious} or not.

\begin{definition}[Pretentious and non-pretentious functions] \label{def:Pret}
 Let $f:\mathbb{N}\to \mathbb{D}$ be a multiplicative function. We say that $f$ is \emph{pretentious} if there exist a  Dirichlet character $\chi$ and a real number $t$ such that
 \begin{align*}
\mathbb{D}(f,\chi(n)n^{it};\infty)<\infty, 
\end{align*}
where for sequences $g_1,g_2 : \mathbb{N} \rightarrow \mathbb{D}$,
\begin{align}\label{eq:pret}
\mathbb{D}(g_1,g_2;x):= \left(\sum_{p\leq x}\frac{1-\textnormal{Re}(g_1(p)\overline{g_2(p)})}{p}\right)^{1/2}    
\end{align}
is the pretentious distance of Granville and Soundararajan~\cite{GS}. If $f$ is not pretentious, we say that it is non-pretentious\footnote{Non-pretentiousness is equivalent to the concept of aperiodicity sometimes used in connection with multiplicative functions. One calls $f$ \emph{aperiodic} if $f$ has mean $0$ on every infinite arithmetic progression. See~\cite[Th\'eor\`eme 1]{delange1983} for the equivalence of these notions.}. 
\end{definition}

A deep and influential conjecture of Elliott~\cite{elliott-book} states that if $f_1,\ldots, f_k:\mathbb{N}\to \mathbb{D}$ are multiplicative functions with $f_1$ non-pretentious, then the correlation averages~\eqref{eq0} tend to $0$ as $x\to \infty$. Elliott's conjecture includes as a special case Chowla's conjecture~\cite{chowla}, which is the case where $f_1=\cdots=f_k=\mu$, with $\mu$ the M\"obius function. 

The original formulation of Elliott's conjecture turns out to be technically false, with a counterexample constructed by Matom\"aki, Radziwi\l{}\l{} and Tao~\cite[Appendix B]{MRT}. The conjecture can be corrected to a form believed to be true by strengthening the hypothesis on the distances $\mathbb{D}(f_1,\chi(n)n^{it},x)$; see Conjecture~\ref{conj_elliott_Tao} below (we pose as Conjecture~\ref{conj_elliott:modified} another possible way to correct Elliott's conjecture without strengthening the hypothesis on $f_1$). 

Tao~\cite{tao-chowla}, and Tao and the third author~\cite{tt-duke} considered a logarithmically averaged version of Elliott's conjecture in the cases $k=2$ and $k\geq 3$, respectively. These results were strengthened by the same authors in~\cite[Corollaries 1.8 and 1.13]{tt-ant} to cover asymptotic averages on a dense sequence of scales.

\begin{customthm}{A}[\cite{tt-ant}]\label{thm_A}
Let $k\geq 2$, and let $h_1,\ldots, h_k\in \mathbb{N}$ be distinct. Let $f_1,\ldots, f_k:\mathbb{N}\to \mathbb{D}$ be multiplicative functions. 
\begin{enumerate}
    \item Suppose that $f_1$ satisfies
    \begin{align}\label{eq28}
     \inf_{|t|\leq x}\mathbb{D}(f_1,\chi(n)n^{it};x)\xrightarrow[x\to \infty]{}\infty  
    \end{align}
    for every Dirichlet character $\chi$.
    Then there exists a set $\mathcal{X}\subset \mathbb{N}$ with\footnote{By $\delta_{\log}(A)=\lim_{X\to \infty}\frac{1}{\log X}\sum_{n\leq X}1_{A}(n)/n$ we denote the logarithmic density of a set $A$, and by $\delta_{\log \log}(A)=\lim_{X\to \infty}\frac{1}{\log \log X}\sum_{2\leq n\leq X}1_{A}(n)/(n\log n)$ we denote its double-logarithmic density, provided these limits exist. The upper densities $\delta^{+}_{\log}$, $\delta^+_{\log \log}$ are defined with $\limsup$ in place of $\lim$, and the lower densities $\delta^{-}_{\log}$, $\delta^-_{\log \log}$ are defined with $\liminf$ in place of $\lim$. For a set $\mathcal{X}=\{x_1,x_2,\ldots\}\subset \mathbb{N}$ with $x_1<x_2<\cdots$ and $F:\mathbb{N}\to \mathbb{C}$, we define $\displaystyle\lim_{\substack{x\to \infty\\x\in \mathcal{X}}}F(x):=\lim_{j\to \infty}F(x_j)$.} $\delta_{\log}(\mathcal{X})=1$ such that we have 
    \begin{align*}
     \lim_{\substack{x\to \infty\\x\in \mathcal{X}}}\frac{1}{x}\sum_{n\leq x}f_1(n+h_1)f_2(n+h_2)=0. 
    \end{align*}
\item    Suppose that the product $f_1\cdots f_k$ satisfies
    \begin{align}\label{eq29}
\limsup_{x\to \infty} \frac{1}{\log \log x}\mathbb{D}(f_1\cdots f_k,\chi(n)n^{it};x)^2>0
    \end{align}
    for every Dirichlet character $\chi$ and real number $t$. 
    Then there exists a set $\mathcal{X}\subset \mathbb{N}$ with $\delta_{\log}(\mathcal{X})=1$ such that we have 
    \begin{align*}
     \lim_{\substack{x\to \infty\\x\in \mathcal{X}}}\frac{1}{x}\sum_{n\leq x}f_1(n+h_1)\cdots f_k(n+h_k)=0. 
    \end{align*}
\end{enumerate}
\end{customthm}

Theorem~\ref{thm_A} covers most of the known cases in the literature where the correlation averages~\eqref{eq0} are known to tend to $0$ along a dense subsequence. See however also the work of Gomilko, Lema\'nczyk and de la Rue~\cite{lemanczyk-MRT} for a proof that functions in the MRT class (a set of functions including the Matom\"aki--Radziwi\l{}\l{}--Tao counterexample) have vanishing autocorrelations along a subsequence. We also mention a result of the first author~\cite{klurman} which complements Theorem~\ref{thm_A} by giving a formula for the limit of the correlation averages~\eqref{eq0} in the case where $f_1,\ldots, f_k$ are all pretentious (the limit is ``typically'' nonzero, but in some cases it may vanish). 

Theorem~\ref{thm_A} and the other aforementioned results do not cover all bounded multiplicative functions even in the case of $2$-point correlations. Our first main result bridges this gap.
\begin{theorem}[Two-point Elliott conjecture at almost all scales]\label{thm_elliott_2point}
Let $f_1,f_2:\mathbb{N}\to \mathbb{D}$ be multiplicative functions. Suppose that $f_1$ is non-pretentious. Then there exists a set $\mathcal{X}\subset \mathbb{N}$ with $\delta_{\log}^{+}(\mathcal{X})=1$ such that for any distinct $h_1,h_2\in \mathbb{N}$ we have
\begin{align*}
\lim_{\substack{x\to \infty\\x\in \mathcal{X}}}\frac{1}{x}\sum_{n\leq x}f_1(n+h_1)f_2(n+h_k)=0.    
\end{align*}    
\end{theorem}

This result settles the two-point case of Problem 7.3 of the 2018 AIM conference on Sarnak's conjecture~\cite{AIM} in a strong form. This problem asks if the non-pretentiousness of $f$ is enough for the two-point correlations of $f$ and $\overline{f}$ to tend to $0$ along a subsequence. We note that a logarithmically averaged version of this conjecture was solved very recently by Frantzikinakis, Lema\'nczyk and de la Rue in~\cite[Proposition C.1]{flr}. Thus, Theorem~\ref{thm_elliott_2point} strengthens~\cite[Proposition C.1]{flr}. 

Theorem~\ref{thm_elliott_2point} is essentially optimal. The condition that $f_1$ is non-pretentious is clearly necessary and the condition $\delta^{+}_{\log}(\mathcal{X})=1$ cannot be strengthened to $\delta_{\log}(\mathcal{X})=1$, or even to $\delta_{\log \log}(\mathcal{X})=1$ (which is weaker) as the following result shows.

\begin{proposition}\label{prop_MRT}
There exists a non-pretentious multiplicative function $f:\mathbb{N}\to \mathbb{D}$  and a set $\mathcal{X}\subset \mathbb{N}$ with $\delta_{\log \log}^{+}(\mathcal{X})=1$ such that
\begin{align*}
\lim_{\substack{x\to \infty\\x\in \mathcal{X}}}\left|\frac{1}{x}\sum_{n\leq x}f(n)\overline{f}(n+1)\right|= 1. 
\end{align*}
\end{proposition}

 Our second main result (which is a consequence of the more general Theorem~\ref{thm_elliott_higher} below) shows that the odd order cases of Elliott's conjecture hold for any non-pretentious multiplicative function $f:\mathbb{N}\to [-1,1]$ along a dense sequence of scales.

\begin{corollary}[Odd order Elliott conjecture at almost all scales for real-valued functions]\label{cor_odd} Let $k\geq 1$ be an odd integer. Let $f:\mathbb{N}\to [-1,1]$ be a multiplicative function. Suppose that $\mathbb{D}(f,\chi;\infty)=\infty$ for any real Dirichlet character $\chi$. Then there exists a set $\mathcal{X}\subset \mathbb{N}$ with $\delta^{+}_{\log \log}(\mathcal{X})=1$ such that for any distinct $h_1,\ldots, h_k\in \mathbb{N}$ we have
\begin{align*}
\lim_{\substack{x\to \infty\\x\in \mathcal{X}}}\frac{1}{x}\sum_{n\leq x}f(n+h_1)\cdots f(n+h_k)=0.    
\end{align*}
\end{corollary}

\subsection{Further main results}

Our next main result interpolates between pretentious functions covered by~\cite{klurman} and the class of functions covered by Theorem~\ref{thm_A} by showing that functions that are non-pretentious but not too strongly so have vanishing autocorrelations of \emph{any} order along a dense sequence of scales. The class of functions to which our next main theorem applies are what we call \emph{moderately non-pretentious} functions. 

\begin{definition}[Moderately non-pretentious functions] We say that a multiplicative function $f:\mathbb{N}\to \mathbb{D}$ is \emph{moderately non-pretentious} if $f$ is non-pretentious and there exists a constant $A\geq 1$ such that 
\begin{align*}
\lim_{X\to \infty}\frac{1}{\log \log X}\inf_{|t|\leq X^{A}}\,\min_{\substack{\chi\pmod q\\q\leq (\log X)^{A}}}\mathbb{D}(f,\chi(n)n^{it};X)^2=0.  
\end{align*}
\end{definition}

This class of functions is clearly closed under complex conjugation and under multiplication by a twisted Dirichlet character $\xi(n)n^{iu}$. Moreover, by the pretentious triangle inequality~\eqref{eq:triangle} we see that it is also closed under products. In view of Theorem~\ref{thm_A}, important examples to keep in mind are those non-pretentious $f$ for which $\mathbb{D}(f,\chi(n)n^{it};X)^2=o(\log \log X)$ for some fixed $\chi,t$ or $\inf_{|t|\leq X}\mathbb{D}(f,\chi(n)n^{it};X)=O(1)$ for some fixed $\chi$.\footnote{Note that our notion of moderate non-pretentious is somewhat disjoint from the notion of strong aperiodicity, i.e.~\eqref{eq28}; neither condition implies the other.} 

\begin{theorem}[Moderately non-pretentious functions satisfy Elliott's conjecture at almost all scales]\label{thm_main} Let $k\geq 1$. Let $f_1,\ldots, f_k:\mathbb{N}\to \mathbb{D}$ be multiplicative functions, with $f_1$ moderately non-pretentious and each of $f_2,\ldots, f_k$ either moderately non-pretentious or pretentious. Then there exists a set $\mathcal{X}\subset \mathbb{N}$, depending only on the set $\{f_1,\ldots, f_k\}$, with $\delta^{+}_{\log \log}(\mathcal{X})=1$ such that for any distinct $h_1,\ldots, h_k\in \mathbb{N}$ we have
\begin{align*}
\lim_{\substack{x\to \infty\\x\in \mathcal{X}}}\frac{1}{x}\sum_{n\leq x}f_1(n+h_1)\cdots f_k(n+h_k)=0.    
\end{align*}   
\end{theorem}

We can use Theorem~\ref{thm_main} to prove that the condition~\eqref{eq29} on $f_1\cdots f_k$ in Theorem~\ref{thm_A}(2) can be relaxed to just $f_1\cdots f_k$ being non-pretentious in the case where $f_j$ are powers of the same multiplicative function taking values in (the convex hull of) the roots of unity of a certain order. 

\begin{theorem}[Higher order Elliott conjecture for functions with non-pretentious powers]\label{thm_elliott_higher}
Let $k\geq 1$ and $e_1,\ldots, e_k\geq 1$ be integers.  Let $d\geq 1$ be an integer with $(d,e_1+\cdots +e_k)=1$, and let $f:\mathbb{N}\to \mathbb{D}$ be a multiplicative function taking values in the convex hull of the $d$th roots of unity. Suppose also that $\mathbb{D}(f^{e_1},\chi;\infty)=\infty$ for any Dirichlet character $\chi$.  Then there exists a set $\mathcal{X}\subset \mathbb{N}$ with $\delta_{\log \log}^{+}(\mathcal{X})=1$ such that for any fixed, distinct $h_1,\ldots, h_k\in \mathbb{N}$ we have
\begin{align*}
\lim_{\substack{x\to \infty\\x\in \mathcal{X}}}\frac{1}{x}\sum_{n\leq x}f(n+h_1)^{e_1}\cdots f(n+h_k)^{e_k}=0.  
\end{align*}
\end{theorem}

\section{Applications}

Previous results on correlations of bounded multiplicative functions have had many applications, including Tao's solution of the Erd\H{o}s discrepancy problem~\cite{tao-edp}, and a proof of the logarithmic form of Sarnak's conjecture for uniquely ergodic systems by Frantzikinakis and Host~\cite{fh-annals}. See also~\cite{f-erg-descrete},~\cite{fh-IMRN},~\cite{lemanczyk-MRT},~\cite{RAP},~\cite{flr} for some works on classifying Furstenberg systems of multiplicative functions. Results on correlations have also played an important role in establishing various ``rigidity theorems" for multiplicative functions, see e.g.~\cite{klurman},~\cite{klurman-mangerel-rigidity},~\cite{ant-km},~\cite{tams-kmpt},~\cite{f-descrete},~\cite{aymone},  and results on sign patterns in multiplicative sequences, see e.g.~\cite{MRT-sign},~\cite{tao-chowla},~\cite{tt-duke},~\cite{km-effective},~\cite{fh-annals},~\cite{tt-FMS},~\cite{mcnamara},~\cite{MRTTZ}.

The new results from the preceding section enable progress on several  problems related to correlations, sign patterns and ``rigidity phenomena" for multiplicative functions, where e.g. the lack of a version of the three-point Elliott conjecture with weaker hypotheses was previously a problem. Further consequences of these results will be explored in future works. Here we demonstrate their applicability in solving several open problems. We remark that for some of our applications it is crucial that the set $\mathcal{X}\subset \mathbb{N}$ in our main theorems has full upper density in the sense $\delta_{\log}^{+}(\mathcal{X})=1$ (rather than just $\mathcal{X}$ being infinite).

\subsection{Chowla property for ``Liouville-like'' functions}

Using the proof ideas of Theorem~\ref{thm_main}, we can settle in a strong form a problem posed by de la Rue at the AIM conference on Sarnak's conjecture (Problem 7.2 in the list~\cite{AIM}). This problem asked if there is a set $\mathcal{A}=\{a_1,a_2,\ldots\}\subset \mathbb{N}$ with $\sum_{j\geq 1}1/a_j=\infty$ such that the ``Liouville-like'' function
\begin{align*}
 f(n)=(-1)^{|\{j\in \mathbb{N}:\,\, a_j\mid n\}|}   
\end{align*}
satisfies the \emph{Chowla property},
i.e. the autocorrelations of $f$ of any order tend to $0$ (or in other words Elliott's conjecture holds with $f_1=\cdots =f_k=f$ for any $k\geq 1$).
In what follows, for a subset $\mathcal{P}$ of the primes $\mathbb{P}$, $\omega_{\mathcal{P}}(n)$ denotes the number of distinct prime divisors of $n$ from $\mathcal{P}$ and $\Omega_{\mathcal{P}}(n)$ denotes the number of  prime divisors of $n$ from $\mathcal{P}$ counted with multiplicities. 

\begin{theorem}[Elliott's conjecture for Liouville-like functions involving sparse sets]\label{thm_omega} Let $f(n)=(-1)^{\omega_{\mathcal{P}}(n)}$ or $f(n)=(-1)^{\Omega_{\mathcal{P}}(n)}$, where $\mathcal{P}\subset \mathbb{P}$ has relative density\footnote{The relative density of a subset $\mathcal{P}\subset \mathbb{P}$ within $\mathbb{P}$ is defined as $\lim_{x\to \infty}|\mathcal{P}\cap [1,x]|/|\mathbb{P}\cap [1,x]|$.} $0$ in $\mathbb{P}$ and $\sum_{p\in \mathcal{P}}\frac{1}{p}=\infty$. Then, for any  $k\geq 1$ and $a_1,\ldots, a_k,h_1,\ldots, h_k\in \mathbb{N}$ with $a_ih_j\neq a_jh_i$ whenever $i\neq j$, we have 
  \begin{align*}
  \lim_{x\to \infty}\frac{1}{x}\sum_{n\leq x}f(a_1n+h_1)\cdots f(a_kn+h_k)=0.
  \end{align*}
\end{theorem}

To the best of our knowledge, this result gives the first construction of a multiplicative function $f:\mathbb{N}\to [-1,1]$ whose autocorrelations of all orders vanish (for random $f$, however, one can show that the autocorrelations vanish almost surely).

By a general implication from Chowla-type statements to Sarnak-type statements, we deduce that these Liouville-like functions are orthogonal to all deterministic sequences. In what follows, we say that a sequence $a:\mathbb{N}\to \mathbb{D}$ is \emph{deterministic} if there exists a topological dynamical system $(X,T)$ of zero entropy, a continuous function $F:X\to \mathbb{C}$ and a point $y\in X$ such that $a(n)=F(T^n y)$ for all $n\geq 1$.

\begin{corollary}[A Sarnak-type conjecture for Liouville-like functions involving sparse sets]\label{cor:sarnak} Let $f(n)=(-1)^{\omega_{\mathcal{P}}(n)}$ or $f(n)=(-1)^{\Omega_{\mathcal{P}}(n)}$, where $\mathcal{P}\subset \mathbb{P}$ has relative density $0$ in $\mathbb{P}$ and $\sum_{p\in \mathcal{P}}\frac{1}{p}=\infty$. Let $a:\mathbb{N}\to \mathbb{D}$ be any deterministic sequence. Then
\begin{align*}
\lim_{x\to \infty}\frac{1}{x}\sum_{n\leq x}f(n)a(n)=0.    
\end{align*}
\end{corollary}

\subsection{Hudson's conjecture}
Motivated by the question of the longest consecutive runs of quadratic residues, we say that a completely multiplicative function $f: \mb{N} \ra \{-1,+1\}$ has \emph{length} $\ell$ if $\ell$ is the maximal number of consecutive $+1$'s that occur in the sequence $(f(n))_{n \geq 1}$. That is, $f$ has length $\ell$ if there exists $m \in \mb{N}$ such that 
$$
(f(m),f(m+1),\ldots,f(m+\ell-1)) = (+1,+1,\ldots,+1),
$$
but for any $n \in \mb{N}$,
$$
(f(n),f(n+1),\ldots,f(n+\ell)) \neq (+1,\ldots,+1).
$$
We denote by $F_{2,\ell}$ the set of all completely multiplicative $f: \mb{N} \ra \{-1,1\}$ of length $\ell$. In 1962,
D.H. Lehmer and E. Lehmer~\cite{Lehmer} noted (and later Mills conjectured~\cite{Mills}) that  
$F_{2,2}$ consists of two functions, $f_3^+$ and $f_3^-$, defined by complete multiplicativity as follows: for $p \neq 3$ we set $f_3^+(p) = f_3^-(p) = \left(\frac{p}{3}\right)$, and for $p = 3$ we set $f_3^+(3) = -f_3^-(3) = +1$. That is, $f_3^{\pm}$ are defined by taking the Legendre symbol modulo $3$ and \emph{modifying} the value at its conductor, $3$, so that it takes values in $\{-1,+1\}.$ This conjecture was later proved by Schur~\cite{Schur}.

Almost fifty years ago, Hudson~\cite{hudson} considered the much more difficult classification problem for $F_{2,3}$, i.e., completely multiplicative $\pm 1$-valued functions with length $3.$ He conjectured the following:
\begin{conjecture}[Hudson, 1974]\label{4pattern}
    
The set $F_{2,3}$ consists of precisely $13$ functions, namely $F_{2,3} = S \cup T$, where
$$
S := \{f_p^{+} : p \in \{5,7,11,13,53\}\} \cup \{f_p^{-} : p \in \{5,7,11,13,53\}\},
$$
$f_p^{\pm}$ are modified Legendre symbols ($f_p^{\pm}(p')=\left(\frac{p'}{p}\right)$ for $p'\neq p$ and $f_p^{\pm}(p)=\pm 1$), and $T = \{g_1,g_2,g_3\}$, where the completely multiplicative functions $g_i$ are defined as
\begin{alignat*}{2}
&g_1(p) = g_2(p) = (-1)^{\tfrac{p-1}{2}} \text{ if $p \neq 2$}, \quad &&g_1(2) = -g_2(2) = +1 \\
&g_3(p) = +1 \text{ if $p \neq 2$}, \quad &&g_3(2) = -1 .
\end{alignat*}
\end{conjecture}

We resolve Conjecture~\ref{4pattern} in a strong form.
\begin{theorem}[Hudson's conjecture, density version]\label{thm_hudson} Hudson's conjecture is true, i.e., $F_{2,3} = S \cup T$. Moreover, if $f: \mathbb{N} \to \{-1,+1\}$ is a completely multiplicative function for which $f \notin F_{2,3}$ then
\begin{align*}
\limsup_{x\to \infty}\frac{1}{x}{|\{n\leq x:\,\, f(n+1)=f(n+2)=f(n+3)=f(n+4)=+1\}|}>0.    
\end{align*}
\end{theorem}
In principle, our methods allow us to characterize all completely multiplicative functions $f:\mathbb{N}\to \{-1,+1\}$ that omit a given sign pattern of length $\le 4$ (see also related work of Hildebrand~\cite{Hildcons} which studies sign patterns of length $\le 3$). Unlike in the case of Conjecture~\ref{4pattern},
for some of these patterns infinitely many exceptional functions exist. For example, for any fixed prime $p_0$, the completely multiplicative function satisfying $f(p_0)=-1$ and $f(p)=+1$ for all $p \neq p_0$, avoids the sign pattern $\{-1,-1,+1,+1\}$.
\subsection{A density version of the Erd\H{o}s discrepancy theorem}

The next application relates to the celebrated Erd\H{o}s discrepancy theorem of Tao~\cite{tao-edp} that settled an old conjecture of Erd\H{o}s~\cite{erdos-edp}. The Erd\H{o}s discrepancy theorem tells us, in particular, that if $f:\mathbb{N}\to \{-1,+1\}$ is completely multiplicative, then  the partial sums
\begin{align*}
 \sum_{n\leq x}f(n)   
\end{align*}
are unbounded. It is a natural question to ask if one can quantify the size of the set of $x$ along which the partial sums are unbounded. 
Our next theorem establishes such a density version.

\begin{theorem}[Density version of the Erd\H{o}s discrepancy theorem]\label{thm_edp}
Let $f:\mathbb{N}\to \{-1,+1\}$ be completely multiplicative. Let $M\geq 1$, and let
\begin{align*}
\mathcal{X}_M:=\Bigg\{x\in \mathbb{N}:\,\, \Bigg|\sum_{n\leq x}f(n)\Bigg|\geq M\Bigg\}.    
\end{align*}
\begin{enumerate}
    \item If $f$ is non-pretentious, then $\delta_{\log }^{+}(\mathcal{X}_M)>0$.
    
    \item If $f$ is pretentious, then\footnote{By $d(A)=\lim_{X\to \infty}\frac{1}{X}\sum_{n\leq X}1_A(n)$ we denote the asymptotic density of a set $A$. The upper density $d^{+}$ is defined with $\limsup$ in place of $\lim$ and the lower density $d^{-}$ is defined with $\liminf$ in place of $\lim$.} $d^{-}(\mathcal{X}_M)>0$.
\end{enumerate}
\end{theorem}

In particular, in any case, we have $\delta_{\log}^{+}(\mathcal{X}_M)>0$.

\subsection{Furstenberg systems of multiplicative functions}
An important object that has played a pivotal role in recent advances towards the M\"obius disjointness conjecture of Sarnak~\cite{Sar-dis} is the notion of
 Furstenberg systems. To facilitate our discussion, we introduce several definitions. Given any sequence $\boldsymbol{\eta}:\mathbb{N}\to \mathbb{D}$, we interpret $\boldsymbol{\eta}$ as an infinite word on $\mathbb{D}^{\mathbb{Z}}$ (by extending $\boldsymbol{\eta}$ for instance as $0$ to the nonpositive integers). Let $T$ be the left-shift on $\mathbb{D}^{\mathbb{Z}}$, that is, $T(\boldsymbol{y})(j)=\boldsymbol{y}(j+1)$ for $\boldsymbol{y}\in \mathbb{D}^{\mathbb{Z}}$. Equip $\mathbb{D}^{\mathbb{Z}}$ with the product topology, and let $\mathcal{B}$ be the Borel $\sigma$-algebra of $\mathbb{D}^{\mathbb{Z}}$.

We say that the measure-preserving system $(\mathbb{D}^{\mathbb{Z}},\mathcal{B},T,\nu)$ with $\nu$ a $T$-invariant Borel probability measure on $\mathbb{D}^{\mathbb{Z}}$ is a \emph{Furstenberg system} of $\boldsymbol{\eta}$ (and $\nu$ is a \emph{Furstenberg measure} of $\boldsymbol{\eta}$) if there exists an infinite sequence $(x_j)_{j\geq 1}$ tending to infinity such that
\begin{align}\label{eq:furst}
\lim_{j\to \infty}\frac{1}{x_j}\sum_{n\leq x_j}h(T^n\boldsymbol{\eta})=\int_{\mathbb{D}^{\mathbb{Z}}} h\, d\nu\,\, \textnormal{ for all continuous } h:\mathbb{D}^{\mathbb{Z}}\to \mathbb{C}.
\end{align}
It is not difficult to see (using compactness of the space of Borel probability measures on $\mathbb{D}^{\mathbb{Z}}$ in the weak*-topology) that any sequence $\boldsymbol{\eta}$ has at least one Furstenberg system and that (by the Riesz representation theorem) the Furstenberg system is unique if and only if the limit on the left-hand side of~\eqref{eq:furst} exists with $x_j=j$. Moreover, it is not hard to show that the Furstenberg measure $\nu$ is always $T$-invariant and supported on $\overline{\{T^n\boldsymbol{\eta}:\,\, n\in \mathbb{Z}\}}$. For more on Furstenberg systems of multiplicative functions, see e.g.~\cite{fh-IMRN},~\cite{lemanczyk-MRT}. 

We say that a measure-preserving system $(X^{\mathbb{Z}},\mathcal{B},T,\nu)$ is a \emph{Bernoulli system} if $\nu=\kappa^{\otimes \mathbb{Z}}$ for some measure $\kappa$ on $X$.

At the AIM workshop on Sarnak's conjecture~\cite{AIM}, Lema\'nczyk asked for a 
 construction of a (deterministic) multiplicative function taking values in $\{-1,+1\}$ and having as its unique Furstenberg measure the Bernoulli measure $(1/2,1/2)^{\mathbb{Z}}$, assigning mass $1/2$ to each of $\pm 1$. We note that all functions satisfying Theorem~\ref{thm_omega} provide corresponding examples and deduce this from the following more general result.

\begin{theorem}[Constructing multiplicative functions with a given Furstenberg system]\label{thm:furst} Let $g:\mathbb{N}\to \{0,1\}$ be a pretentious multiplicative function. Let $\mathcal{S}=(\{0,1\}^{\mathbb{Z}},\mathcal{B},T,\nu)$ be its (necessarily unique) Furstenberg system. Then there is (an explicit) multiplicative function $f:\mathbb{N}\to \{-1,0,+1\}$ whose unique Furstenberg system is isomorphic to the direct product of $\mathcal{S}$ and a Bernoulli system. In fact, one can take $f=\tilde{f}g$, where $\tilde{f}$ is any of the functions in Theorem~\ref{thm_omega}.  
\end{theorem}

\begin{remark} Specializing to $g=\mu^2,$ Theorem~\ref{thm:furst} provides an example of $f$ whose Furstenberg system is isomorphic to the direct product of the Mirsky system (corresponding to the Mirsky measure; see e.g.~\cite{HKPLR}) and a Bernoulli system. Chowla's conjecture is equivalent to the statement that this holds for $f=\mu$. 
\end{remark}

Frantzikinakis and Host~\cite[Conjecture 1]{fh-annals} conjectured that the Furstenberg system of any multiplicative function $f:\mathbb{N}\to [-1,1]$ is ergodic and isomorphic to the direct product of an ergodic odometer (an inverse limit of periodic systems) and a Bernoulli system\footnote{They defined Furstenberg systems in terms of logarithmic averages, but the same should be true for Furstenberg systems defined in terms of Ces\`aro averages as well.}. Theorem~\ref{thm:furst} goes in the other direction: it shows with an explicit construction the existence of a non-pretentious multiplicative function whose Furstenberg system is up to isomorphism the direct product of the Furstenberg system of a pretentious multiplicative function (which is necessarily an ergodic odometer by~\cite[Theorem 1.7]{RAP}) and a Bernoulli system. We conjecture that these systems are a complete characterization of Furstenberg systems of non-pretentious multiplicative functions taking values in $\{-1,0,+1\}$. 

\begin{conjecture}[Classification of Furstenberg systems]\label{conj:furst} Let $f:\mathbb{N}\to \{-1,0,+1\}$ be a non-pretentious multiplicative function. Then it has a unique Furstenberg system, which is isomorphic to the direct product of the Furstenberg system of the function $|f|$ with a Bernoulli system. 
\end{conjecture}

We support Conjecture~\ref{conj:furst} by showing that it is implied by Conjecture~\ref{conj_elliott_Tao}, an alternative form of Elliott's conjecture due to Tao~\cite{tao-blog} and to be discussed below.

\begin{proposition}\label{prop:conj} Conjecture~\ref{conj_elliott_Tao} implies Conjecture~\ref{conj:furst}. 
\end{proposition}

Let us remark that for general complex-valued multiplicative functions $f:\mathbb{N}\to \mathbb{D}$ even determining the right conjectural statement about the structure of their Furstenberg systems does not seem easy, in light for example of the work~\cite{lemanczyk-MRT}, where the authors construct multiplicative functions whose Furstenberg systems have as their ergodic components affine extensions of irrational  rotations.

We conclude by mentioning one other direct application of our results to Furstenberg systems of multiplicative functions. Very recently,  Frantzikinakis, Lema\'nczyk and de la Rue~\cite[Theorem 2.7]{flr} proved that if $f:\mathbb{N}\to\mathbb{D}$ is any pretentious multiplicative function,  all of its Furstenberg systems for logarithmic or Ces\`aro averages have  has rational discrete spectrum. For the case of logarithmic averages, they also proved the converse implication in~\cite[Appendix C]{flr}. By following
the arguments in~\cite[Appendix C]{flr} verbatim but using our Theorem~\ref{thm_elliott_2point} in place of~\cite[Proposition C.1]{flr}, we can extend the converse direction to the case of Ces\`aro averages (the same observation is also made in~\cite{flr}). 

We shall explore some further ergodic-theoretic consequences in a future work.

\subsection{Alternative to Elliott's conjecture for non-pretentious functions}

Our Theorems~\ref{thm_elliott_2point} and~\ref{thm_main} give some indication that Elliott's original formulation of his conjecture may be \emph{almost} right, in the sense that whenever $f_1$ is non-pretentious the correlation average~\eqref{eq0} tends to $0$ along a sequence of scales that is dense with respect to the upper logarithmic density. We pose this as the following conjecture.

\begin{conjecture}[A modified Elliott conjecture for non-pretentious functions]\label{conj_elliott:modified} Let $k\geq 1$, and let $f_1,\ldots, f_k:\mathbb{N}\to \mathbb{D}$ be multiplicative functions with $f_1$ non-pretentious. Then there exists a set $\mathcal{X}\subset \mathbb{N}$, depending only on $ f_1$, with $\delta^{+}_{\log}(\mathcal{X})=1$ such that for any distinct $h_1,\ldots, h_k\in \mathbb{N}$ we have
\begin{align*}
 \lim_{\substack{\substack{x\to \infty}\\x\in \mathcal{X}}} \frac{1}{x}\sum_{n\leq x}f_1(n+h_1)\cdots f_k(n+h_k)=0.  
\end{align*}
\end{conjecture}

We also pose a generalized Sarnak conjecture for multiplicative functions in this fashion.

\begin{conjecture}[A generalized Sarnak conjecture for multiplicative functions]\label{conj_sarnak} Let $f:\mathbb{N}\to \mathbb{D}$ be a non-pretentious multiplicative function. Then there exists a set $\mathcal{X}\subset \mathbb{N}$ with $\delta_{\log}^{+}(\mathcal{X})=1$ such that for any deterministic sequence $a:\mathbb{N}\to \mathbb{D}$ we have
\begin{align*}
\lim_{\substack{x\to \infty\\x\in \mathcal{X}}}\frac{1}{x}\sum_{n\leq x}f(n)a(n)=0.     
\end{align*}
\end{conjecture}

The following corrected asymptotic version of Elliott's conjecture was formulated by Tao in~\cite{tao-blog}.

\begin{conjecture}[An asymptotic Elliott conjecture]\label{conj_elliott_Tao} Let $k\geq 1$ and $h_1,\ldots, h_k\in \mathbb{N}$ be fixed and distinct.  Let $x\geq A\geq 2$, and let $f_1,\ldots, f_k:\mathbb{N}\to \mathbb{D}$ be multiplicative functions such that  
\begin{align*}
\inf_{|t|\leq Ax^{k-1}}\mathbb{D}(f_1,\chi(n)n^{it};x)\geq A
\end{align*}
for every Dirichlet character $\chi$ of modulus $\leq A$.
Then we have
\begin{align*}
\left|\frac{1}{x}\sum_{n\leq x}f_1(n+h_1)\cdots f_k(n+h_k)\right|=o_{A\to \infty}(1)+o_{x\to \infty}(1).  
\end{align*}
    
\end{conjecture}

We show that Conjecture~\ref{conj_elliott_Tao}, which has a more restrictive hypothesis, implies Conjecture~\ref{conj_elliott:modified}, which uses Elliott's original hypothesis (but involves almost all scales).

\begin{proposition}\label{prop_elliott_implication}
Conjecture~\ref{conj_elliott_Tao} implies Conjecture~\ref{conj_elliott:modified}, and Conjecture~\ref{conj_elliott:modified} implies Conjecture~\ref{conj_sarnak}.    
\end{proposition}

Thus, if Conjecture~\ref{conj_elliott_Tao} holds, then Elliott's original conjecture holds at almost all scales in the sense of Conjecture~\ref{conj_elliott:modified}, and so does the Sarnak-type Conjecture~\ref{conj_sarnak}.

\subsection{Organization of the paper.} The proofs of our main correlation results crucially rely on the interplay between different scales. In Section~\ref{cor_pret} we prove an important technical result, Theorem~\ref{thm_correlation_conclusion} (see also Proposition~\ref{prop_pretentious_correlation}), where the correlations of $f_1,\ldots, f_k$ are estimated in terms of the truncated pretentious distances $\max_{j\leq k}\mathbb{D}(f_j,\chi_j(n)n^{it_j};x^{\varepsilon},x)$. 

In Section~\ref{distance_pret}, we prove a key  Lemma~\ref{le_rigidity}, which says that if we can control pretentious distance at an infinite sequence of scales $(x_n)$, and $x_n$ does not grow too rapidly ($x_n+1\leq x_{n+1}\leq x_n^{O(1)}$), then we can control it at all scales. 

In Section~\ref{proof_cor1}, we use the results of the two previous sections  to prove Theorem~\ref{thm_main}.

Section~\ref{proof_cor2} is devoted to the proofs of the other correlation results. Here, to prove Theorem~\ref{thm_elliott_2point}, we prove a strengthening of Theorem~\ref{thm_A}(1) (Theorem~\ref{thm_elliott_2point_stronger}) and combine it with our Theorem~\ref{thm_main} to conclude. To prove Theorem~\ref{thm_elliott_higher}, we use Theorem~\ref{thm_A}(2) and Theorem~\ref{thm_main}, and to prove Theorem~\ref{thm_omega} we use Theorem~\ref{thm_correlation_conclusion}.

Section~\ref{proof_furst} is devoted to constructing Furstenberg systems (proof of Theorem~\ref{thm:furst}) using Theorem~\ref{thm_omega} and Proposition~\ref{prop:furst} as inputs.

In Section~\ref{proof_huds}, we use Theorem~\ref{thm_main} together with an upper bound  bound for $4$-point correlations (Lemma~\ref{le_upper_bound_correlation}), the ``rotation trick" for multiplicative functions from~\cite{tams-kmpt}, and various combinatorial arguments to resolve Hudson's conjecture (Theorem~\ref{thm_hudson}).

Finally, in Section~\ref{proof_edp} we prove the quantitative version of the Erd\H{o}s discrepancy theorem (Theorem~\ref{thm_edp}) using Theorem~\ref{thm_elliott_2point} in the non-pretentious case and Proposition~\ref{prop_pretentious_density} in the pretentious case.

\subsection{Acknowledgments}

The first author was supported by an FRG grant from the Heilbronn Institute for Mathematical Research. The third author was supported by a Titchmarsh Fellowship, Academy of Finland grant no. 340098, a von Neumann Fellowship (NSF grant \texttt{DMS-1926686}), and funding from European Union's Horizon
Europe research and innovation programme under Marie Sk\l{}odowska-Curie grant agreement No
101058904. 

We thank Nikos Frantzikinakis and Mariusz Lema\'nczyk for many helpful discussions.

\section{Notation and some preliminaries}\label{sec:densities}

\subsection{Arithmetic functions}

If $a_1,\ldots, a_k$ are integers, by $(a_1,\ldots, a_k)$ we denote their greatest common divisor and by $[a_1,\ldots, a_k]$ their least common multiple. If $a,b\in \mathbb{N}$, we write $a\mid b^{\infty}$ as a shorthand for the statement that $a\mid b^{\ell}$ for some $\ell\geq 1$. The notation $p^{\ell}\mid \mid n$ for integers $n,\ell$ and prime $p$ stands for $p^{\ell}\mid n$ and $p^{\ell+1}\nmid n$. We denote by $\varphi$ the Euler totient function. The notation $\left(\frac{a}{m}\right)$ stands for the Jacobi symbol of $a$ modulo $m$, when $m \in \mb{N}$ is odd. 
If $\mathcal{P}$ is a subset of the primes $\mathbb{P}$, $\Omega_{\mathcal{P}}(n)$ and $\omega_{\mathcal{P}}(n)$ are the number of prime factors of $n$ from $\mathcal{P}$ counted with and without multiplicities, respectively. Also let $\omega:=\omega_{\mathbb{P}}$. Finally, we use the standard abbreviation $e(t)=e^{2\pi i t}$ for $t \in \mathbb{R}$.  

\subsection{Averaging notation}
Let $f: A\to\mathbb{C}$ be a function defined on a non-empty set $A.$ We define averaging operators by
$$\mathbb{E}_{n\in A}f(n):=\frac{\sum_{n\in A}f(n)}{\sum_{n\in A}1},$$
and 
$$\mathbb{E}^{\log}_{n\in A}f(n):=\frac{\sum_{n\in A}\frac{f(n)}{n}}{\sum_{n\in A}\frac{1}{n}},$$
where in the latter case we assume that $A\subset \mathbb{N}$.

\subsection{Asymptotic notation}

As usual, we use $A=O(B)$ or $A\ll B$ or $B\gg A$ to denote that $|A|\leq C|B|$ for some absolute constant $C$. If we instead denote $A=O_{c}(B)$ or $A\ll_c B$, then the implied constant $C$ may depend on $c$. We also write $A=o_{x\to \infty}(B)$ if $|A|\leq C(x)|B|$ for some function $C(x)$ tending to $0$ as $x\to \infty$. If the underlying variable $x$ is clear from context, we just write $A=o(B)$.  

\subsection{The pretentious distance}

The pretentious distance, defined in~\eqref{eq:pret}, can be generalized to a truncated pretentious distance as follows: for $f,g : \mathbb{N} \ra \mathbb{D}$ and $x \geq y \geq 1$,
\begin{align*}
\mathbb{D}(f,g;y,x):=\left(\sum_{y\leq p\leq x}\frac{1-\textnormal{Re}(f(p)\overline{g(p)})}{p}\right)^{1/2}.   \end{align*}
We will frequently use the triangle inequalities for the pretentious distance:
\begin{align}\label{eq:triangle}\begin{split}
\mathbb{D}(f_1,f_3;y,x)&\leq \mathbb{D}(f_1,f_2;y,x)+\mathbb{D}(f_2,f_3;y,x),\\
\mathbb{D}(f_1f_2,g_1g_2;y,x) &\leq \mathbb{D}(f_1,g_1;y,x) + \mathbb{D}(f_2,g_2;y,x)
\end{split}
\end{align}
for $f_i,g_i : \mathbb{N} \rightarrow \mathbb{D}$ (see e.g.~\cite[Lemma 3.1]{GS} and~\cite{GS-zeta}).

\subsection{Densities}

Define the \emph{asymptotic density}, \emph{logarithmic density} and \emph{double-logarithmic density} of a set $A\subset \mathbb{N}$ as
\begin{align*}
d(A):&=\lim_{x\to \infty}\frac{1}{x}\sum_{n\in A\cap [1,x]}1,\\
\delta_{\log}(A):&=\lim_{x\to \infty}\frac{1}{\log x}\sum_{n\in A\cap [1,x]}\frac{1}{n},\\
\delta_{\log \log}(A):&=\lim_{x\to \infty}\frac{1}{\log \log x}\sum_{n\in A\cap [2,x]}\frac{1}{n\log n}.     
\end{align*}
respectively, whenever the limits exists.

Define the upper densities $d^{+},\delta_{\log}^{+},\delta_{\log \log}^{+}$ similarly, but with $\limsup$ in place of $\lim$. Define the lower densities $d^{-},\delta_{\log}^{-},\delta_{\log \log}^{-}$ also similarly, but with $\liminf$ in place of $\lim$.  

For an infinite set $A=\{a_1,a_2,a_3\ldots\}\subset \mathbb{N}$ with $a_1<a_2<\cdots$, and a sequence $f:\mathbb{N}\to \mathbb{C}$, define the \emph{limit of $f$ along $A$} as
\begin{align}\label{eq:lim}
 \lim_{\substack{a\to \infty\\a\in A}}f(a):=\lim_{j\to \infty}f(a_{j}),   
\end{align}
whenever the limit exists. We also define $\limsup_{\substack{a\to \infty\\a\in A}}f(a)$ and $\liminf_{\substack{a\to \infty\\a\in A}}f(a)$ similarly by replacing in~\eqref{eq:lim} $\lim$ with $\limsup$ and $\liminf$, respectively. We collect the following simple, standard facts to be used without further mention in the following sections.

\begin{lemma}[Basic properties of densities] Let $\nu\in \{d,\delta_{\log},\delta_{\log \log}\}$.
\begin{enumerate}
    \item For any set $A\subset \mathbb{N}$, we have
    \begin{align*}
    d^{-}(A)\leq \delta_{\log}^{-}(A)\leq \delta_{\log \log}^{-}(A)\leq \delta_{\log \log}^{+}(A)\leq \delta_{\log}^{+}(A)\leq d^{+}(A).     
    \end{align*}

    \item For any $A\subset \mathbb{N}$, we have $\nu^{-}(A)+\nu^{+}(\mathbb{N}\setminus A)= 1$.
    
    \item  Let $k\geq 1$, and let $A_1,\ldots, A_k\subset \mathbb{N}$ satisfy $\nu(A_i)=1$ for $i\leq k$. Then $\nu(A_1\cap \cdots \cap A_k)=1$. 
    
    \item Let $h\in \mathbb{N}$. Then $\nu^{-}(A+h)=\nu^{-}(A)$ and $\nu^{+}(A+h)=\nu^{+}(A)$.
\end{enumerate}
\end{lemma}

\begin{proof}
 All the pairwise inequalities in part (1) follow from partial summation, except the trivial inequality $\delta_{\log \log}^{-}(A)\leq \delta_{\log \log}^{+}(A)$.  
 
 Part (2) follows from averaging the identity
 $$1_{A}(n)=1-1_{\mathbb{N}\setminus{A}}(n)$$
with weights corresponding to $\nu$ and using the identity
$$\liminf_{x\to \infty}(1-a(x))=1-\limsup_{x\to \infty}a(x),$$
valid for any function $a(x)$.

Part (3) follows from averaging the inequality $$1_{A_1\cap\cdots \cap A_k}(n)\geq 1-(1_{\mathbb{N}\setminus A_1}(n)+\cdots +1_{\mathbb{N}\setminus A_k}(n))$$
 with weights corresponding to $\nu$ and noting that $\nu(A)=1$ implies $\nu(\mathbb{N}\setminus A)=0$ by part (2).
 
 Part (4) follows from the approximate translation invariance of asymptotic, logarithmic and doubly logarithmic averages. 
\end{proof}

\section{Correlations via truncated distance}\label{cor_pret}

In this section, we establish an upper bound for the correlations 
\begin{align}\label{eq:pretentious12}\frac{1}{x}\sum_{n\leq x}f_1(a_1n+h_1)\cdots f_k(a_kn+h_k)   
\end{align}
in terms of the truncated pretentious distances $D_j(x,\chi_j,t_j):=\mathbb{D}(f_j,\chi_j(n)n^{it_j};x^{\varepsilon},x)$. In particular, if $f_1,\ldots, f_k$ are non-pretentious and $\max_{1\leq j\leq k}D_j(x,\chi_j,t_j)=o(1)$ for some $\chi_j$, $t_j$, then our upper bound implies that~\eqref{eq:pretentious12} converges to $0$ as $x\to \infty$.   

\begin{theorem}[Correlations of multiplicative functions having small truncated pretentious distance]\label{thm_correlation_conclusion}
 Let $k\geq 1$ and $a_1,\ldots, a_k,h_1,\ldots, h_k\in \mathbb{N}$ be fixed with $a_ih_j\neq a_jh_i$ whenever $i\neq j$.  Also let Dirichlet characters $\chi_1,\ldots, \chi_k$ and real numbers $t_1,\ldots, t_k$ be fixed. Then for any  $x\geq 3$, $\varepsilon \in (1/(\log \log x),1/2)$ and any multiplicative functions $f_1,\ldots, f_k:\mathbb{N}\to \mathbb{D}$ satisfying
 \begin{align*}
\max_{1\leq j\leq k}\mathbb{D}(f_j,\chi_j(n)n^{it_j};x^{\varepsilon},x)\leq \varepsilon\quad \textnormal{ and } \quad \max_{1\leq j\leq k}\mathbb{D}(f_j,\chi_j(n)n^{it_j};x)\geq 1/\varepsilon
 \end{align*}
 we have
 \begin{align}\label{eq:correlK}
 \left|\frac{1}{x}\sum_{n\leq x}f_1(a_1n+h_1)\cdots f_k(a_kn+h_k)\right|\ll \left(\log \frac{1}{\varepsilon}\right)^{1/2}\varepsilon.   
 \end{align}   
\end{theorem}

\begin{remark}
In~\cite{klurman}, the first author gave a formula for computing the limit as $x\to \infty$ of~\eqref{eq:pretentious12} 
where $f_1,\ldots, f_k$ are pretentious functions (the formula was written explicitly for $k=2$, but it works for larger $k$ as well, and with the linear forms $a_jn+h_j$ replaced by polynomials). The error term in this formula depends on $\mathbb{D}(f_j,\chi_j(n)n^{it_j};\log x,x)$. For the proofs of our main theorems, however, it is crucial to work with the ``shorter'' pretentious distance $\mathbb{D}(f_j,\chi_j(n)n^{it_j};x^{\varepsilon},x)$ in Theorem~\ref{thm_correlation_conclusion}. We are able to work with this shorter pretentious distance by in particular using the fundamental lemma of the sieve and by carefully analyzing various error terms.  
\end{remark}

Theorem~\ref{thm_correlation_conclusion} follows immediately from the following proposition.

\begin{proposition}\label{prop_pretentious_correlation} Let $k\geq 1$ and $a_1,\ldots, a_k,h_1,\ldots, h_k\in \mathbb{N}$ be fixed with $a_ih_j\neq a_jh_i$ whenever $i\neq j$.   Also let Dirichlet characters $\chi_1,\ldots, \chi_k$ and real numbers $t_1,\ldots, t_k$ be fixed. Then for any $x\geq 3$,  $\varepsilon\in (1/(\log \log x),1/2)$ and any multiplicative functions $f_1,\ldots, f_k:\mathbb{N}\to \mathbb{D}$ we have
\begin{align}\label{eq:pretentious8}\begin{split}
&\left|\frac{1}{x}\sum_{n\leq x}f_1(a_1n+h_1)\cdots f_k(a_kn+h_k)\right|\\
&\ll \left(\log \frac{1}{\varepsilon}\right)^{1/2}\max_{j\leq k}\mathbb{D}(f_j,\chi_j(n)n^{it_j};x^{\varepsilon},x)+\exp\left(-\max_{j\leq k}\mathbb{D}(f_j,\chi_j(n)n^{it_j};x^{\varepsilon})^2\right)+\exp\left(-\frac{1}{8k^2\varepsilon}\right). 
\end{split}
\end{align}
\end{proposition}

Proposition~\ref{prop_pretentious_correlation} will be deduced from the following proposition, which is uniform in the linear forms but requires some further assumptions. 

\begin{proposition}\label{prop_pretentious_correlation2} Let integers $k\geq 1$ and $A\geq 2$ be fixed.  Let $x\geq 3$,  $\varepsilon\in (1/(\log \log x),1/2)$ and $a_1,\ldots, a_k$, $h_1,\ldots, h_k\in \mathbb{N}$ with $a_ih_j\neq a_jh_i$ whenever $i\neq j$. Suppose that  $(a_j,h_j)=1$, $a_j,h_j\leq (\log x)^{1/2}$, and $a_j\mid A^{\infty}$ for all $j$. Then for any multiplicative functions $f_1,\ldots, f_k:\mathbb{N}\to \mathbb{D}$ satisfying $f_j(p^{\ell})=0$ for all $p\mid \prod_{1\leq i<j\leq k}(a_ih_j-a_jh_i)$ and $\ell\geq 1$ we have
\begin{align*}
&\left|\frac{1}{x}\sum_{n\leq x}f_1(a_1n+h_1)\cdots f_k(a_kn+h_k)\right|\\
&\ll  \left(\log\frac{1}{\varepsilon}\right)^{1/2}\max_{j\leq k}\mathbb{D}(f_j,1;x^{\varepsilon},x)+\exp\left(-\max_{j\leq k}\mathbb{D}(f_j,1;x^{\varepsilon})^2\right)+\exp\left(-\frac{1}{2\varepsilon}\right).  
\end{align*}
    
\end{proposition}

\subsection{Proof of Proposition~\ref{prop_pretentious_correlation}}

We shall first show how Proposition~\ref{prop_pretentious_correlation2} implies Proposition~\ref{prop_pretentious_correlation}, and then in the next subsection we prove Proposition~\ref{prop_pretentious_correlation2}.  

\begin{proof}[Proof of Proposition~\ref{prop_pretentious_correlation} assuming Proposition~\ref{prop_pretentious_correlation2}] \ \\
\textbf{Step 1: reduction to the case $t_j=0$ for all $j\leq k$.} Suppose that Proposition~\ref{prop_pretentious_correlation} has been proved under the additional assumption that $t_j=0$ for all $j$, and with $\exp(-1/(4k^2\varepsilon))$ in place of $\exp(-1/(8k^2\varepsilon))$ in~\eqref{eq:pretentious8}, and consider the general case. We may assume that $\varepsilon>0$ is small enough in terms of all fixed quantities, since otherwise the claim is trivial. Note also that since $\varepsilon\geq 1/\log \log x$ we may absorb any error terms of size $O(1/(\log x)^{1/(4k^2)})$ to an error term of $O(\exp(-1/(4k^2\varepsilon)))$. 

Let $t=t_1+\cdots +t_k$, and let $f_j'(n)=f_j(n)n^{-it_j}$. By the mean value theorem, we have $(y+h)^{iu}=y^{iu}+O(h|u|/y)$ for any $h,u,y\geq 1$, so
\begin{align}\label{eq:pretentious1}
\frac{1}{x}\sum_{n\leq x}\prod_{j=1}^k f_j(a_jn+h_1)= \prod_{j=1}^k a_j^{it_j}\cdot \frac{1}{x}\sum_{n\leq x}n^{it}\prod_{j=1}^kf_j'(a_jn+h_j)+O\left(\frac{\log x}{x}\right). \end{align}
The error term above is certainly admissible. By partial summation, the  main term on the right-hand side of~\eqref{eq:pretentious1} is
\begin{align}\label{eq:pretentious2}
\ll (1+|t|)\max_{\sqrt{x} \leq y\leq x}\left|\frac{1}{y}\sum_{n\leq y}\prod_{j=1}^kf_j'(a_jn+h_j)\right| + |t|\frac{\log x}{\sqrt{x}}.    
\end{align}
Let $y\in [x^{1/2},x]$ be a point where the maximum is attained. 
Let $\varepsilon'\in [\varepsilon,2\varepsilon]$ be such that $y^{\varepsilon'}=x^{\varepsilon}$. Applying the case of the theorem where $t_j=0$ (with the slight improvement mentioned above), we obtain
\begin{align*}
&\left|\frac{1}{y}\sum_{n\leq y}\prod_{j=1}^kf_j'(a_jn+h_j)\right|\\
&\ll \left(\log \frac{1}{\varepsilon'}\right)^{1/2}\max_{j\leq k}\mathbb{D}(f_j',\chi_j(n);x^{\varepsilon},y)^2+\exp\left(-\max_{j\leq k}\mathbb{D}(f_j',\chi_j(n);x^{\varepsilon})^2\right)+\exp\left(-\frac{1}{4k^2\varepsilon'}\right).
\end{align*}
This gives the desired bound for the left-hand side of~\eqref{eq:pretentious1}, since $\varepsilon'\in [\varepsilon,2\varepsilon]$,  $y\leq x$, and $\mathbb{D}(f_j',\chi_j(n);x_1,x_2)=\mathbb{D}(f_j,\chi_j(n)n^{it_j};x_1,x_2)$ for any $x_1,x_2\geq 1$. Hence for the rest of the proof we may assume that $t_j=0$ for all $j\leq k$, and we seek to prove~\eqref{eq:pretentious8} with $\exp(-1/(4k^2\varepsilon))$ in place of $\exp(-1/(8k^2\varepsilon))$.

\textbf{Step 2: Reduction to Proposition~\ref{prop_pretentious_correlation2}.} In what follows, let the modulus of $\chi_j$ be $q_j$, and let $Q=q_1\cdots q_k$. Let 
$$\Delta=a_1\cdots a_k\prod_{1\leq i<j\leq k}(a_ih_j-a_jh_i)$$
and $A=Q\Delta$. Define new multiplicative functions $\widetilde{f_j}$ on the prime powers by
\begin{align*}
\widetilde{f_j}(p^{\ell})=\begin{cases}
f_j(p^{\ell}),\quad &p\nmid A,\\
0,\quad &p\mid A.
\end{cases}    
\end{align*}
Then writing
$$f_j(n)=\sum_{\substack{e\mid n\\ e\mid A^{\infty}}}f_j(e)\widetilde{f_j}(n/e)$$
we see that
\begin{align}\label{eq1}
&\frac{1}{x}\sum_{n\leq x}\prod_{j=1}^kf_j(a_jn+h_j)=\sum_{e_1,\ldots, e_k\mid A^{\infty}}\frac{f_1(e_1)\cdots f_k(e_k)}{[e_1,\ldots,e_k]}S(x;e_1,\ldots, e_k), 
\end{align}
where
\begin{align*}
S(x;e_1,\ldots, e_k)&:=\frac{[e_1,\ldots,e_k]}{x}\sum_{\substack{n\leq x\\e_j\mid a_jn+h_j\, \forall\, j\leq k}}\widetilde{f_1}\left(\frac{a_1n+h_1}{e_1}\right)\cdots \widetilde{f_k}\left(\frac{a_kn+h_k}{e_k}\right).
\end{align*}

Let us first estimate the contribution to~\eqref{eq1} from the case where one of the $e_j$ is large.  Note that we have $|S(x;e_1,\ldots, e_k)|\ll [e_1,\ldots, e_k]/e_j$ for all $j\leq k$, so $|S(x;e_1,\ldots, e_k)|\ll [e_1,\ldots, e_k]/(e_1\cdots e_k)^{1/k}$. Hence, the contribution to~\eqref{eq1} of those tuples with one of the $e_j$ (say $e_1$) being $\geq z$, for any $z\geq 1$, is crudely
\begin{align}\label{eq:pretentious4}\begin{split}
\ll \sum_{\substack{e_1,\ldots, e_k\mid A^{\infty}\\e_1\geq z}}\frac{1}{(e_1\cdots e_k)^{1/k}}&\leq \frac{1}{z^{9/(10k)}}\sum_{e_1,\ldots, e_k\mid A^{\infty}}\frac{1}{(e_1\cdots e_k)^{1/(10k)}}\\
&=\frac{1}{z^{9/(10k)}}\left(\prod_{p\mid A}\left(1+\frac{1}{p^{1/(10k)}}+\frac{1}{p^{2/(10k)}}+\frac{1}{p^{3/(10k)}}+\cdots \right)\right)^k\\
&\ll \frac{1}{z^{9/(10k)}}. 
\end{split}
\end{align}
Hence, we may truncate the sums over $e_j$ in~\eqref{eq1} to $e_j\leq (\log x)^{1/(3k)}$ with admissible error. 

Consider now the contribution to~\eqref{eq1} from the terms with $e_1,\ldots, e_k\leq (\log x)^{1/(3k)}$. If the system of $k$ congruences $a_jn+h_j\equiv 0\Mod{e_j}$, $1\leq j\leq k$ is solvable, it has a unique solution of the form $n\equiv b\pmod M$, where $1\leq b\leq M$ and  $$M=M(e_1,\ldots, e_k):=[e_1,\ldots, e_k]\leq (\log x)^{1/3}.$$
Hence we have
\begin{align}\label{eq:pretentious5}
S(x;e_1,\ldots, e_k)=\frac{1}{x/M}\sum_{m\leq x/M}\prod_{j=1}^k\widetilde{f_j}\left(\frac{a_j(Mm+b)+h_j}{e_j}\right)+O\left(\frac{1}{x/M}\right).    
\end{align}

Since $\widetilde{f}_j(p^{\ell})=0$ for $p\mid A$ and since $Q\mid A$, we can now write $\widetilde{f}_j=\chi_jg_j$ for some multiplicative functions $g_j:\mathbb{N}\to \mathbb{D}$ also satisfying $g_j(p^{\ell}) = 0$ whenever $p\mid A$. Splitting the sum in~\eqref{eq:pretentious5} into residue classes, we now obtain
\begin{align}\label{eq:pretentious7}\begin{split}
&S(x;e_1,\ldots, e_k)\\
&=\sum_{1\leq u\leq Q}\prod_{j=1}^k\chi_j\left(\frac{a_j(Mu+b)+h_j}{e_j}\right) \frac{1}{x/M}\sum_{\substack{m\leq x/M\\m\equiv u\Mod{Q}}}\prod_{j=1}^kg_j\left(\frac{a_j(Mm+b)+h_j}{e_j}\right)+O\left(\frac{1}{x/M}\right).
\end{split}
\end{align}

 We can write the inner expression in~\eqref{eq:pretentious7} as
\begin{align}\label{eq:pretentious9}
\frac{1}{x/M}\sum_{n\leq x/(MQ)}\prod_{j=1}^kg_j\left(\frac{a_j(M(Qn+u)+b)+h_j}{e_j}\right)+O\left(\frac{1}{x/M}\right).
\end{align}
The error term here is $O(x^{-1/2})$, since $M\leq (\log x)^{1/3}\ll x^{1/2}$, and this gives a negligible total contribution to~\eqref{eq1}
by using~\eqref{eq:pretentious4} with $z=1$.
Also by~\eqref{eq:pretentious4} with $z=1$, it now suffices to show that the main term in~\eqref{eq:pretentious9} is bounded by the right-hand side of~\eqref{eq:pretentious8} (with $\exp(-1/(4k^2\varepsilon))$ in place of $\exp(-1/(8k^2\varepsilon))$).

Write 
$$\frac{a_j(M(Qn+u)+b)+h_j}{e_j}=K_jn+B_j.$$
We may assume that
\begin{align}\label{eq:pretentious15}
 (K_j,B_j,A)=1\,\, \forall j\leq k,   
\end{align}
since otherwise  noting that $K_j\mid A^{\infty}$ the sum in~\eqref{eq:pretentious9} is $0$ by the fact that $g_j(p^{\ell})=0$ for $p\mid A,\ell\geq 1$. We now claim that
\begin{align}\label{eq:pretentious14}\begin{split}(K_j,B_j)=1\,\,\forall j\leq k\\
p\mid  \prod_{1\leq i_1<i_2\leq k}(K_{i_1}B_{i_2}-K_{i_2}B_{i_1})\implies p\mid A.\end{split}\end{align}
To verify the first claim in~\eqref{eq:pretentious14}, note that if $p\mid K_j=\frac{a_jMQ}{e_j}$, then $p\mid a_jMQ\mid A^{\infty}$, so $p\mid A$. Hence, $(K_j,B_j)=(K_j,B_j,A^{\ell})$ for some $\ell$, but this implies $(K_j,B_j)=1$ by~\eqref{eq:pretentious15}. To verify the second claim in~\eqref{eq:pretentious14}, note that 
\begin{align*}
K_{i_1}B_{i_2}-K_{i_2}B_{i_1}=MQ\frac{a_{i_1}h_{i_2}-a_{i_2}h_{i_1}}{e_{i_1}e_{i_2}},
\end{align*}
so if $p\mid K_{i_1}B_{i_2}-K_{i_2}B_{i_1}$, we must have $p\mid MQ(a_{i_1}h_{i_2}-a_{i_2}h_{i_1}) \mid A^\infty$, and so $p|A$. 

 Setting $X=x/(MQ)\in [x^{1/2},x]$, the main term in~\eqref{eq:pretentious9} can now be rewritten as
\begin{align*}
\frac{1}{QX}\sum_{n\leq X}\prod_{j=1}^k g_j(K_jn+B_j).
\end{align*}
Note that $K_j\mid A^{\infty}$ and  $K_j,B_j\leq \frac{1}{10}(\log X)^{1/2}$, say, for all large $X$ (since $M\ll (\log X)^{1/3}$). Note also that by~\eqref{eq:pretentious14} we have $(K_j,B_j)=1$ for all $j\leq k$ and $g_j(p^{\ell})=0$ for all $j\leq k$ whenever $p\mid \prod_{1\leq i_1<i_2\leq k}(K_{i_1}B_{i_2}-K_{i_2}B_{i_1})$, $\ell\geq 1$. Letting $\varepsilon'\in [\varepsilon,2\varepsilon]$ satisfy $X^{\varepsilon'}=x^{\varepsilon}$, we can apply Proposition~\ref{prop_pretentious_correlation2} to bound the main term in~\eqref{eq:pretentious9} by 
\begin{align*}
\ll \left(\log \frac{1}{\varepsilon'}\right)^{1/2}\max_{j\leq k}\mathbb{D}(f_j,\chi_j;x^{\varepsilon},X)^2+\exp\left(-\max_{j\leq k}\mathbb{D}(f_j,\chi_j;x^{\varepsilon})^2\right)+\exp\left(-\frac{1}{2\varepsilon'}\right).    
\end{align*}
Since $X\leq x$, and $\varepsilon'\in [\varepsilon,2\varepsilon]$, this completes the proof.
\end{proof}

\subsection{Proof of Proposition~\ref{prop_pretentious_correlation2}}

We begin with the following lemma.

\begin{lemma}\label{le_eulerproduct} Let integers $k\geq 1$ and $A\geq 1$ be fixed. Let $y\geq 3$, $P(y)=\prod_{p\leq y}p$, and let $f_1,\ldots, f_k:\mathbb{N}\to \mathbb{D}$ be multiplicative functions. Suppose that $(a_p)_p$ is a complex-valued sequence with $a_p=O(1)$. Then 
\begin{align}\label{eq:pretentious22}
\prod_{k<p\leq y}\left(1-\frac{k}{p}\right)\Bigg|\sum_{\substack{d_1,\ldots, d_k\mid P(y)\\(d_i,d_j)=1\,\forall\, i\neq j\\(d_j,A)=1\, \forall \, j\leq k}}\frac{f_1(d_1)\cdots f_k(d_k)}{d_1\cdots d_k}\prod_{j=1}^k\prod_{p\mid d_j}\left(1+\frac{a_p}{p}\right)\Bigg|\ll \exp\left(-\max_{j\leq k}\mathbb{D}(f_j,1;y)^2\right).
\end{align}\end{lemma}

\begin{proof} By the multiplicativity of the $f_j$, the left-hand side of~\eqref{eq:pretentious22} is 
\begin{align*}
&=\prod_{k<p\leq y}\left(1-\frac{k}{p}\right)\Bigg|\prod_{\substack{p_1,\ldots, p_k\leq y\\p_i\neq p_j\,\forall\, i\neq j\\(p_j,A)=1\, \forall\, j\leq k}}\left(1+\sum_{j=1}^k \frac{f_j(p)}{p}+O\left(\frac{1}{p^2}\right)\right)\Bigg|\\
&= \prod_{p\leq y}\left|1+\sum_{i=1}^k\left(\frac{f_j(p)-1}{p}+O\left(\frac{1}{p^2}\right)\right)\right|\\
&=  \prod_{p\leq y}\left|\exp\left(\sum_{j=1}^k\frac{f_j(p)-1}{p}+O\left(\frac{1}{p^2}\right)\right)\right|\\
&=\exp\left(\textnormal{Re}\left(\sum_{p\leq y}\left(\sum_{j=1}^k\frac{f_j(p)-1}{p}+O\left(\frac{1}{p^2}\right)\right)\right)\right),
\end{align*}
and exchanging the order of summation and estimating the sum over $j$ by the maximum of the terms, this becomes $\ll \exp(-\max_{j\leq k}\mathbb{D}(f_j,1;y)^2)$.   
\end{proof}

\begin{proof}[Proof of Proposition~\ref{prop_pretentious_correlation2}] Set
\begin{align*}
\eta:=\left(\log \frac{1}{\varepsilon}\right)^{1/2}\max_{j\leq k}\mathbb{D}(f_j,\chi_j;x^{\varepsilon},x)+\exp\left(-\max_{j\leq k}\mathbb{D}(f_j,\chi_j;x^{\varepsilon})^2\right)+\exp\left(-\frac{1}{2\varepsilon}\right).    
\end{align*}
We may assume that $\eta$ is small enough in terms of all fixed quantities, as otherwise the claim is trivial. 

Let $y=x^{\varepsilon}$. Define multiplicative functions $\widetilde{f}_j$ by
\begin{align*}
\widetilde{f}_j(p^{\ell})=\begin{cases}1,\quad \quad \quad p> y\\f_j(p^{\ell}),\quad p\leq y.\end{cases}    
\end{align*}
for $\ell\in \mathbb{N}$. Also 
let $P(y)=\prod_{p\leq y}p$. 

For complex numbers $z_1,\ldots, z_m,w_1,\ldots, w_m$ with $|z_i|,|w_i|\leq 1$, we have the inequality
\begin{align}\label{eq:zw}
|z_1\cdots z_m-w_1\cdots w_m|\leq \sum_{j=1}^m |z_j-w_j|,
\end{align}
which is easily proved by induction on $m$. Hence for any $j\leq k$ we have \begin{align}\label{eq:fjaj}
|f_j(a_jn+h_j) - \tilde{f}_j(a_jn+h_j)| \leq \sum_{\substack{p^{\ell}\mid \mid a_jn+h_j \\ p > y}} |1-f_j(p^{\ell})|
\end{align}
Applying~\eqref{eq:zw} and~\eqref{eq:fjaj}, we find
\begin{align}\label{eq:approximation}\begin{split}
&\sum_{n\leq x}f_1(a_1n+h_1)\cdots f_k(a_kn+h_k)=\sum_{n\leq x}\widetilde{f}_1(a_1n+h_1)\cdots \widetilde{f}_k(a_kn+h_k)\\
&+O\left(\sum_{n\leq x}\sum_{\substack{p^{\ell}\mid \mid a_jn+h_j \\ p > y}} |1-f_j(p^{\ell})|\right) 
\end{split}
\end{align}
for some $j\leq k$.

Consider the error term in~\eqref{eq:approximation}. Since $a_jn+h_j \leq 2x(\log x)^{1/2}$ and the number of $n\leq X$ that are divisible by $p^2$ for some $p> y$ is by the union bound $\ll \sum_{p> y}\frac{X}{p^2}\ll \frac{X}{y}$, and furthermore as any $n\leq z$ has $\ll \log z$ prime factors, the error term in~\eqref{eq:approximation} is 
\begin{align}\label{eq:pretentious10}
\sum_{n\leq x}\sum_{\substack{p \mid a_jn+h_j \\ p > y}} |1-f_j(p)|+O\left(\frac{x(\log x)^{3/2}}{y}\right).
\end{align}
Since $y=x^{\varepsilon}\geq x^{1/\log \log x}$, the error term here is admissible. 
Exchanging the order of summation, evaluating the $n$ sum, and using the prime number theorem, we see that the main term in~\eqref{eq:pretentious10} is
\begin{align}\label{eq:pretentious13}
\leq \sum_{y< p\leq a_jx+h_j}\left(x\frac{|1-f_j(p)|}{p}+1\right)\ll x\sum_{y< p\leq a_jx+h_j}\frac{|1-f_j(p)|}{p}+\frac{a_jx+h_j}{\log(a_jx+h_j)}.
\end{align}
By the assumption $a_j,h_j\leq (\log x)^{1/2}$,  the second term on the right of~\eqref{eq:pretentious13} is $\ll x/(\log x)^{1/2}$, which is dominated by $O(\exp(-1/(2\varepsilon)))$ for $\varepsilon\geq 1/(\log \log x)$. Moreover, the first term on the right of~\eqref{eq:pretentious13} can be truncated to $p\leq x$ at the cost of an acceptable error of $\ll x(\log \log x)/(\log x)$ by Mertens' theorem.

Now, by the Cauchy--Schwarz inequality and Mertens' theorem,~\eqref{eq:pretentious13} is
\begin{align*}
&\ll x\left(\sum_{y< p\leq x}\frac{|1-f_j(p)|^2}{p}\right)^{1/2}\left(\sum_{y< p\leq x}\frac{1}{p}\right)^{1/2}+\frac{x}{(\log x)^{1/2}}\\
&\ll x \left(\sum_{y< p\leq x}\frac{2-2\textnormal{Re}(f_j(p))}{p}\right)^{1/2}\left(\log\frac{1}{\varepsilon}\right)^{1/2}+\frac{x}{(\log x)^{1/2}}\\
&\ll x \left(\log\frac{1}{\varepsilon}\right)^{1/2}\max_{j\leq k}\mathbb{D}(f_j,1;x^{\varepsilon},x)+\frac{x}{(\log x)^{1/2}}\ll \eta x,
\end{align*}
recalling that $\varepsilon\geq 1/(\log \log x)$.

Now we are left with estimating the main term on the right of~\eqref{eq:approximation}. This can be expanded out as
\begin{align}\label{eq:S}
\sum_{d_1,\ldots, d_k\mid P(y)}f_1(d_1)\cdots f_k(d_k)\sum_{\substack{n\leq x\\d_j\mid a_jn+h_j\,\forall j\leq k\\((a_jn+h_j)/d_j,P(y))=1}}1=:\sum_{d_1,\ldots, d_k\mid P(y)}f_1(d_1)\cdots f_k(d_k)\Sigma(x;d_1,\ldots, d_k).    \end{align}
Since $f_j(p^{\ell})=0$ whenever $p\mid \prod_{1\leq i_1<i_2\leq k}(a_{i_1}h_{i_2}-a_{i_2}h_{i_1})$, we have $f(d_1)\cdots f(d_k)\Sigma(x;d_1,\ldots, d_k)=0$ unless $(d_i,d_j)=1$ for all $i\neq j$. 

We claim that we can truncate the sum over $d_j$ in~\eqref{eq:S} to $d_j\leq x^{1/(4k)}$ for all $j\leq k$. To see this, note that the contribution of the terms with e.g. $d_1>x^{1/(4k)}$ to $\Sigma(x;d_1,\ldots, d_k)$ is 
\begin{align}\label{eq:pretentious17}
 x\sum_{\substack{d_1,\ldots, d_k\mid P(y)\\d_1>x^{1/(4k)}}}\sum_{n\leq x}\prod_{j=1}^k 1_{d_j\mid a_jn+h_j}1_{(\frac{a_jn+h_j}{d_j},P(y))=1} &\ll  x\sum_{\substack{d_1\mid P(y)\\d_1>x^{1/(4k)}}}\sum_{n\leq x}1_{d_1\mid a_1n+h_1}1_{(\frac{a_1n+h_1}{d_1},P(y))=1}\nonumber\\
 &\ll x\sum_{\substack{e_1\leq (a_1x+h_1)/x^{1/(4k)}\\(e_1,P(y))=1}}\sum_{\substack{d_1\leq (a_1x+h_1)/e_1\\d_1e_1\equiv h_1\pmod{a_1}\\d_1\mid P(y)}}1, \end{align}
using the fact that 
$$\sum_{d\mid P(y)}1_{d\mid m}1_{(m/d,P(y))=1}=1$$ 
for every $m\geq 1$.
By an estimate for smooth numbers in arithmetic progressions~\cite{friedlander-smooth}, for any $z\in [y^{10},y^{(\log y)^{1/10}}]$ and any coprime $a,q$ with $1\leq a\leq q\leq y$ we have 
\begin{align*}
 \sum_{\substack{d\leq z\\d\equiv a\pmod q\\d\mid P(y)}}1\ll \rho\left(\frac{\log z}{\log y}-\frac{\log q}{\log y}-4\right)\frac{z}{q},   
\end{align*}
where $\rho(\cdot)$ is the Dickman function. By~\cite[(1.7)]{hildebrand-tenenbaum},  we have $\rho(u)\ll u^{-u/2}$ for $u\geq 1$. Note that, by Selberg's sieve, for any $Y\geq y$ the number of $e\in [1,Y]$ with $(e,P(y))=1$ is $\ll Y/(\log y)$, so~\eqref{eq:pretentious17} becomes
\begin{align*}
 &\ll x\left(\frac{1}{4k\varepsilon}-5\right)^{-1/(8k\varepsilon)+5/2}\sum_{\substack{e_1\leq (a_1x+h_1)/x^{1/(4k)}\\(e_1,P(y))=1}}\frac{x}{e_1} \ll x\left(\frac{1}{4k\varepsilon}-5\right)^{-1/(8k\varepsilon)+5/2}\frac{\log x}{\log y}\\
 &\ll x\exp\left(-\frac{1}{2\varepsilon}\right),
\end{align*}
say. 

We conclude that~\eqref{eq:S} is up to acceptable error 
\begin{align}\label{eq:pretentious12b}
 \sum_{\substack{d_1,\ldots, d_k\mid P(y)\\d_1,\ldots, d_k\leq x^{1/(4k)}\\(d_i,d_j)=1\,\forall i\neq j}} f_1(d_1)\cdots f_k(d_k)\Sigma(x;d_1,\ldots, d_k).   
\end{align}
Now, by the Chinese remainder theorem we can write
\begin{align*}
\Sigma(x;d_1,\ldots, d_k)=\sum_{\substack{n\leq x/(d_1\cdots d_k)\\(a_j\frac{d_1\cdots d_k}{d_j}n+\frac{a_jb+h_j}{d_j},P(y))=1\,\forall j\leq k}}1+O(x^{1/2}),    
\end{align*}
where $b=b(d_1,\ldots, d_k)\in [0,(d_1\cdots d_k))\cap \mathbb{Z}$ is the unique solution to 
\begin{align*}
a_jb+h_j\equiv 0\pmod{d_j}\quad \forall j\leq k.
\end{align*}
Since $d_j\leq x^{1/(4k)}$ for all $j\leq k$ in~\eqref{eq:pretentious12b}, by the fundamental lemma of sieve theory~\cite[Fundamental lemma 6.3]{iwaniec-kowalski} we have
\begin{align}\label{eq:pretentious18}
\Sigma(x;d_1,\ldots, d_k)=\left(1+O\left(\exp\left(-\frac{1}{2\varepsilon}\right)\right)\right)\frac{x}{d_1\cdots d_k}\prod_{p\leq y}\left(1-\rho(p;d_1,\ldots, d_k)\right)+O(x^{1/2}),    
\end{align}
where
\begin{align*}
\rho(p;d_1,\ldots, d_k):=\frac{|\{w\in \mathbb{F}_p:\,\, \prod_{j=1}^k \left(a_j\frac{d_1\cdots d_k}{d_j}w+\frac{a_jb+h_j}{d_j}\right)=0\}|}{p}.    
\end{align*}

Let $p$ be any prime coprime to $A$. We claim that
\begin{align}\label{eq:pretentious16}
 \rho(p;d_1,\ldots, d_k)=\begin{cases} \frac{k}{p},\,\,\,\, p\nmid d_1\cdots d_k\\
 \frac{1}{p},\,\,\,\, p\mid d_1\cdots d_k.
 \end{cases} 
\end{align}
Suppose first that $p\nmid d_1\cdots d_k$.  Then $p\mid a_i\frac{d_1\cdots d_k}{d_i}w+\frac{a_ib+h_i}{d_i}$ for $i\in \{i_1,i_2\}$ with $i_1\neq i_2$ implies $p\mid \frac{d_1\cdots d_k}{d_{i_1}d_{i_2}}(a_{i_1}h_{i_2}-a_{i_2}h_{i_1})$, so $p\mid A$, a contradiction. This leads to $k$ distinct solutions $w \in \mathbb{F}_p$. 

Suppose next that $p\mid d_1\cdots d_k$. Then as $(d_i,d_j)=1$ for $i\neq j$, it must be that $p$ divides exactly one of the $d_i$, say $p\mid d_1$. But then 
\begin{align*}
\prod_{j=1}^k \left(a_j\frac{d_1\cdots d_k}{d_j}w+\frac{a_jb+h_j}{d_j}\right)\equiv  \left(a_1d_2\cdots d_kw+\frac{a_1b+h_1}{d_1}\right) \prod_{j=2}^k \frac{a_jb+h_j}{d_j} \pmod p,    
\end{align*}
and this linear congruence has either $1$ or $p$ solutions. If it has $p$ solutions, then for some $2\leq j\leq k$ we have $p\mid a_jb+h_j$. But since $p\mid d_1$ we also have $p\mid a_1b+h_1$, so $p\mid a_1h_j-a_jh_1\mid A$. We conclude that~\eqref{eq:pretentious16} holds. 

Now by~\eqref{eq:pretentious16} and~\eqref{eq:pretentious18} we can write~\eqref{eq:pretentious12b} as
\begin{align}\label{eq:pretentious19}\begin{split}
& x\sum_{\substack{d_1,\ldots, d_k\mid P(y)\\d_1,\ldots, d_k\leq x^{1/(4k)}\\(d_i,d_j)=1\,\forall\, i\neq j\\(d_j,A)=1\, \forall \, j\leq k}}\frac{f_1(d_1)\cdots f_k(d_k)}{d_1\cdots d_k}\prod_{j=1}^k\prod_{p\mid d_j}\left(1-\frac{1}{p}\right)\cdot \prod_{\substack{p\leq y\\ p\nmid A\\ p\nmid d_1\cdots d_k}}\left(1-\frac{k}{p}\right)\\
&\quad +O\Bigg(x\exp\left(-\frac{1}{2\varepsilon}\right)\sum_{\substack{d_1,\ldots, d_k\mid P(y)\\d_1,\ldots, d_k\leq x^{1/(4k)}\\(d_i,d_j)=1\,\forall\, i\neq j\\(d_j,A)=1\, \forall \, j\leq k}}\frac{1}{d_1\cdots d_k}\prod_{j=1}^k\prod_{p\mid d_j}\left(1-\frac{1}{p}\right)\cdot\prod_{\substack{p\leq y\\p\nmid A\\p\nmid d_1\cdots d_k}}\left(1-\frac{k}{p}\right)+x^{3/4}\Bigg).
\end{split}
\end{align}
The first error term in~\eqref{eq:pretentious19} is by Lemma~\ref{le_eulerproduct} (with $f_j=1$)
\begin{align*}
&\ll x\exp\left(-\frac{1}{2\varepsilon}\right).
\end{align*}

We then estimate the main term in~\eqref{eq:pretentious19}. Completing the sums over $d_j$ and using Mertens's theorem, we see that the main term is
\begin{align}\label{eq:pretentious20}\begin{split}
&x\sum_{\substack{d_1,\ldots, d_k\mid P(y)\\(d_i,d_j)=1\,\forall\, i\neq j\\(d_j,A)=1\, \forall \, j\leq k}}\frac{f_1(d_1)\cdots f_k(d_k)}{d_1\cdots d_k}\prod_{j=1}^k\prod_{p\mid d_j}\left(1-\frac{1}{p}\right)\left(1-\frac{k}{p}\right)^{-1}\prod_{\substack{p\leq y\\p\nmid A}}\left(1-\frac{k}{p}\right)\\
&+O\Bigg(x\sum_{\substack{d_1,\ldots, d_k\mid P(y)\\ d_1>x^{1/(4k)}\\(d_i,d_j)=1\,\forall\, i\neq j\\(d_j,A)=1\, \forall \, j\leq k}}\frac{1}{d_1\cdots d_k}\prod_{j=1}^k\prod_{p\mid d_j}\left(1-\frac{1}{p}\right)\left(1-\frac{k}{p}\right)^{-1}(\log y)^{-k}\Bigg).
\end{split}
\end{align}
We then bound the error term in~\eqref{eq:pretentious20}. By Mertens's theorem and the fact that the density of $y$-smooth numbers in any interval $[D,2D]$ with $D\geq y$ is $\ll \exp(-\frac{1}{2}\frac{\log D}{\log y}\log \frac{\log D}{\log y})+\exp(-(\log D)^{1/20})$ by~\cite[Theorem 1.1 and (1.7)]{hildebrand-tenenbaum}, the error term is
\begin{align*}
 &\ll x(\log y)^{-k}\left(\sum_{\substack{d\mid P(y)\\d> x^{1/(4k)}}}\frac{1}{d}\right)\left(\sum_{d\mid P(y)}\frac{1}{d}\prod_{\substack{p\mid d\\p>k}}\left(1-\frac{k}{p}\right)^{-1}\right)^{k-1}\\
 &\ll x(\log y)^{-k}\left(\sum_{\substack{d\mid P(y)\\d> x^{1/(4k)}}}\frac{1}{d}\right)\left(\prod_{p\leq y}\left(1+\frac{1}{p}+O\left(\frac{1}{p^2}\right)\right)\right)^{k-1}\\
 &\ll x(\log y)^{-1}\sum_{\substack{D\geq x^{1/(4k)}/2\\ D=2^{\ell},\ell\in \mathbb{N}}}\left(\exp\left(-\frac{1}{2}\frac{\log D}{\log y}\log \frac{\log D}{\log y}\right)+\exp(-(\log D)^{1/20})\right)\\
 &\ll x\exp\left(-\frac{1}{2\varepsilon}\right),
\end{align*}
say.

We are left with the main term in~\eqref{eq:pretentious20}, which by Lemma~\ref{le_eulerproduct} is
\begin{align*}
\ll x\exp\left(-\max_{j\leq k}\mathbb{D}(f_j,1;x^{\varepsilon})\right),  
\end{align*}
completing the proof.
\end{proof}

\section{Lemmas on pretentious distance}\label{distance_pret}

 In this section, we present some lemmas on the pretentious distance, with the aim to prove an iterative Lemma~\ref{le_rigidity}. The first lemma is a standard bound on the pretentious distance between twisted Dirichlet characters and the constant function $1$.

\begin{lemma}\label{le_pretentious_distance}
Let $A\geq 1$ be fixed. Let $X\geq 10$ and let $\chi$ be a Dirichlet character of modulus $q\leq (\log X)^{A}$. Then  we have
\begin{align*}
\inf_{\delta(\chi)\leq |u|\leq X^A}\mathbb{D}(\chi(n)n^{iu},1;X)^2\geq \left(\frac{1}{4}-o(1)\right) \log \log X,   
\end{align*}
where $\delta(\chi)=0$ if $\chi$ is non-principal and $\delta(\chi)=1$ if $\chi$ is principal. Additionally for $0\leq |u|\leq 1$ and $\chi\pmod q$ principal we have
\begin{align*}
\mathbb{D}(\chi(n)n^{iu},1;X)^2=
\log(1+|u|\log X)+O(\log \log(3\omega(q))).
\end{align*}   
\end{lemma}

\begin{proof}
 Suppose first that $\chi$ is non-principal. We have
 \begin{align*}
  \inf_{|u|\leq X^A}\mathbb{D}(\chi(n)n^{iu},1;X)^2= \inf_{|u|\leq X^A}\mathbb{D}(\chi(n)n^{iu},1;X^{A})^2-O(1)\geq \left(\frac{1}{4}-o(1)\right)\log \log X   
 \end{align*}
 by~\cite[Lemma 7.9]{kmt-shortAP}. Suppose then that $\chi \pmod q$ is principal. Note that by Mertens's theorem
 \begin{align}\label{eq:omega1}
 \sum_{p\mid q}\frac{1}{p}\leq \sum_{p\leq \omega(q)}\frac{1}{p}\ll \log \log(3\omega(q))    
 \end{align}
  Then for $u\in [-X^{A}, X^{A}]$ we have
 \begin{align*}
 \mathbb{D}(\chi(n)n^{iu},1;X)^2&=\mathbb{D}(n^{iu},1;X)^2 -O\left(\sum_{p\mid q}\frac{1}{p}\right)\\
 &=\mathbb{D}(n^{iu},1;X)^2 -O(\log \log(3\omega(q))).  \end{align*}
 Now for $1\leq |u|\leq X^{A}$ the claim follows from~\cite[Lemma C.1]{MRT}, and if instead $|u|\leq 1$, the claim follows from the prime number theorem and partial summation. 
\end{proof}

For later use, we need a lemma on the pretentious distance of real-valued multiplicative functions.

\begin{lemma}\label{le:real} Let $f:\mathbb{N}\to [-1,1]$ be a multiplicative function such that $\mathbb{D}(f,\chi;\infty)=\infty$ for all real Dirichlet characters $\chi$. Then for any $A\geq 1$ we have
\begin{align*}
\inf_{|t|\leq x^{A}}\,\, \min_{\substack{\chi\pmod q\\q\leq A}}\mathbb{D}(f,\chi(n)n^{it};x)\xrightarrow[x\to \infty]{}\infty.    
\end{align*}
\end{lemma}
\begin{proof}
 This follows directly from~\cite[Lemma C.1]{MRT}.   
\end{proof}

Our next crucial iterative lemma says that a multiplicative function cannot pretend to be different characters at an infinite sequence of scales $(x_n)$, if $x_n$ does not grow too rapidly ($x_n+1\leq x_{n+1}\leq x_n^{O(1)}$).

\begin{lemma}\label{le_rigidity}
Let $A\geq 1$ be fixed, and let $f:\mathbb{N}\to \mathbb{D}$ be multiplicative. Let $F:\mathbb{N}\to \mathbb{R}_{\geq 1}$ be an increasing function with $\log \log \log \log X\leq F(X)\leq \frac{1}{10}(\log \log X)^{1/2}$ and $F(X^{A})^2\leq F(X)^2+1$ for all large enough $X$. Let $(x_n)$ be a sequence of real numbers in $[10,\infty)$ with $x_n+1\leq x_{n+1}\leq x_n^{A}$ for all $n\geq 1$.  Suppose that for each $j\in \mathbb{N}$ there exists a Dirichlet character $\chi_j$ of modulus $q_j\leq (\log x_j)^{A}$ and a real number $t_j\in [-x_j^{A},x_j^{A}]$ such that $$\mathbb{D}(f,\chi_j(n)n^{it_j};x_j)\leq F(x_j).$$
Then there exists a primitive Dirichlet character $\chi$ and a real number $t$ such that 
$$\mathbb{D}(f,\chi(n)n^{it};x)\leq 4F(x)$$ 
for all large enough $x$.

Additionally, if $\omega(q_j)$ is bounded, then the conclusion holds without the assumption $F(X)\geq \log \log \log \log X$.  
\end{lemma}

\begin{proof}
In what follows, let $\omega_{\max}(m):=\max_{1\leq j\leq m}\omega(q_{j})$.
By passing to a subsequence of $(x_n)$ if necessary (and adjusting $A$), we may assume that $x_{n+1}\geq x_n^{100}$ for all $n\geq 1$. Also by passing to a subsequence, we may assume that $x_1$ is large enough. 

By assumption, we have
\begin{align*}
\mathbb{D}(f,\chi_{m+1}(n)n^{it_{m+1}};x_{m+1})^2&\leq F(x_{m+1})^2\\
\mathbb{D}(f,\chi_m(n) n^{it_{m}};x_{m+1})^2&\leq F(x_m)^2+O\left(\sum_{x_m<p\leq x_{m+1}}\frac{1}{p}\right) \leq F(x_m)^2+O(1).
\end{align*}
Applying the triangle inequality for the pretentious distance and the fact that $F$ is increasing, this implies that
\begin{align*}
 \mathbb{D}(\chi_m(n) n^{it_m},\chi_{m+1}(n)n^{it_{m+1}};x_{m+1})^2\leq 4F(x_{m+1})^2+O(1). 
\end{align*}
Since $\chi_{m+1}\overline{\chi_{m}}(n)$ is unimodular for all primes not dividing $q_mq_{m+1}$, by~\eqref{eq:omega1} we deduce that
\begin{align}\label{eq5}\begin{split}
\mathbb{D}(\chi_{m+1}\overline{\chi_m}(n) n^{i(t_{m+1}-t_m)},1;x_m)^2&\leq 4F(x_{m+1})^2+O(1)+O\left(\sum_{p\mid q_mq_{m+1}}\frac{1}{p}\right)\\
&=4F(x_{m+1})^2+O(\log \log(3\omega_{\max}(m+1))). 
\end{split}
\end{align}
Applying Lemma~\ref{le_pretentious_distance} to~\eqref{eq5} (and recalling that $F(X)\leq \frac{1}{10}(\log \log X)^{1/2}$ for all large enough $X$), we conclude that for all large enough $m$ the character $\chi_{m+1}\overline{\chi_m}$ is principal, and additionally
\begin{align}\label{eq7}
 |t_{m+1}-t_m|\ll \frac{\exp(4F(x_{m+1})^2+O(\log \log(3\omega_{\max}(m+1))))}{\log x_{m+1}}.   
\end{align}

Note that by our assumption on $F$ and the fact that $\omega_{\max}(m+1)\ll \log \log x_{m+1}$ we have  
\begin{align*}
 \exp(4F(x_{m+1})^2+O(\log \log(3\omega_{\max}(m+1)))))\ll (\log x_m)^{1/10},   
\end{align*}
so
\begin{align}\label{eq7b}
 |t_{m+1}-t_m|\ll \frac{1}{(\log x_m)^{9/10}}.   
\end{align}
By telescoping~\eqref{eq7b}, and using the fact that $x_{n+1}\geq x_n^{100}$, we see that for any $j\geq 1$ we have 
\begin{align}\label{eq:7b}
 |t_{m+j}-t_m|\ll \sum_{i=1}^j\frac{1}{(\log x_{m+i})^{9/10}} \ll \frac{1}{(\log x_{m+1})^{9/10}}.   
\end{align}
Hence, the sequence $(t_n)$ is a Cauchy sequence, so it has some limit $t\in \mathbb{R}$. By~\eqref{eq7}, $t$  satisfies
\begin{align*}
 |t-t_m|&\ll \sum_{j\geq 1}\frac{\exp(4F(x_{m+j})^2+O(\log \log(3\omega_{\max}(m+j)))))}{\log x_{m+j}}\\
 &\ll \frac{\exp(4F(x_{m+1})^2+O(\log \log(3\omega_{\max}(m+1)))))}{\log x_{m+1}},   
\end{align*}
where the second estimate comes from the assumption  $F(X^{A})^2\leq F(X)^2+1$, which implies that the summands decay exponentially in $j$. 

From the triangle inequality for the pretentious distance and Lemma~\ref{le_pretentious_distance}, we now conclude that
\begin{align}\label{eq6}\begin{split}
\mathbb{D}(f,\chi_m(n)n^{it};x_m)&\leq \mathbb{D}(f,\chi_m(n)n^{it_m};x_m)+\mathbb{D}(\chi_m(n)n^{it_m},\chi_m(n)n^{it};x_m)\\
&\leq \mathbb{D}(f,\chi_m(n)n^{it_m};x_m)+O\left(\left(\sum_{p\mid q_m}\frac{1}{p}\right)^{1/2}\right)+\mathbb{D}(n^{i(t_m-t)},1;x_m)\\
&\leq  \mathbb{D}(f,\chi_m(n)n^{it_m};x_m)+O(\log \log(3\omega_{\max}(m))))^{1/2})\\
&\quad +\left(\log(1+|t_m-t|\log x_m)\right)^{1/2}\\
&\leq F(x_m)+ 2F(x_{m+1})+O((\log \log(3\omega_{\max}(m+1))))^{1/2}).
\end{split}
\end{align}

Since $\chi_{m+1}\overline{\chi_m}$ is principal for all $m\geq m_0$, by telescoping we see that
\begin{align*}
 \chi_{m}\overline{\chi_{m_0}}\quad  \textnormal{ is principal.}   
\end{align*}
Let $\chi_{m_0}^{*}$ be the primitive character inducing $\chi_{m_0}$. For $m\geq m_0$, write $\chi_m=\chi_{m_0}^{*}\chi_{0,m}$, where $\chi_{0,m}\pmod{q_m}$ is principal. Using~\eqref{eq6}, this implies that
\begin{align*}
 \mathbb{D}(f,\chi_{m_0}^{*}(n)n^{it};x_m)^2&\leq \mathbb{D}(f,\chi_{m}(n)n^{it};x_m)^2+O\left(\sum_{p\mid q_mq_{m_0}}\frac{1}{p}\right)\\
 &\leq   9F(x_{m+1})^2+O(\log \log(3\omega_{\max}(m+1))))  
\end{align*}

Since any large enough $x$ belongs to some interval $[x_m,x_{m+1}]$ and since $F(x_m)\gg \log \log(3\omega_{\max}(m+1))$, we conclude that
\begin{align*}
\mathbb{D}(f,\chi_{m_0}^{*}(n)n^{it};x)^2&\leq \mathbb{D}(f,\chi_{m_0}^{*}(n)n^{it};x_{m-1})^2+O\left(\sum_{x_{m-1}\leq p\leq x_{m+1}}\frac{1}{p}\right)\\
&= \mathbb{D}(f,\chi_{m_0}^{*}(n)n^{it};x_{m-1})^2+ O(1)\leq 10F(x_m)^2\leq 16F(x)^2  
\end{align*}
for all large enough $x$. This completes the proof.
\end{proof}

\section{Proof of Theorem~\ref{thm_main}}\label{proof_cor1}

As we shall see later, Theorem~\ref{thm_main} follows by combining Theorem~\ref{thm_correlation_conclusion} with the next result.

\begin{lemma}[Boosting a pretentiousness hypothesis]\label{le_pretentious_conclusion}
Let $A\geq 1$ be fixed. Let $f:\mathbb{N}\to \mathbb{D}$ be a multiplicative function such that for every $\varepsilon>0$ we have
\begin{align*}
\delta^{-}_{\log}\Big(\Big\{x\in \mathbb{N}:\,\, \inf_{|t|\leq x^{A}}\,\,\inf_{\substack{\chi\pmod q\\q\leq (\log x)^{A}}}\mathbb{D}(f,\chi(n)n^{it};x)^2\leq \varepsilon\log \log x\Big\}\Big)>0.    
\end{align*}
Then there exists a real number $t$, a primitive Dirichlet character $\chi$, and a function $\eta(X)$ tending to $0$ as $X\to \infty$ such that
\begin{align*}
 \delta_{\log\log }\left(\left\{x\in \mathbb{N}:\,\, \mathbb{D}(f,\chi(n)n^{it};x^{\eta(x)},x)\leq \eta(x)\right\}\right)=1.   
\end{align*}
\end{lemma}

In particular, the result applies to any function $f$ that is moderately non-pretentious. 

\begin{proof} 
If $f$ is pretentious, the claim is clear, since then there exist a Dirichlet character $\chi$ and a real number $t$ such that $\mathbb{D}(f,\chi(n)n^{it};y,\infty)=o_{y\to \infty}(1)$. Suppose from now on that $f$ is non-pretentious. 

Let $\varepsilon>0$. In what follows, any implied constants will be independent of $\varepsilon$. By assumption, the set
\begin{align*}
 \mathcal{S}_{\varepsilon}:= \Big\{x\in \mathbb{N}:\,\, \inf_{|t|\leq x^{A}}\,\,\inf_{\substack{\chi\pmod q\\q\leq (\log x)^{A}}}\mathbb{D}(f,\chi(n)n^{it};x)^2\leq \varepsilon\log \log x\Big\}  
\end{align*}
satisfies
\begin{align*}
 \delta^{-}_{\log}(\mathcal{S}_{\varepsilon})>0.   
\end{align*}
In particular, this implies that there is some $a\geq 1$ (depending on $\varepsilon$) such that $\mathcal{S}_{\varepsilon}$ contains a sequence of the form $(x_n)_{n\geq 1}$ with $x_n+1\leq x_{n+1}\leq x_n^{a}$ for all $n\geq 1$. Hence we have
\begin{align}\label{eq:dist}
\inf_{|t|\leq x_n^{A}}\,\,\inf_{\substack{\chi\pmod q\\q\leq (\log x_n)^{A}}}\mathbb{D}(f,\chi(n)n^{it};x_n)^2\leq \varepsilon\log \log x_n.    
\end{align}

Combining~\eqref{eq:dist} and Lemma~\ref{le_rigidity}, we see that for some real number $t_{\varepsilon}$ and some primitive Dirichlet character $\chi_{\varepsilon}$ we have
\begin{align}\label{eq15}
\mathbb{D}(f,\chi_{\varepsilon}(n)n^{it_{\varepsilon}};x)^2\ll \varepsilon \log \log x.    \end{align}
for all $x\geq 10$.
We claim that $t_{\varepsilon}$ and $\chi_{\varepsilon}$ are in fact independent of $\varepsilon$. Indeed, if  $\varepsilon_0>0$ is a small constant and $\varepsilon \in (0,\varepsilon_0)$ then
\begin{align*}
\mathbb{D}(f,\chi_{\varepsilon_0}(n)n^{it_{\varepsilon_0}};x)^2&\ll \varepsilon_0 \log \log x,\\
\mathbb{D}(f,\chi_{\varepsilon}(n)n^{it_{\varepsilon}};x)^2&\ll \varepsilon_0 \log \log x
\end{align*}
for $j\in \{1,2\}$ and all $x\geq 10$, and by the triangle inequality for the pretentious distance this gives
\begin{align*}
\mathbb{D}(\chi_{\varepsilon_0}(n)n^{it_{\varepsilon_0}},\chi_{\varepsilon}(n)n^{it_{\varepsilon}};x)^2\ll \varepsilon_0 \log \log x  
\end{align*}
for all $x\geq 10$. If $\varepsilon_0>0$ is small enough, by Lemma~\ref{le_pretentious_distance} this implies that $t_{\varepsilon}=t_{\varepsilon_0}$ and $\chi_{\varepsilon}=\chi_{\varepsilon_0}$. Therefore, there exist some real numbers $t$ and $\chi$ such that for any $\varepsilon>0$ we have 
\begin{align*}
\mathbb{D}(f,\chi(n)n^{it};x)^2\ll \varepsilon \log \log x.        
\end{align*}
for all $x\geq 10$. This means that there exists a decreasing function $\eta_1(X)$ tending to $0$ as $X\to \infty$ such that
\begin{align}\label{eq:15b}
D(x) := \mathbb{D}(f,\chi(n)n^{it};x)^2\leq \eta_1(x)\log \log x
\end{align}
for all $x\geq 10$.

Let 
\begin{align*}
\mathcal{D}=\{x\in \mathbb{N}:\,\, D(x)-D(x^{\eta_1(x)})\geq \eta_1(x)^{1/2}\}.    
\end{align*}
Now, for any $X\geq 10$ we have
\begin{align}\label{eq9}\begin{split}
 &\frac{1}{\log \log X}\sum_{3\leq n\leq X}\frac{1_{n\in \mathcal{D}}}{n\log n}\\
 \leq&  \frac{1}{\log \log X}\sum_{0\leq j\leq \log \log X}  \sum_{e^{e^j}\leq n\leq e^{e^{j+1}}}\frac{1}{n\log n}1_{D(e^{e^{j+1}})-D(e^{e^{j-\log(1/\eta_1(X))}})\geq \eta_1\left(X\right)^{1/2}},
 \end{split}
 \end{align}
 using the fact that $\eta_1(\cdot)$ is decreasing.
By telescopic summation and~\eqref{eq:15b}, we see that the left-hand side of~\eqref{eq9} is
\begin{align}\label{eq:15c}
&\ll \frac{1}{\eta_1(X)^{1/2}\log \log X}\sum_{0\leq j\leq \log \log X}(D(e^{e^{j+1}})-D(e^{e^{j-\log(1/\eta_1(X))}}))\nonumber\\
&\ll \frac{1}{\eta_1(X)^{1/2}\log \log X}\sum_{0\leq j\leq \log \log X}\,\sum_{0\leq k\leq 1+\log(1/\eta_1(X))}(D(e^{e^{j+1-k}})-D(e^{e^{j-k}}))\nonumber\\
&\ll \frac{\log(1/\eta_1(X))}{\eta_1(X)^{1/2}\log \log X}\max_{0\leq k\leq 1+\log(1/\eta_1(X))}\,\sum_{0\leq j\leq \log \log X}(D(e^{e^{j+1-k}})-D(e^{e^{j-k}}))\nonumber\\
&\ll \frac{\log(1/\eta_1(X))}{\eta_1(X)^{1/2}\log \log X}D(X^{e})\nonumber\\
&\ll \frac{\log(1/\eta_1(X))}{\eta_1(X)^{1/2}\log \log X}\eta_1(X^{e})(\log \log X)\ll \eta_1(X)^{1/2}\log(1/\eta_1(X)).
\end{align}

 Putting together~\eqref{eq9} and~\eqref{eq:15c} and sending $X\to \infty$, we obtain
\begin{align*}
 \delta_{\log \log}^{+}(\mathcal{D})=0,    
\end{align*}
so
\begin{align*}
 \delta_{\log \log}^{-}(\mathbb{N}\setminus \mathcal{D})=1,    
\end{align*}
Now, the function $\eta(X):=\eta_1(X)^{1/4}$ satisfies
\begin{align*}
 \delta_{\log\log }\left(\left\{x\in \mathbb{N}:\,\, \mathbb{D}(f,\chi(n)n^{it};x^{\eta(x)},x)\leq \eta(x)\right\}\right)=1,  
 \end{align*}
as desired. 
\end{proof}

We are now ready to prove Theorem~\ref{thm_main}.

\begin{proof}[Proof of Theorem~\ref{thm_main}] Let $k\geq 1$ and $h_1,\ldots, h_k\in \mathbb{N}$. Let $f_1,\ldots, f_k:\mathbb{N}\to \mathbb{D}$ be multiplicative functions with $f_1$ moderately non-pretentious and $f_2,\ldots, f_k$ either moderately non-pretentious or pretentious. Then, by Lemma~\ref{le_pretentious_conclusion} there exist real numbers $t_1,\ldots, t_k$, Dirichlet characters $\chi_1,\ldots, \chi_k$, and a function $\eta(X)$ tending to $0$ as $X\to \infty$ such that the sets
\begin{align*}
\mathcal{S}_j=\{x\in \mathbb{N}:\,\, \mathbb{D}(f_j,\chi_j(n)n^{it_j};x^{\eta(x)},x)\leq \eta(x)\}
\end{align*}
for $1\leq j\leq k$ satisfy
\begin{align*}
 \delta_{\log \log}(\mathcal{S}_j)=1.   
\end{align*}
Hence, denoting $\mathcal{S}=\bigcap_{1\leq j\leq k}\mathcal{S}_j$, we also have
\begin{align}\label{eq16}
 \delta_{\log \log}(\mathcal{S})=1.   
\end{align}
Moreover, the set $\mathcal{S}$ only depends on the set $\{f_1,\ldots, f_k\}$. 

Now, for any $\varepsilon>0$, we have for all large enough $x\in \mathcal{S}$ the bound
\begin{align}\label{eq22}
 \max_{1\leq j\leq k}\mathbb{D}(f_j,\chi_j(n)n^{it_j};x^{\varepsilon},x)\leq \varepsilon.  
\end{align}
Furthermore, by the assumption that $f_1$ is moderately non-pretentious, we have
\begin{align}\label{eq23}
\mathbb{D}(f_1,\chi_1(n)n^{it_1};x^{\varepsilon})\xrightarrow[x\to \infty]{}\infty.  
\end{align}
Applying Theorem~\ref{thm_correlation_conclusion}, taking~\eqref{eq22} and~\eqref{eq23} into account, it follows that
\begin{align*}
 \limsup_{\substack{x\to \infty\\x\in \mathcal{S}}}\left|\frac{1}{x}\sum_{n\leq x}f_1(n+h_1)\cdots f_k(n+h_k)\right|\ll \left(\log \frac{1}{\varepsilon}\right)^{1/2}\varepsilon.   
\end{align*}
Sending $\varepsilon\to 0$, we conclude that
\begin{align*}
 \lim_{\substack{x\to \infty\\ x\in \mathcal{S}}}\left|\frac{1}{x}\sum_{n\leq x}f_1(n+h_1)\cdots f_k(n+h_k)\right|=0.   
 \end{align*}
 Combining this with~\eqref{eq16}, Theorem~\ref{thm_main} follows. 
\end{proof}

\section{Proofs of the other correlation results}\label{proof_cor2}

\subsection{Two-point correlations}

The proof of Theorem~\ref{thm_elliott_2point} is based on the following strengthening of Theorem~\ref{thm_A}(1) that tells us more about the set of scales where the correlations of two multiplicative functions can be large. We prove this with a somewhat more direct argument than the proof of Theorem~\ref{thm_A}(1) in~\cite{tt-ant} (related to the one in ~\cite[Section 8]{helfgott-radziwill}), using ideas of Tao from~\cite{tao-chowla} (as in~\cite{tt-ant}) but avoiding some other tools (approximate homomorphisms and the Hardy--Littlewood maximal inequality) that were used in~\cite{tt-ant}.

\begin{theorem}\label{thm_elliott_2point_stronger}
Let $f_1,f_2:\mathbb{N}\to \mathbb{D}$ be multiplicative functions. Let $A(X)$ be any function tending to infinity as $X\to \infty$. Let
\begin{align*}
\mathcal{X}=\Big\{x\in \mathbb{N}:\,\, \max_{j\in \{1,2\}}\inf_{|t|\leq A(x)x}\,\,\min_{\substack{\chi\pmod q\\q\leq A(x)}}\mathbb{D}(f_1,\chi(n)n^{it};x)\geq A(x)\Big\},   
\end{align*}
and suppose that $\mathcal{X}$ is infinite.
Then, for some function $\psi(X)$ tending to infinity as $X\to \infty$ we have 
\begin{align}\label{eq:lim2}
 \lim_{\substack{x\to \infty\\x\in \mathcal{X}}}\frac{1}{\log x}\sum_{N\leq x}\max_{\substack{1\leq h_1,h_2\leq \psi(N)\\h_1\neq h_2}}\left|\frac{1}{N}\sum_{n\leq N}f_1(n+h_1)f_2(n+h_2)\right|=0.   
\end{align}
\end{theorem}

We note that Theorem~\ref{thm_A}(1) essentially corresponds to the case $\mathcal{X}=\mathbb{N}$ of Theorem~\ref{thm_elliott_2point_stronger}.
 
For the proof of Theorem~\ref{thm_elliott_2point_stronger}, we first need what might be called a ``decoupling inequality'' for correlations, developed by Tao  in~\cite{tao-chowla} (and generalized in~\cite[Theorem 3.6]{tt-duke},~\cite[Proposition C.1]{fh-annals}). This allows one to compare a one-variable correlation average with a two-variable correlation average (we shall only need the two-point case of the result).

\begin{lemma}[Decoupling inequality for correlations]\label{le_entropy}
Let $k\geq 1$ and $h_1,\ldots, h_k\in \mathbb{N}$ be fixed.
Then, for any function $P(x)$ tending to infinity slowly enough and any $g_1,\ldots, g_k:\mathbb{N}\to \mathbb{D}$, we have
\begin{align*}
 \mathbb{E}_{p\leq P(x)}^{\log}|\mathbb{E}_{n\leq x}g_1(n+ph_1)\cdots g_k(n+ph_k)-\mathbb{E}_{n\leq x/p}g_1(p(n+h_1))\cdots g_k(p(n+h_k))|=o(1).   
\end{align*}
\end{lemma}

\begin{proof}
This follows from~\cite[Proposition 2.3]{tt-ant} (where the multiplicativity assumption is not used). See also~\cite[Corollary 1.1]{helfgott-radziwill} for a quantitative version of the case $k=2$.
\end{proof}

We also need the following lemma that can be extracted from~\cite{tao-chowla} (and is based on the work in~\cite{MRT}). It allows one to deal with the two-dimensional averages that result from Lemma~\ref{le_entropy}.

\begin{lemma}\label{le_Tao} Let $h_1,h_2\in \mathbb{N}$ be distinct fixed integers. Let $1\leq P(X)\leq X$ be a function tending to infinity. Let $f_1,f_2:\mathbb{N}\to \mathbb{D}$ be multiplicative.  Let $x\geq 3$, and let $$M_A=\min_{j\in \{1,2\}}\,\inf_{|t|\leq Ax}\,\min_{\substack{\chi\pmod     q\\q\leq A}}\mathbb{D}(f_j,\chi(n)n^{it};x).$$
Then we have
\begin{align*}
\mathbb{E}_{p\leq P(x)}^{\log}|\mathbb{E}_{n\leq x}f_1(n+ph_1)f_2(n+ph_2)|=o_{M_A\to \infty}(1)+o_{A\to \infty}(1)+o_{x\to \infty}(1). \end{align*}
\end{lemma}

\begin{proof}
This follows from the arguments in~\cite[Section 3]{tao-chowla}. Alternatively, combine~\cite[Proposition 8.3]{helfgott-radziwill} (with $\mathbf{Q}$ there a suitably chosen subset of the primes in $[H/2,H]$ for any $H\leq P(x)$) with~\cite[Theorem 1.7]{MRT}. 
\end{proof}

We are now ready to prove Theorem~\ref{thm_elliott_2point_stronger}.

\begin{proof}[Proof of Theorem~\ref{thm_elliott_2point_stronger}] Let us first note that we can assume for every $c>0$ and for $j\in \{1,2\}$ that
\begin{align}\label{eq21a}
\sum_{\substack{p\in \mathbb{P}\\|f_j(p)|\leq 1-c}}\frac{1}{p}<\infty,
\end{align}
since otherwise the claim follows (with the limit in~\eqref{eq:lim2} being over all $x\in \mathbb{N}$)  by applying the triangle inequality and Delange's theorem~\cite[Thm. III.4.2]{Tenenbaum} to bound
\begin{align*}
 \left|\frac{1}{x}\sum_{n\leq x}f_1(n+h_1)f_2(n+h_2)\right|\leq \frac{1}{x}\sum_{n\leq x}|f_j(n+h_j)|\ll \exp\left(-\sum_{p\leq x}\frac{1-|f_j(p)|}{p}\right)=o(1).   
\end{align*}

Let $A(\cdot)$ be as in the theorem, and let
\begin{align*}
 \mathcal{X}_1=\{x\in \mathbb{N}:\,\, \max_{j\in \{1,2\}}\inf_{|t|\leq A(x)x}\,\min_{\substack{\chi\pmod q\\q\leq A(x)}}\mathbb{D}(f_j,\chi(n)n^{it};x)\geq A(x)/2\}.   
\end{align*}

By the triangle inequality for the pretentious distance, we see that for all large enough $x$ we have
\begin{align}\label{eq30}
 x\in \mathcal{X}\implies [x^{1/A(x)},x]\cap \mathbb{N}\subset \mathcal{X}_1.   
\end{align}

Let 
\begin{align*}
S(x;h_1,h_2):=\mathbb{E}_{n\leq x}f_1(n+h_1)f_2(n+h_2).   
\end{align*} 
By Lemma~\ref{le_entropy} and the fact that $f_j(pn)=f_j(p)f_j(n)+O(1_{p\mid n})$, for any slowly enough growing function $P(\cdot)$ and any fixed, distinct $h_1,h_2\in \mathbb{N}$ we have
\begin{align}\label{eq:lim5}
\lim_{x\to \infty}\mathbb{E}_{p\leq P(x)}^{\log}|S(x;ph_1,ph_2)-f_1(p)f_2(p)S(x/p;h_1,h_2)|=0.      \end{align}
We may choose $P(\cdot)$ so that $P(\cdot)$ is increasing and $P(\log x)\geq P(x)-1$ for all $x$. 

We make the following simple observation which we will use several times in this section without further mention. If $F:\mathbb{N}^k\to \mathbb{R}_{\geq 0}$ satisfies $\lim_{x\to \infty}F(x,n_1,\ldots, n_{k-1})=0$ for any $n_i\in \mathbb{N}$, then there exists some function $N(X)$ tending to infinity as $X\to \infty$ such that
\begin{align}\label{eq:simple}
\max_{1\leq n_1,\ldots, n_{k-1}\leq N(x)}F(x,n_1,\ldots, n_{k-1})\leq \frac{1}{N(x)}
\end{align}
for all $x\in \mathbb{N}$.

By this observation and~\eqref{eq:lim5}, there exists a function $H(\cdot)$ tending to infinity such that 
\begin{align}\label{eq:entropy}
\lim_{x\to \infty}\max_{\substack{1\leq h_1,h_2\leq H(x)\\h_1\neq h_2}}\mathbb{E}_{p\leq P(x)}^{\log}|S(x;ph_1,ph_2)-f_1(p)f_2(p)S(x/p;h_1,h_2)|=0.  
\end{align}

By Lemma~\ref{le_Tao}, if $H(\cdot)$ is growing sufficiently slowly we have 
\begin{align*}
\lim_{\substack{x\to \infty\\x\in \mathcal{X}_1}}\max_{\substack{1\leq h_1, h_2\leq H(x)\\h_1\neq h_2}}\mathbb{E}_{p\leq P(x)}^{\log}|S(x;ph_1,ph_2)|=0,    
\end{align*}
so by~\eqref{eq:entropy} we have
\begin{align*}
\lim_{\substack{x\to \infty\\x\in \mathcal{X}_1}}\max_{\substack{1\leq h_1, h_2\leq H(x)\\h_1\neq h_2}}\mathbb{E}_{p\leq P(x)}^{\log}|f_1(p)f_2(p)S(x/p;h_1,h_2)|=0. \end{align*}
Recalling~\eqref{eq21a}, we obtain
\begin{align*}
\lim_{\substack{x\to \infty\\x\in \mathcal{X}_1}}\max_{\substack{1\leq h_1, h_2\leq H(x)\\h_1\neq h_2}}\mathbb{E}_{p\leq P(x)}^{\log}|S(x/p;h_1,h_2)|=0.     
\end{align*}
Therefore, for some increasing function
 $h(X)\leq H(X)$ tending to infinity slowly we have 
 \begin{align}\label{eq:lim3}
\max_{\substack{1\leq h_1, h_2\leq H(x)\\h_1\neq h_2}}\mathbb{E}_{p\leq P(x)}^{\log}|S(x/p;h_1,h_2)|\leq \frac{1}{h(x)^3}     
\end{align}
for all $x\in \mathcal{X}_1$. We have
\begin{align*}
\max_{\substack{1\leq h_1, h_2\leq h(x)\\h_1\neq h_2}}\mathbb{E}_{p\leq P(x)}^{\log}|S(x/p;h_1,h_2)|&\geq \frac{1}{h(x)^2}\sum_{\substack{1\leq h_1,h_2\leq h(x)\\h_1\neq h_2}}\mathbb{E}_{p\leq P(x)}^{\log}|S(x/p;h_1,h_2)|\\
&\geq \frac{1}{h(x)^2}\mathbb{E}_{p\leq P(x)}^{\log}\max_{\substack{1\leq h_1, h_2\leq h(x)\\h_1\neq h_2}}|S(x/p;h_1,h_2)|,
\end{align*}
and combined with~\eqref{eq:lim3} this yields
\begin{align*}
\lim_{\substack{x\to \infty\\x\in \mathcal{X}_1}}  \mathbb{E}_{p\leq P(x)}^{\log}\max_{\substack{1\leq h_1, h_2\leq h(x)\\h_1\neq h_2}}|S(x/p;h_1,h_2)|=0.    
\end{align*}

Recalling~\eqref{eq30} and using the fact that $P(x^{1/A(x)})\geq P(\log x)\geq P(x)-1$ for all large $x$ by the choice of $P(\cdot)$,  we have 
\begin{align*}
 \lim_{\substack{x\to \infty\\x\in \mathcal{X}}}\mathbb{E}_{p\leq P(x)}^{\log} \frac{1}{\log x}\int_{x^{1/A(x)}}^{x}\max_{\substack{1\leq h_1, h_2\leq h(x^{1/A(x)})\\h_1\neq h_2}}|S(y/p;h_1,h_2)|\, \frac{dy}{y}=0.  
\end{align*}
By the pigeonhole principle, we can find some $P_x\in [2,P(x)]$ such that
\begin{align*}
 \lim_{\substack{x\to \infty\\x\in \mathcal{X}}}\frac{1}{\log x}\int_{x^{1/A(x)}}^{x}\max_{\substack{1\leq h_1, h_2\leq h(x^{1/A(x)})\\h_1\neq h_2}}|S(y/P_x;h_1,h_2)|\, \frac{dy}{y}=0,  
\end{align*}
and making the change of variables $w=y/P_x$ in the integral, this implies that 
\begin{align*}
 \lim_{\substack{x\to \infty\\x\in \mathcal{X}}}\frac{1}{\log x}\int_{x^{1/A(x)}}^{x}\max_{\substack{1\leq h_1, h_2\leq h(x^{1/A(x)})\\h_1\neq h_2}}|S(w;h_1,h_2)|\, \frac{dw}{w}=0,  
\end{align*}
since $P_x\leq P(x)$ and $P(\cdot)$ is growing sufficiently slowly. Widening the integration range to $[1,x]$ (noting that the contribution of $[1,x^{1/A(x)}]$ is negligible) and approximating the integral with a sum, we see that
\begin{align}\label{eq:lim4}
\lim_{\substack{x\to \infty\\x\in \mathcal{X}}}\frac{1}{\log x}\sum_{n\leq x}\frac{1}{n}\max_{\substack{1\leq h_1, h_2\leq h(x^{1/A(x)})\\h_1\neq h_2}}|S(n;h_1,h_2)|=0.    
\end{align}
Taking $\psi(X)=h(X^{1/A(X)})$ completes the proof.
\end{proof}

Theorem~\ref{thm_elliott_2point}
 follows quickly by combining Theorem~\ref{thm_elliott_2point_stronger} and Theorem~\ref{thm_main}.
 
\begin{proof}[Proof of Theorem~\ref{thm_elliott_2point}] Let 
\begin{align*}
\mathcal{Y}=\{x\in \mathbb{N}:\,\, \max_{j\in \{1,2\}}\inf_{|t|\leq x^2}\,\,\min_{\substack{\chi\pmod q\\q\leq \log x}}\mathbb{D}(f_j,\chi(n)n^{it};x)\geq \log \log \log x\}.
\end{align*}
If $\mathcal{Y}$ is infinite, the claim follows from Theorem~\ref{thm_elliott_2point_stronger} (taking $A(x)=\log \log \log x$) and Markov's inequality. Suppose then that $\mathcal{Y}$ is finite.  Then $f_1$ is moderately non-pretentious (recalling that by assumption $f_1$ is non-pretentious) and $f_2$ is either pretentious or moderately non-pretentious. Hence, the claim follows from Theorem~\ref{thm_main}.
\end{proof}

We next prove Proposition~\ref{prop_MRT}, which complements Theorem~\ref{thm_elliott_2point}.

\begin{proof}[Proof of Proposition~\ref{prop_MRT}] We modify the  parameters in the construction of Matom\"aki--Radziwi\l{}\l{}--Tao~\cite[Appendix B]{MRT} of a multiplicative function violating two-point Elliott. We define a sequence $(f(p))_p$ iteratively as follows. 
\begin{enumerate}
    \item Set $t_1=100$ and $f(p)=1$ for $p\leq t_1$. 

    \item Suppose that $f(p)$ has been defined for $p\leq t_m$. Since the vector $(p^{it})_{p\leq t_m}$ is asymptotically equidistributed in $\prod_{p\leq t_m}S^1$ as $t\to \infty$, we can find some $s_{m+1}>\exp(t_{m})$ such that
    \begin{align*}
     |p^{is_{m+1}}-f(p)|\leq \frac{1}{t_m^2}   
    \end{align*}
    for all $p\leq t_m$. Set $t_{m+1}=\exp(s_{m+1})$ and define $f(p)=p^{is_{m+1}}$ for $t_m<p\leq t_{m+1}$. Now  iterate step (2).
\end{enumerate}
Then we define a completely multiplicative function $f:\mathbb{N}\to \mathbb{D}$ by
\begin{align*}
 f(n)=\prod_{p^{\ell}\mid\mid n}f(p)^{\ell}.   
\end{align*}

Now, for any $n\in [s_{m+1}^2,t_{m+1}]$ such that $p^{\ell}\mid n$ with $p\leq t_m$ implies $\ell\leq \log t_m$, we compute
\begin{align*}
f(n)&=\prod_{p^{\ell}\mid\mid n}f(p)^{\ell}=\prod_{\substack{p^{\ell}\mid \mid n\\p\leq t_m}}f(p)^{\ell}\cdot \prod_{\substack{p^{\ell}\mid\mid n\\t_m<p\leq t_{m+1}}}f(p)^{\ell}\\
&=\prod_{\substack{p^{\ell}\mid \mid n\\p\leq t_m}}\left(p^{is_{m+1}}+O\left(\frac{1}{t_m^2}\right)\right)^{\ell}\cdot \prod_{\substack{p^{\ell}\mid\mid n\\t_m<p\leq t_{m+1}}}p^{is_{m+1}\ell}\\
&=n^{is_{m+1}}\left(1+O\left(\frac{1}{t_m^2}\right)\right)^{O(t_m)}= n^{is_{m+1}}+O\left(\frac{1}{t_m}\right),
\end{align*}
so 
\begin{align*}
f(n)\overline{f(n+1)}&=\left(1-\frac{1}{n+1}\right)^{is_{m+1}}+O\left(\frac{1}{t_m}\right)\\
&=1+O\left(\frac{s_{m+1}}{s_{m+1}^2}\right)+O\left(\frac{1}{t_m}\right)\\
&=1+O\left(\frac{1}{t_m}\right).   
\end{align*}

Now let $x\in [(\log t_{m+1})^4,t_{m+1}]$. Note that the proportion of $n\leq x$ that satisfy $p^{\ell}\mid n$ for some $p\leq t_m$ and $\ell>\log t_m$ is $o_{x\to \infty}(1)$. Hence, we have by the union bound (recalling that $x^{1/2}\geq (\log t_{m+1})^2= s_{m+1}^2$)
\begin{align}\label{eq27}\begin{split}
 \frac{1}{x}\sum_{n\leq x}f(n)\overline{f(n+1)}&=\frac{1}{x}\sum_{x^{1/2}\leq n\leq x}f(n)\overline{f(n+1)}+o_{x\to \infty}(1)=1+o_{x\to \infty}(1).
 \end{split}
\end{align}

Let $\mathcal{X}$ be the set of $x\in \mathbb{N}$ for which~\eqref{eq27} holds. We have
\begin{align*}
\frac{1}{\log \log t_{m+1}}\sum_{n\leq t_{m+1}}\frac{1_{n\in \mathcal{X}}}{n\log n}\geq \frac{1}{\log \log t_{m+1}}\sum_{(\log t_{m+1})^4\leq n\leq t_{m+1}}\frac{1}{n\log n}=1-o_{m\to \infty}(1), 
\end{align*}
so we have $\delta^{+}_{\log \log}(\mathcal{X})=1$. 

What remains to be shown is that $f$ is non-pretentious. Let $\chi$ be any fixed Dirichlet character and $t$ any fixed real number.  Then by Lemma~\ref{le_pretentious_distance} and Mertens's theorem (recalling $s_{m+1}=\log t_{m+1}$) we have
\begin{align*}
 \mathbb{D}(f,\chi(n)n^{it};t_m,t_{m+1})^2&=\mathbb{D}(n^{is_{m+1}},\chi(n)n^{it};t_m,t_{m+1})^2\\
 &\geq \mathbb{D}(n^{i\log t_{m+1}},\chi(n)n^{it};t_{m+1})^2-O(\log \log t_m)\\
 &\gg   \log \log t_{m+1}.   
\end{align*}
Letting $m\to \infty$, we see that $f$ is non-pretentious. 
\end{proof}

\begin{remark}
By modifying the growth rate of $t_m$ in terms of $s_m$ in the construction above, one can show that for any function $\psi(X)\leq X$ tending to infinity as $X\to \infty$ and for any $\delta>0$ there exists a multiplicative function $f:\mathbb{N}\to \mathbb{D}$ such that
\begin{align}\label{eq:psi}
\min_{y\in [\psi(X),X]}\left|\frac{1}{x}\sum_{n\leq x}f(n)\overline{f(n+1)}\right|\geq 1-\delta   
\end{align}
holds for infinitely many $X\in \mathbb{N}$. 
\end{remark}

\subsection{Higher order correlations}

\begin{proof}[Proof of Theorem~\ref{thm_elliott_higher}] Let us first note that we can assume for every $c>0$ that
\begin{align}\label{eq21}
\sum_{\substack{p\in \mathbb{P}\\|f(p)|\leq 1-c}}\frac{1}{p}<\infty,
\end{align}
since otherwise by the triangle inequality and Delange's theorem~\cite[Thm. III.4.2]{Tenenbaum} we have
\begin{align}\label{eq:correlate24}
 \left|\frac{1}{x}\sum_{n\leq x}f^{e_1}(n+h_1)\cdots f^{e_k}(n+h_k)\right|\leq \frac{1}{x}\sum_{n\leq x}|f^{e_1}(n+h_1)| \ll \exp\left(\sum_{p\leq x}\frac{1-|f^{e_1}(p)|}{p}\right)=o(1),  
\end{align}
since $|f(p)|\leq 1-c$ implies $|f(p)^{e_1}|\leq 1-c$.

Set $m := e_1 + \cdots + e_k$. By Theorem~\ref{thm_A}(2), we have the desired conclusion (with $\mathcal{X}$ satisfying $\delta_{\log}(\mathcal{X})=1$, so in particular $\delta_{\log \log}(\mathcal{X})=1$), unless there exist a Dirichlet character $\chi$ and a real number $t$ such that
\begin{align}\label{eq19}
 \mathbb{D}(f^{m},\chi(n)n^{it};x)^2=o(\log \log x)   
\end{align}
for $x\geq 10$. 

We shall first show that $t=0$. Let $\ell$ be the order of $\chi$. Then by the triangle inequality for the pretentious distance we deduce that
\begin{align}\label{eq18}
 \mathbb{D}(f^{d\ell m},n^{id\ell t};x)^2=o(\log \log x).   
\end{align}
Since $f$ takes values in the convex hull of the $d$th roots of unity, there exists some $\eta>0$ such that all the values $f^{d\ell m}(p)$ are at least distance $\eta$ away from $-1$. Now, we can estimate
\begin{align}\label{eq25}
 \mathbb{D}(f^{d\ell m},n^{id\ell t};x)^2\geq \sum_{\substack{p\leq x\\|p^{itd\ell }+1|\leq \eta/3}}\frac{1-\textnormal{Re}(f(p)^{d\ell m}p^{-id\ell t})}{p}\gg_{\eta} \sum_{\substack{p\leq x\\|p^{id\ell t}+1|\leq \eta/3}}\frac{1}{p}.    
\end{align}
If $t \neq 0$ observe that for a large enough constant $j_0(\eta)\geq 1$ we have
$$
p \in \bigcup_{j \geq j_0(\eta)} [A_j,B_j]\implies |p^{id\ell t} + 1|\leq \eta/3,
$$
where we have set
$$
A_j = e^{\frac{\pi}{d\ell|t|}(2j+1-\eta/(10\pi))}, \quad B_j = e^{\frac{\pi}{d\ell|t|}(2j+1+\eta/(10\pi))}.
$$
Note that $B_j = e^{\eta/(5d\ell|t|)} A_j$, so if $j \geq j_0(\eta)$ then the prime number theorem implies that
$$
\sum_{A_j \leq p \leq B_j} \frac{1}{p} \gg_{\eta} \frac{B_j-A_j}{A_j \log A_j} \gg_{\eta} \frac{1}{\log A_j} \gg \frac{1}{j},
$$
and thus 
\begin{align}\label{eq25b}
\sum_{\substack{p\leq x\\|p^{id\ell t}+1|\leq \eta/3}}\frac{1}{p} \geq \sum_{j_0(\eta) \leq j \leq \tfrac{d\ell |t|}{2\pi} \log x-1}\, \sum_{A_j \leq p \leq B_j} \frac{1}{p} \gg_{\eta} \log\left(\frac{d\ell|t|(\log x)}{2\pi}-1\right) \gg \log\log x.
\end{align}
Together with~\eqref{eq25} this contradicts~\eqref{eq18}. Thus, we may assume that $t=0$ in~\eqref{eq18}.

Now~\eqref{eq19} reads as
\begin{align*}
 \mathbb{D}(f^{m},\chi;x)^2=o(\log \log x).    
\end{align*}
Since $(d,m)=1$, there exist $a,b\in \mathbb{N}$ such that $am=1+bd$. Hence, by the triangle inequality,
\begin{align}\label{eq20}
 \mathbb{D}(f^{1+bd},\chi^{a};x)^2=o(\log \log x).    
\end{align}

Let $\eta>0$ be small, and let 
\begin{align*}
 \mathcal{P}_{\eta}=\{p\in \mathbb{P}:\,\, |f(p)^{bd}-1|\geq \eta\}.   
\end{align*}
Note that
\begin{align*}
 |f(p)^{bd}-1|\geq \eta\implies \min_{1\leq j\leq bd}\left|f(p)-e\left(\frac{j}{bd}\right)\right|\gg \eta.   
\end{align*}
Since $f$ takes values in the convex hull of the $d$th roots of unity $f$, this implies that $|f(p)|\leq 1-c_0\eta$ for some constant $c_0>0$. Hence, 
\begin{align*}
p\in \mathcal{P}_{\eta}\implies |f(p)|\leq 1-c_0\eta.
\end{align*}
But by~\eqref{eq21} this implies that
\begin{align*}
\sum_{p\in \mathcal{P}_{\eta}}\frac{1}{p}\ll_{\eta} 1.     
\end{align*}
Therefore, for any $\eta>0$ we have by Mertens's theorem
\begin{align*}
\mathbb{D}(f^{bd},1;x)^2\leq O_{\eta}(1)+\sum_{p\leq x}\frac{\eta}{p}\leq O_{\eta}(1)+ \eta \log \log x.    
\end{align*}
Sending $\eta\to 0$, we conclude that
\begin{align*}
 \mathbb{D}(f^{bd},1;x)^2=o(\log \log x).   
 \end{align*}
 
Now, recalling~\eqref{eq20} and applying the triangle inequality, we see that
\begin{align*}
 \mathbb{D}(f,\chi^{a};x)^2\leq    
 2(\mathbb{D}(f^{1+bd},\chi^{a};x)^2+\mathbb{D}(\overline{f}^{bd},1;x)^2)=o(\log \log x).   
\end{align*}
Hence, by the triangle inequality again, for each $1\leq j\leq k$ we have
\begin{align*}
 \mathbb{D}(f^{e_j},\chi^{ae_j};x)^2=o(\log \log x).     \end{align*}
This implies that each of $f_1^{e_1},\ldots, f_k^{e_k}$ is either pretentious of moderately non-pretentious.

We shall show that $f_1^{e_1}$ cannot be pretentious. Supposing for a contradiction that $f_1^{e_1}$ is pretentious,   there exist some real number $u$ and Dirichlet character $\chi'$ such that
\begin{align*}
 \mathbb{D}(f_1^{e_1},\chi'(n)n^{iu};\infty)<\infty.    
\end{align*}
We must have $u\neq 0$, as otherwise we contradict the assumption of the theorem. If $\ell'$ is the order of $\chi'$, then the triangle inequality for the pretentious distance gives
\begin{align}\label{eq25c}
 \mathbb{D}(f_1^{de_1\ell'},n^{id\ell' u};\infty)<\infty.
\end{align}
By our assumption that $f_1$ takes values in the convex hull of the $d$th roots of unity, there exists some $\eta>0$ such that $|f_1(p)^{d e_1\ell'}+1|\geq \eta$ for all primes $p$. This leads to a contradiction with~\eqref{eq25c}, arguing as in~\eqref{eq25} and~\eqref{eq25b}. Summarizing, we now know that $f_1^{e_1}$ is moderately non-pretentious and $f_2^{e_2},\ldots, f_k^{e_k}$ are either moderately non-pretentious or pretentious, so the claim follows from Theorem~\ref{thm_main}. 
\end{proof}

\begin{proof}[Proof of Corollary~\ref{cor_odd}]
This follows immediately from Theorem~\ref{thm_elliott_higher} with $e_1=\cdots =e_k=1$ and $d=2$ after we note that for $f:\mathbb{N}\to [-1,1]$ and $\chi$ any non-real Dirichlet character we always have $\mathbb{D}(f,\chi;\infty)=\infty$; this follows from Lemma~\ref{le:real}.
\end{proof}

\begin{proof}[Proof of Theorem~\ref{thm_omega}]
Let $f:\mathbb{N}\to \{-1,+1\}$ be either of the two multiplicative functions $(-1)^{\omega_{\mathcal{P}}(n)}$ or $(-1)^{\Omega_{\mathcal{P}}(n)}$. Since $\sum_{p\in \mathcal{P}}\frac{1}{p}=\infty$, we have
\begin{align*}
 \mathbb{D}(f,1;\infty)=\infty.   
\end{align*}

On the other hand, since $\mathcal{P}$ has relative density $0$ in the primes, for any fixed $\varepsilon>0$ we have
\begin{align*}
 \mathbb{D}(f,1;x^{\varepsilon},x)^2=\sum_{\substack{x^{\varepsilon}< p\leq x\\p\in \mathcal{P}}}\frac{2}{p}=o(1)   
\end{align*}
as $x\to \infty$ by partial summation and the fact that $|\{p\in [y/2,y]:\,\, p\in \mathcal{P}\}|=o(\pi(y))$. Therefore, there exists some function $\eta(X)$ tending to zero slowly as $X\to \infty$ such that
\begin{align*}
\mathbb{D}(f,1;x^{\eta(x)},x)^2=o(1)    
\end{align*}
as $x\to \infty$. Hence, the claim follows from Theorem~\ref{thm_correlation_conclusion}. 
\end{proof}

\begin{proof}[Proof of Corollary~\ref{cor:sarnak}] This follows immediately by combining Theorem~\ref{thm_omega} with~\cite[Theorem 4.10]{HKPLR}.
\end{proof}

\begin{proof}[Proof of Proposition~\ref{prop_elliott_implication}] We shall prove the two claimed implications separately.

\textbf{Implication from Conjecture~\ref{conj_elliott_Tao} to Conjecture~\ref{conj_elliott:modified}.}

Since $f_1$ is non-pretentious, there must exist some function $M(x)$ tending to infinity such that
\begin{align}\label{eq:M(x)}
\inf_{|t|\leq M(x)}\,\, \min_{\substack{\chi\pmod q\\q\leq M(x)}}\mathbb{D}(f_1,\chi(n)n^{it};x)\geq M(x)   
\end{align}
for infinitely many $x\in \mathbb{N}$ (using~\eqref{eq:simple} with  $F(x,\chi, N):=1/\max_{|t|\leq N}\mathbb{D}(f_1,\chi(n)n^{it};x)$). 

 Denote by $\mathcal{M}$ the collection of multiplicative functions taking values in $\mathbb{D}$. By Conjecture~\ref{conj_elliott_Tao}, there exist functions $x_0(\cdot)$ and $A(\cdot)$ such that the following holds. For any $\varepsilon>0$ $k\in \mathbb{N}$, $H\geq 1$, if $x\geq x_0(\varepsilon,k,H)$ and $h_1,\ldots, h_k\in [1,H]$ are distinct natural numbers, we have
 \begin{align*}
\max_{f_2,\ldots, f_k\in \mathcal{M}} \left|\frac{1}{x}\sum_{n\leq x}f_1(n+h_1)\cdots f_k(n+h_k)\right|\leq \varepsilon  \end{align*}
 unless 
 \begin{align*}
 \inf_{|t|\leq A(\varepsilon,k,H)x^{k-1}}\,\,\min_{\substack{\chi\pmod q\\q\leq A(\varepsilon,k,H)}}\mathbb{D}(f_1,\chi(n)n^{it};x)\leq A(\varepsilon,k,H).     
 \end{align*}
By letting $\varepsilon(x)$ be a function tending to $0$ sufficiently slowly (in terms of $x_0(\cdot)$, $A(\cdot)$ and $M(\cdot)$), we see that 
 \begin{align}\label{eq26}
\max_{J\leq 1/\varepsilon(x)}\max_{f_2,\ldots, f_J\in \mathcal{M}}\left|\frac{1}{x}\sum_{n\leq x}\prod_{1\leq j\leq J}f_j(n+j)\right|\leq \varepsilon(x)   
 \end{align}
 unless 
 \begin{align}\label{eq26b}
 \inf_{|t|\leq x^{B(x)}}\min_{\substack{\chi\pmod q\\q\leq B(x)}}\mathbb{D}(f_1,\chi(n)n^{it};x)\leq B(x),     
 \end{align}
 say, for some function $B(x)$ that may be made to tend to infinity arbitrarily slowly by making $\varepsilon(x)\to 0$ slowly enough. In particular, we may require that $B(x)\leq M(x)^{1/2}$ and $B(e^{x})\leq B(x)+1$ for all $x\in \mathbb{N}$.
 
Let $a\geq 2$. Let $\mathcal{E}$ be the set of $x$ for which~\eqref{eq26b} holds. If $\mathcal{E}$ contains a subsequence $(x_n)$ with $x_n+1\leq x_{n+1}\leq x_n^{a}$ for some $a > 1$, then by Lemma~\ref{le_rigidity}, and the fact that $B(x)$ grows sufficiently slowly, there exist a real number $t_1$ and a Dirichlet character $\chi_1$ such that
 \begin{align*}
  \mathbb{D}(f_1,\chi_1(n)n^{it_1};x)\ll B(x)\leq  M(x)^{1/2}
  \end{align*}
 for $x\geq 10$. But this contradicts~\eqref{eq:M(x)}. Hence, the set $\mathcal{E}$ of $x\in \mathbb{N}$ for which~\eqref{eq26b} holds contains no such subsequence $(x_n)$ for any $a > 1$. But then $\delta^{-}_{\log}(\mathcal{E})=o_{a\to \infty}(1)$, and sending $a\to \infty$ we obtain $\delta^{-}_{\log}(\mathcal{E})=0$, so $\delta^{+}_{\log}(\mathbb{N}\setminus \mathcal{E})=1$. Noting that~\eqref{eq26} holds for $x\in \mathbb{N}\setminus \mathcal{E}$, the statement of Conjecture~\ref{conj_elliott:modified} follows with $\mathcal{X}=\mathbb{N}\setminus\mathcal{E}$, noting that $\mathcal{X}$ depends only on $f_1$.

 \textbf{Implication from Conjecture~\ref{conj_elliott:modified} to Conjecture~\ref{conj_sarnak}. }

 To see that Conjecture~\ref{conj_elliott:modified} implies Conjecture~\ref{conj_sarnak}, we may follow verbatim Sarnak's combinatorial proof (see~\cite{tao-blog-sarnak}) that the Chowla conjecture implies the Elliott conjecture. For the sake of completeness we give details. Let $f:\mathbb{N}\to \mathbb{D}$ be a non-pretentious multiplicative function. Let $\mathcal{X}$ be the set corresponding to $f$ in Conjecture~\ref{conj_elliott:modified}. Then $\delta^{+}_{\log}(\mathcal{X})=1$ and
 \begin{align}\label{eq:sarnak4}
 \lim_{\substack{x\to \infty\\x\in \mathcal{X}}}\mathbb{E}_{n\leq x}f(n+h)\prod_{j=1}^kf_j(n+h_j)=0    
 \end{align}
 for any $k\geq 1$, any multiplicative functions $f_j:\mathbb{N}\to \mathbb{D}$ and any $h_1,\ldots, h_k\in \mathbb{N}$ distinct from $h$.

 Let $a:\mathbb{N}\to \mathbb{D}$ be deterministic, and let \begin{align*}
 x\ggg H\ggg 1/\varepsilon\ggg 1   
 \end{align*} 
with $x\in \mathcal{X}$ be parameters, each one large enough in terms of the ones to the right of it. By the triangle inequality, we can estimate
\begin{align}\label{eq:sarnak1}
\left|\mathbb{E}_{n\leq x}f(n)a(n)\right|\leq \mathbb{E}_{n\leq x}\left|\mathbb{E}_{j\leq H}f(n+j)a(n+j)\right|+o_{H\to \infty}(1).  
\end{align}
This can be bounded as 
\begin{align*}
\ll \varepsilon+\mathbb{P}_{n\leq x}(\{\left|\mathbb{E}_{j\leq H}f(n+j)a(n+j)\right|\geq \varepsilon\}),
\end{align*}
where for any event $E_n$ we write $\mathbb{P}_{n\leq x}(E_n):=\frac{1}{x}|\{n\leq x:\,\, E_n\textnormal{ holds}\}|$.

Since $a$ is deterministic, the set $((a(n+1),\ldots, a(n+H)))_{n\in \mathbb{N}}$ can be covered with $O(\exp(\varepsilon^{10}H))$ balls of radius $\varepsilon/2$ (since $H$ is large in terms of $\varepsilon$). Using this, the triangle inequality and the union bound, we see that
\begin{align}\label{eq:sarnak2}
 \mathbb{P}_{n\leq x}(\{\left|\mathbb{E}_{j\leq H}f(n+j)a(n+j)\right|\geq \varepsilon\})\ll \exp(\varepsilon^{10}H)  \mathbb{P}_{n\leq x}(\{\left|\mathbb{E}_{j\leq H}f(n+j)\zeta_j\right|\geq \varepsilon/2\})  
\end{align}
for some $\zeta_j\in \mathbb{D}$. 
By Chebyshev's inequality, for any integer $1\leq m\leq H/2$ (to be specified later) the right-hand side of~\eqref{eq:sarnak2} is
\begin{align*}
\ll \exp(\varepsilon^{10}H)  \frac{1}{(\varepsilon/2)^{2m}}\mathbb{E}_{n\leq x}\left|\mathbb{E}_{j\leq H}f(n+j)\zeta_j\right|^{2m}.    
\end{align*}
Expanding out $2m$th moment here, this becomes
\begin{align}\label{eq:sarnak3}
 \exp(\varepsilon^{10}H)  \frac{1}{(\varepsilon/2)^{2m}} \frac{1}{H^{2m}}\sum_{j_1,\ldots, j_{2m}\leq H}\zeta_{j_1}\cdots \zeta_{j_m}\overline{\zeta_{j_{m+1}}}\cdots \overline{\zeta_{j_{2m}}}\mathbb{E}_{n\leq x} \prod_{i=1}^m f(n+j_i)\prod_{i=1}^m \overline{f}(n+j_{i+m}).  
\end{align}

Using~\eqref{eq:sarnak4} (which is based on Conjecture~\ref{conj_elliott:modified}), the contribution of all the terms in~\eqref{eq:sarnak3} where at least one of the $j_i$ occurs with multiplicity $1$ in $(j_1,\ldots, j_{2m})$ is $o_{x\to \infty}(1)$, recalling that $x\in \mathcal{X}$. We then upper bound the number of tuples $j_1,\ldots, j_{2m}$ for which all the $j_i$ have multiplicity $\geq 2$ in $(j_1,\ldots, j_{2m})$. Such tuples necessarily satisfy $|\{j_1,\ldots, j_{2m}\}|\leq m$, so there are at most
\begin{align*}
 \sum_{1\leq j\leq m}\binom{H}{j}(2m)!\leq  m(2m)!\binom{H}{m}\leq m(2m)!\frac{H^m}{m!}  
\end{align*}
such tuples (recalling that $m\leq H/2$), each giving a contribution of $\leq 1$ to the sum in~\eqref{eq:sarnak3}. Hence,~\eqref{eq:sarnak3} is bounded by
\begin{align*}
 \ll  \exp(\varepsilon^{10}H) m(2m)!\frac{H^{-m}}{m!}(\varepsilon/2)^{-2m}+o_{x\to \infty}(1). 
\end{align*}
Using the simple inequalities $n\leq 2^n$ and $n^ne^{-n}\leq n!\leq n^n$, the main term here is 
\begin{align*}
  \exp(\varepsilon^{10}H)\left(\frac{2\cdot (2m)^2}{H(\varepsilon/2)^2m/e}\right)^m=\exp(\varepsilon^{10}H)\left(\frac{32em}{H\varepsilon^2}\right)^m.  
\end{align*}
Taking $m$ to be the largest integer less than $\varepsilon^2H/(64e)$ (which is necessarily $>1$ by the assumption that $H$ is large in terms of $\varepsilon$), the previous expression is bounded by
\begin{align*}
\ll  \exp(\varepsilon^{10}H)2^{-m}\ll \exp\left(\varepsilon^{10}H-\frac{\varepsilon^2}{10}\frac{H}{64e}\right)\ll \exp(-\varepsilon^3 H),   
\end{align*}
using again $H\ggg 1/\varepsilon\ggg 1$. Collecting all the error terms, we see that 
\begin{align*}
 \limsup_{\substack{x\to \infty\\x\in \mathcal{X}}} \left|\mathbb{E}_{n\leq x}f(n)a(n)\right|\ll \varepsilon+\exp(-\varepsilon^3H).   
\end{align*}
Since $H\geq \varepsilon^{-10}$ and $1/\varepsilon$ can be taken to be arbitrarily large, the claim follows. 
\end{proof}

\section{Constructing Furstenberg systems}\label{proof_furst}

For proving Theorem~\ref{thm:furst}, we first need to know that pretentious multiplicative functions have unique Furstenberg systems.\footnote{We thank Nikos Frantzikinakis for pointing out the following argument to us.}

\begin{proposition}\label{prop:furst}
 Let $g:\mathbb{N}\to \{0,1\}$ be a pretentious multiplicative function. Then $g$ has a unique Furstenberg system.    
\end{proposition}

\begin{proof}  By~\cite[Theorem 6]{daboussi-delange}, $g$ is ``rationally almost periodic'' in the sense that for any $\varepsilon>0$ and $n\geq 1$ we can write 
\begin{align}\label{eq:rap}
g(n)=P_{\varepsilon}(n)+E_{\varepsilon}(n)\,\,\textnormal{ where }\,\,\limsup_{x\to \infty}\frac{1}{x}\sum_{n\leq x}|E_{\varepsilon}(n)|\leq \varepsilon,
\end{align}
where $P_{\varepsilon}$ is some periodic function.
It follows from~\cite[Theorem 1.7]{RAP} that any rationally almost periodic sequence taking values in a finite set has a unique Furstenberg system. Hence, $g$ has a unique Furstenberg system. 
\end{proof}

We then show that a class of non-pretentious multiplicative functions has a unique and explicitly given Furstenberg measure.

\begin{theorem}\label{thm:furst2} Let $g:\mathbb{N}\to \{0,1\}$ be a pretentious multiplicative function, and let $\nu_g$ be the corresponding unique Furstenberg measure on $\{0,1\}^{\mathbb{Z}}$. Let $\mathcal{P}$ be any subset of the primes with $\sum_{p\in \mathcal{P}}\frac{1}{p}=\infty$ such that the relative density of $\mathcal{P}$ within $\mathbb{P}$ is $0$. Then, if $f(n)=(-1)^{\Omega_{\mathcal{P}}(n)}$ or $f(n)=(-1)^{\omega_{\mathcal{P}}(n)}$, the function $fg$ has a unique Furstenberg measure $\nu_{fg}$ given on cylinder sets $C$ of $\{-1,0,+1\}^{\mathbb{Z}}$ by
\begin{align}\label{eq:furst2}\nu_{fg}(C)=2^{-|\textnormal{Supp}(C)|}\nu_g(C^2),     
\end{align}
where $\textnormal{Supp}(C)$ is the set of coordinates of $C$ fixed to be $\neq 0$ and $C^2:=\{\mathbf{x}^2:\,\, \mathbf{x}\in C\}$, where $\mathbf{x}^2(j):=\mathbf{x}(j)^2$ for all $j\in \mathbb{Z}$.
\end{theorem}

\begin{remark}
Note that there are many explicit choices of $f$ satisfying the assumption. For example, if $p_m$ denotes the $m$th prime one can take $f(n)=(-1)^{\Omega_{\mathcal{P}}(n)}$,  $\mathcal{P}=\{p_{m\lfloor \log \log \log m\rfloor}:\,\, m\geq e^{100}\}$.    
\end{remark}

\begin{proof}
 Let $\mathcal{F}$ be the collection of completely multiplicative functions $f:\mathbb{N}\to \{-1,+1\}$ that satisfy Elliott's conjecture in the following sense: for any $k\geq 1$, any integers $a_1,\ldots, a_k\geq 1$  and any distinct $h_1,\ldots, h_k\in \mathbb{N}$ with $a_ih_j\neq a_jh_i$ for all $i\neq j$ we have
\begin{align}\label{eq:correlate7}
 \lim_{x\to \infty}\frac{1}{x}\sum_{n\leq x}f(a_1n+h_1)\cdots f(a_kn+h_k)=0.   
\end{align}
By Theorem~\ref{thm_omega} we have
\begin{align}\label{eq:furst1}
\mathcal{F}\supset \{n\mapsto (-1)^{\Omega_{\mathcal{P}}(n)}:\,\, \sum_{p\in \mathcal{P}}\frac{1}{p}=\infty,\,\, d_{\mathbb{P}}(\mathcal{P})=0\},    
\end{align}
where $d_{\mathbb{P}}(\mathcal{P})$ denotes the relative density of $\mathcal{P}$ within $\mathbb{P}$. The same holds also with $\omega_{\mathcal{P}}(n)$ in place of $\Omega_{\mathcal{P}}(n)$.

We claim that for any $f\in \mathcal{F}$ the multiplicative function $fg:\mathbb{N}\to \{-1,0,+1\}$ has a unique Furstenberg measure, and that this measure satisfies~\eqref{eq:furst2}. After that the desired claim follows~\eqref{eq:furst1}.

Using the identity \begin{align*}
 1_{fg(n)=v}=\frac{1}{2}g(n)+\frac{v}{2}fg(n)   
\end{align*}
for $v\in \{-1,+1\}$, we see that for any distinct integers $h_1,\ldots, h_k$ and $v_1,\ldots, v_k\in \{-1,0,+1\}$ and for any $x\geq 1$ we have 
\begin{align}\label{eq:correlate3}
\frac{1}{x}\sum_{n\leq x}1_{fg(n+h_1)=v_1}\cdots 1_{fg(n+h_k)=v_k}&= \frac{1}{x}\sum_{n\leq x}\prod_{\substack{j\leq k\\v_j\neq 0}}\left(\frac{1}{2}g(n+h_j)+\frac{v_j}{2}fg(n+h_j)\right)\prod_{\substack{j\leq k\\v_j=0}}(1-g(n+h_j)).
\end{align}
We claim that 
\begin{align}\label{eq:correlate5}
\frac{1}{x}\sum_{n\leq x}1_{fg(n+h_1)=v_1}\cdots 1_{fg(n+h_k)=v_k}&= \frac{1}{x}\sum_{n\leq x}\prod_{\substack{j\leq k\\v_j\neq 0}}\left(\frac{1}{2}g(n+h_j)\right)\prod_{\substack{j\leq k\\v_j=0}}(1-g(n+h_j))+o(1). 
\end{align}
To deduce this from~\eqref{eq:correlate3}, it suffices to show that for any nonempty set $A\subset \{h_1,\ldots, h_k\}$ and any set $B\subset \{h_1,\ldots, h_k\}\setminus A$ we have
\begin{align}\label{eq:correlate6}
 \lim_{x\to \infty}\frac{1}{x}\sum_{n\leq x}\prod_{j\in A}fg(n+j)\prod_{j\in B}g(n+j)=0.   
\end{align}
By applying the decomposition~\eqref{eq:rap} to each occurrence of $g$ in~\eqref{eq:correlate6}, and noting that as $\varepsilon\to 0$ the contribution of any term involving $E_{\varepsilon}$ is negligible, we see that~\eqref{eq:correlate6} follows if we show that 
\begin{align*}
 \lim_{x\to \infty}\frac{1}{x}\sum_{n\leq x}\prod_{j\in A}f(n+j)1_{n\equiv C\pmod{D}}=0   
\end{align*}
for any $C,D\in \mathbb{N}$. But this follows from~\eqref{eq:correlate7} after making a change of variables.

We now conclude that~\eqref{eq:correlate5} holds, so 
\begin{align*}
 \lim_{x\to \infty}\frac{1}{x}\sum_{n\leq x}1_{fg(n+h_1)=v_1}\cdots 1_{fg(n+h_k)=v_k}&= \lim_{x\to \infty}\frac{1}{x}\sum_{n\leq x}\prod_{\substack{j\leq k\\v_j\neq 0}}\left(\frac{1}{2}g(n+h_j)\right)\prod_{\substack{j\leq k\\v_j=0}}(1-g(n+h_j));    
\end{align*}
the limit on the right-hand side exists by Proposition~\ref{prop:furst}. Hence we have 
\begin{align}\label{eq:correlate3b}
 \lim_{x\to \infty}\frac{1}{x}\sum_{n\leq x}1_{fg(n+h_1)=v_1}\cdots 1_{fg(n+h_k)=v_k}&=2^{-r} \lim_{x\to \infty}\frac{1}{x}\sum_{n\leq x}\prod_{\substack{j\leq k\\v_j\neq 0}}1_{g(n+h_j)=1}\prod_{\substack{j\leq k\\v_j=0}}1_{g(n+h_j)=0},    
\end{align}
where $r$ is the number of $j\leq k$ for which $v_j\neq 0$. 
In particular, since these limits exist, there exists a unique Furstenberg measure $\nu_{fg}$ for $fg$. Let $\nu_g$ be the unique Furstenberg measure of $g$.

Now, by the definition of a Furstenberg measure, and the fact that the algebra generated by the indicators of cylinder sets is dense in the space of continuous functions with respect to the uniform topology, we  have
\begin{align*}
\lim_{x\to \infty}\frac{1}{x}\sum_{n\leq x}1_{fg(n+h_1)=v_1}\cdots 1_{fg(n+h_k)=v_k}&=\nu_{fg}(\mathbf{x}\in \{-1,0,+1\}^{\mathbb{Z}}:\,\, \mathbf{x}(h_i)=v_j\,\forall\, j\leq k),\\
\lim_{x\to \infty}\frac{1}{x}\sum_{n\leq x}1_{g(n+h_1)=v_1^2}\cdots 1_{g(n+h_k)=v_k^2}&=\nu_{g}(\mathbf{x}\in \{0,1\}^{\mathbb{Z}}:\,\, \mathbf{x}(h_j)=v_j^2\,\,\forall\, j\leq k)
\end{align*}
for any $v_1,\ldots, v_k\in \{-1,0,+1\}$ and any distinct integers $h_1,\ldots, h_k$. By ~\eqref{eq:correlate3b} this implies that
\begin{align*}
\nu_{fg}(\mathbf{x}\in \{-1,0,+1\}^{\mathbb{Z}}:\,\, \mathbf{x}(h_j)=v_j\,\forall\, j\leq k)=2^{-r}\nu_{g}(\mathbf{x}\in \{0,1\}^{\mathbb{Z}}:\,\, \mathbf{x}(h_j)=v_j^2\,\forall\, j\leq k)    
\end{align*}
for any $v_1,\ldots, v_k\in \{-1,0,+1\}$ and any $h_1,\ldots, h_k\in \mathbb{Z}$, where $r$ is the number of $j$ such that $v_j\neq 0$. Hence, for any cylinder set $C$ of $\{-1,0,+1\}^{\mathbb{Z}}$ we have
\begin{align*}
\nu_{fg}(C)=2^{-|\textnormal{supp}(C)|}\nu_g(C^2),   
\end{align*}
as desired.
\end{proof}

Theorem~\ref{thm:furst} will follow immediately by combining Theorem~\ref{thm:furst2} and the following proposition (cf.~\cite[(15) and Section 4.3]{HKPLR}).\footnote{We thank Mariusz Lema\'nczyk for discussions relating to this.}

\begin{proposition}
 Let $\nu$ be a shift-invariant measure on $\{0,1\}^{\mathbb{Z}}$. Let $\widetilde{\nu}$ be a measure on $\{-1,0,+1\}^{\mathbb{Z}}$ given on cylinder sets $C$ of  $\{-1,0,+1\}^{\mathbb{Z}}$ by
 \begin{align*}
  \widetilde{\nu}(C)=2^{-|\textnormal{Supp}(C)|}\nu(C^2).    
 \end{align*}
 Let $T$ be the left shift on $\{-1,0,+1\}^{\mathbb{Z}}$. 
 Then the dynamical system $(\{-1,0,+1\}^{\mathbb{Z}},T, \widetilde{\nu})$ is isomorphic to the product of the system $(\{0,1\}^{\mathbb{Z}},T,\nu)$ and a Bernoulli system. 
\end{proposition}

\begin{proof} 
Define a measure $\nu'$ on cylinder sets $C$ of $\{-1,0,+1\}^{\mathbb{Z}}$ by $\nu'(C)=\nu(C^2)$. Consider the following dynamical systems 
\begin{align*}
\mathcal{S}_1&=(\{0,1\}^{\mathbb{Z}}\times \{-1,+1\}^{\mathbb{Z}},T\otimes T,\nu\otimes B(1/2)), \\
\mathcal{S}_2&=(\{-1,0,+1\}^{\mathbb{Z}},T,\nu'\cdot B(1/2)),\\
\mathcal{S}_3&=(\{0,1\}^{\mathbb{Z}},T,\nu),
\end{align*}
where $\nu_1\cdot \nu_2(A):=\nu_1(A)\nu_2(A)$ for all $A$ and $B(1/2)$ is the Bernoulli measure $(1/2,1/2)^{\otimes \mathbb{Z}}$. Note that $\tilde{\nu}=\nu'\cdot B(1/2)$. The system $\mathcal{S}_2$ is a factor of $\mathcal{S}_1$ (with the factor map $\pi(\mathbf{x},\mathbf{y})=\mathbf{x}\cdot \mathbf{y}$, where $\cdot$ denotes pointwise multiplication of sequences) and the system $\mathcal{S}_3$ is a factor of $\mathcal{S}_2$ (with the factor map $\pi(\mathbf{x})=\mathbf{x}^2$, with the squaring map of a sequence taken pointwise). Note that $\mathcal{S}_1$ is a relatively Bernoulli extension of $\mathcal{S}_3$. By Thouvenot's theorem~\cite{thouvenot}, $\mathcal{S}_2$ must then be relatively Bernoulli over $\mathcal{S}_3$, and therefore, up to measure-theoretic isomorphism of dynamical systems, the system $\mathcal{S}_2$ is the direct product of $\mathcal{S}_3$ and a Bernoulli system. 
\end{proof}

We then turn to proving Proposition~\ref{prop:conj}. 

\begin{proof}[Proof of Proposition~\ref{prop:conj}] Let $f: N\to {-1,0,+1}$ be a non-pretentious multiplicative function. We may factorize $f=f'g$. where $f':\mathbb{N}\to \{-1,+1\}$ and $g:\mathbb{N}\to \{0,1\}$, $g(n)=|f(n)|$ are multiplicative. We may suppose that $\mathbb{D}(g,1;\infty)<\infty$, since otherwise by the triangle inequality and Delange's theorem~\cite[Thm. III.4.2]{Tenenbaum} similarly as in~\eqref{eq:correlate24} we have 
\begin{align*}
\lim_{x\to \infty}\frac{1}{x}\sum_{n\leq x}f(n+h_1)\cdots f_k(n+h_k)=0    
\end{align*}
for any distinct $h_1,\ldots, h_k\in \mathbb{N}$, which implies that the unique Furstenberg measure of $f$ (and of $|f|$)is the Bernoulli measure $\kappa^{\mathbb{Z}}$, where $\kappa$ is the measure on $\{-1,0,1\}$ supported on $0$. Now, since $f$ is non-pretentious by assumption, from the pretentious triangle inequality it follows that $f'$ is also non-pretentious. 

From the proof of Theorem~\ref{thm:furst2} it is clear that if $f'$ satisfies Elliott's conjecture (in  the sense that $f'$ satisfies~\eqref{eq:correlate7}), then $f'g$ has a unique Furstenberg system, which is isomorphic to the direct product of the Furstenberg system of $g$ and a Bernoulli system. Hence, it suffices to show that any non-pretentious $f':\mathbb{N}\to \{-1,+1\}$ satisfies~\eqref{eq:correlate7}. Let $a_1,\ldots, a_k,h_1,\ldots, h_k\in \mathbb{N}$ satisfy $a_ih_j\neq a_jh_i$ whenever $i\neq j$. Let $q=a_1\cdots a_k$, and factor $f'=f_1f_2$, where $f_2$ is the completely multiplicative function given on the primes by $f_2(p)=f'(p)$, and $f_1$ is the multiplicative function given by $f_1(p^{\ell})=f'(p^{\ell})f'(p)^{\ell}$ for $\ell\geq 1$. 
Then the left-hand side of~\eqref{eq:correlate7} is equivalent to
\begin{align}\label{eq:correlate25}
\frac{1}{x}\sum_{n\leq x}\prod_{j=1}^kf_1(a_jn+h_j)f_2\left(qn+h_j\frac{q}{a_j}\right)=o(1). 
\end{align}

The function $f_1$ is pretentious (in fact, $\mathbb{D}(f_1,1;\infty)=0$) and therefore we have the decomposition~\eqref{eq:rap} for $f_1$. Substituting that into~\eqref{eq:correlate25} and sending $\varepsilon\to 0$, it suffices to show that
\begin{align*}
\frac{1}{x}\sum_{n\leq x}\prod_{j=1}^kf_2\left(qn+h_j\frac{q}{a_j}\right)1_{n\equiv b_j\pmod{r_j}}=o(1) 
\end{align*}
for any $b_j,r_j\in \mathbb{N}$. Making a change of variables $m = qn$, we reduce to showing that
\begin{align*}
 \frac{1}{x}\sum_{m\leq qx}\prod_{j=1}^kf_2\left(m+h_j\frac{q}{a_j}\right)1_{m/q\equiv b_j\pmod{r_j}}1_{m\equiv 0\pmod{q}}=o(1).    
\end{align*}
If the simultaneous congruences $m/q\equiv b_j\pmod{r_j}$ and $m\equiv 0\pmod q$ are solvable in $m$, they have a unique solution $m\equiv b_j'\pmod{r_j'}$ for some $b_j', r_j'\in \mathbb{N}$. Let $s_j$ be distinct integers greater than $\max_{i\leq k}h_i$ such that $(b_j+s_j,r_j)=1$.  We can expand $1_{n\equiv b_j\pmod{r_j}}=1_{n+s_j\equiv b_j+s_j\mod{r_j}}$ as a linear combination of Dirichlet characters to reduce matters to showing that 
\begin{align}\label{eq:correlate26}
\frac{1}{x}\sum_{m\leq qx}\prod_{j=1}^kf_2\left(m+h_j\frac{q}{a_j}\right)\prod_{j=1}^{k}\chi_j(m+s_j)=o(1) 
\end{align}
for any Dirichlet characters $\chi_j\pmod{r_j'}$.

By Conjecture~\ref{conj_elliott_Tao}, we have~\eqref{eq:correlate26} provided that 
for any fixed $A\geq 1$ we have 
\begin{align*}
\inf_{|t|\leq x^{A}}\,\,\min_{\substack{\chi\pmod q\\q\leq A}}\mathbb{D}(f_2,\chi(n)n^{it};x)^2\geq A
\end{align*}
for all large enough $x$. But 
since $f_2$ is real-valued and non-pretentious, this follows from Lemma~\ref{le:real}. 
\end{proof}

\section{Hudson's conjecture}\label{proof_huds}

\subsection{A lemma on correlations of multiplicative functions}

For proving Theorem~\ref{thm_hudson}, we need the following explicit upper bound for correlations of multiplicative functions along a dense set of scales. 

\begin{lemma}\label{le_upper_bound_correlation}
Let $k\geq 1$, and let $f:\mathbb{N}\to \{-1,+1\}$ be multiplicative. Suppose that $\mathbb{D}(f,\chi;\infty)=\infty$ for all Dirichlet characters $\chi$. Then there exists a set $\mathcal{X}\subset \mathbb{N}$ with $\delta_{\log}^{+}(\mathcal{X})=1$ such that we have 
\begin{align*}
\limsup_{\substack{x\to \infty\\x\in \mathcal{X}}}\left|\frac{1}{x}\sum_{n\leq x}f(n+1)\cdots f(n+k)\right|\leq \frac{1}{2}    
\end{align*}
and
\begin{align*}
\lim_{\substack{x\to \infty\\x\in \mathcal{X}}}\frac{1}{x}\sum_{n\leq x}f(n)f(n+j)=0    
\end{align*}
for all $j\in \mathbb{N}$. 
\end{lemma}

\begin{proof}
We follow the proof of~\cite[Lemma 7.2]{tt-FMS}, where a similar statement was proved for logarithmic averages with $f=\lambda$. Let
\begin{align*}
S(x):=\sum_{n\leq x}f(n+1)\cdots f(n+k)   
\end{align*}
with $x\in \mathbb{N}$. By shifting the summation index, we have
\begin{align*}
S(x)=\sum_{n\leq x}f(n+2)\cdots f(n+k+1)+O(1).    
\end{align*}
Combining these and using the triangle inequality, we have
\begin{align*}
2|S(x)|&\leq \sum_{n\leq x}|f(n+1)\cdots f(n+k)+f(n+2)\cdots f(n+k+1)|+O(1)\\
&=\sum_{n\leq x}|f(n+1)+f(n+k+1)|+O(1).    
\end{align*}
 Note that 
\begin{align}\label{eq:signs1}\begin{split}
 &\sum_{n\leq x}|f(n+1)+f(n+k+1)|\\
 =&\sum_{v_1,v_2\in \{-1,+1\}}|v_1+v_2|\sum_{n\leq x}1_{f(n+1)=v_1}1_{f(n+k+1)=v_2}\\
 =&\sum_{v_1,v_2\in \{-1,+1\}}\frac{|v_1+v_2|}{4}\left(x+(v_1+v_2)\sum_{n\leq x}f(n)+v_1v_2\sum_{n\leq x}f(n)f(n+k+1)\right) + o(x)\\
 =& x+\sum_{n\leq x}f(n)f(n+k+1)+o(x)
 \end{split}
\end{align}
by Wirsing's theorem on mean values of real-valued multiplicative functions (see~\cite[Thm. III.4.5]{Tenenbaum}) and the fact that $\mathbb{D}(f,\chi;\infty)=\infty$ for all real characters $\chi$.

Now let $\mathcal{X}$ be the set in Theorem~\ref{thm_elliott_2point}. Then $\delta_{\log}^{+}(\mathcal{X})=1$ and 
\begin{align*}
\limsup_{\substack{x\to \infty\\x\in \mathcal{X}}}\left|\frac{1}{x}\sum_{n\leq x}f(n)f(n+j)\right|=0    
\end{align*}
for all $j\in \mathbb{N}$, so by~\eqref{eq:signs1} we obtain 
$$\limsup_{\substack{x\to \infty\\x\in \mathcal{X}}}\frac{1}{x}|S(x)|\leq \frac{1}{2}.$$
This completes the proof. 
\end{proof}

\subsection{Non-pretentious case}

In the case of non-pretentious functions, we obtain a stronger conclusion than what Theorem~\ref{thm_hudson} states.
\begin{theorem}[Length four sign patterns of non-pretentious functions]\label{thm_signs}
Let $f:\mathbb{N}\to \{-1,+1\}$ be a  multiplicative function such that $\mathbb{D}(f,\chi;\infty)=\infty$ for all real\footnote{In light of Lemma~\ref{le:real}, this condition implies that $f$ is non-pretentious in the sense of Definition~\ref{def:Pret}.} Dirichlet characters $\chi$. Let $v_1,v_2,v_3,v_4\in \{-1,+1\}$. Then
\begin{align*}
\delta_{\log}^{+}(\{n\in \mathbb{N}:\,\, f(n+j)=v_j\,\, \forall 1\leq j\leq 4\})\geq \frac{1}{32}.    
\end{align*}
\end{theorem}

\begin{proof}
Using $1_{f(n+j) = v}=(1+vf(n+j))/2$ for $v\in \{-1,+1\}$, we have
\begin{align}\label{eq:Ca}
\frac{1}{x}\sum_{n\leq x}\prod_{j=1}^{4}1_{f(n+j)=v_j}=\frac{1}{16}\left(1+\sum_{\substack{A\subseteq \{1,2,3,4\}\\ A \neq \emptyset}}C_{A}(x)\right),  
\end{align}
where
\begin{align*}
C_A(x):=\frac{1}{x}\sum_{n\leq x}\prod_{j\in A}v_jf(n+j).   \end{align*}

By Wirsing's theorem, we have $C_A(x)=o(1)$ for $|A|=1$. Moreover, by Lemma~\ref{le_upper_bound_correlation} there exists a set $\mathcal{X}_1$ with $\delta_{\log}^{+}(\mathcal{X}_1)=1$ such that for $x\in \mathcal{X}_1$ we have
\begin{align*}
\sum_{\substack{A\subset \{1,2,3,4\}\\|A|=2\textnormal{ or } |A|=4}}|C_A(x)|\leq \frac{1}{2}+o(1).    
\end{align*}

If $f$ is not moderately non-pretentious, then by Theorem~\ref{thm_A}(2) there exists a set $\mathcal{X}_2$ with $\delta_{\log}(\mathcal{X}_2)=1$ such that for $x\in \mathcal{X}_2$ and $|A|=3$ we have
\begin{align*}
 |C_A(x)|=o(1).    
\end{align*}
Then we can take $\mathcal{X}=\mathcal{X}_1\cap \mathcal{X}_2$ to obtain 
\begin{align}\label{eq:signs2}
\liminf_{\substack{x\to \infty\\x\in \mathcal{X}}}\frac{1}{x}\sum_{n\leq x}\prod_{j=1}^{4}1_{f(n+j)=1}\geq \frac{1}{32}.    
\end{align}
The claim follows in this case, since $\delta_{\log}^{+}(\mathcal{X}_1)=1$ and $\delta_{\log}(\mathcal{X}_2)=1$ imply $\delta_{\log}^{+}(\mathcal{X})=1$.

If instead $f$ is moderately non-pretentious, then applying Theorem~\ref{thm_main} with $k=5$ and $f_j\in \{f,1\}$ we see that there exists a set $\mathcal{X}_3$ with $\delta_{\log}^{+}(\mathcal{X}_3)=1$ such that for $x\in \mathcal{X}_3$ and $1\leq |A|\leq 4$ we have
\begin{align*}
 C_A(x)=o(1);   
\end{align*}
here we make crucial use of the fact that the dense set in Theorem~\ref{thm_main} depends only on the set $\{f_1,\ldots, f_5\}=\{f,1\}$ and not on the corresponding $5$-tuple. Taking $\mathcal{X}=\mathcal{X}_3$, we again obtain~\eqref{eq:signs2}. 
\end{proof}

\subsection{A reduction} 

By Theorem~\ref{thm_signs}, it suffices to prove Theorem~\ref{thm_hudson} for pretentious functions. 
The following lemma reduces this to the seemingly weaker claim that if $f:\mathbb{N}\to \{-1,+1\}$ is a pretentious completely multiplicative function and $f\not \in F_{2,3}$, then $f(n+1)=f(n+2)=f(n+3)=f(n+4)=+1$ for at least one $n\in \mathbb{N}$.

\begin{proposition}
\label{prop_pretentious_density}
Let $f:\mathbb{N}\to \{-1,+1\}$ be multiplicative, and suppose that $\mathbb{D}(f,\chi;\infty)<\infty$ for some real Dirichlet character $\chi$. Let $k\geq 1$, and let $v_1,\ldots, v_k\in \{-1,+1\}$. Then the set
\begin{align*}
\{n\in \mathbb{N}:\,\, f(n+1)=v_1,\ldots, f(n+k)=v_k\}    
\end{align*}
is either empty or has positive asymptotic lower density.
\end{proposition}

In the proof of this proposition and in several subsequent proofs in this section, we will use the following simple lemma.

\begin{lemma}\label{le_equalvalues} Let $Q\geq 2$ be an integer. Let $f:\mathbb{N}\to \{-1,+1\}$ be a multiplicative function such that for some Dirichlet character $\chi$ we have $f(p)=\chi(p)$ for all $p\nmid Q$. Then for any integers $\alpha\geq n+1$ and $m\in \mathbb{N}$ we have
\begin{align*}
f(n+mQ^{\alpha})=f(n).    
\end{align*}
\end{lemma}

\begin{proof} 
Since $\tfrac{n+mQ^{\alpha}}{(n+mQ^{\alpha},Q^{\alpha})}$ is coprime to $Q$ and $(n+mQ^{\alpha},Q^{\alpha})=(n+mQ^{\alpha},Q^{\alpha-1})$ by the assumption $\alpha\geq n+1$, we have
 \begin{align*}
f(n+mQ^{\alpha})&=f((n+mQ^{\alpha},Q^{\alpha}))f\left(\frac{n+mQ^{\alpha}}{(n+mQ^{\alpha},Q^{\alpha})}\right)\\
&=f((n+mQ^{\alpha},Q^{\alpha}))\chi\left(\frac{n+mQ^{\alpha}}{(n+mQ^{\alpha},Q^{\alpha})}\right)\\
&= f((n+mQ^{\alpha},Q^{\alpha}))\chi\left(\frac{n}{(n+mQ^{\alpha},Q^{\alpha})}\right)\\
&=f(n),
 \end{align*}
using the fact that the modulus of $\chi$ necessarily divides $Q$.
\end{proof}

\begin{proof}[Proof of Proposition~\ref{prop_pretentious_density}] Let $S$ be the set in question. Suppose that there is some $n_0\in S$. Let $\mathcal{P}=\{p\in \mathbb{P}:\,\ f(p)\neq \chi(p)\}$. Then by assumption
\begin{align}\label{eq:inverses}
\sum_{p\in \mathcal{P}}\frac{1}{p}<\infty.  \end{align}

Let $q$ be the modulus of $\chi$. Let $y>\max\{q,n_0+k\}$ be large, and factorize
\begin{align*}
f=f_{\leq y}f_{>y},\quad \textnormal{where }\,\, f_{\leq y}(p)=\begin{cases}
 f(p),\quad p\leq y\\ \chi(p), \quad p>y,\quad    
\end{cases} \quad  f_{>y}(p)=\begin{cases}
 1,\quad p< y\\ f(p)\chi(p), \quad p>y.\quad    
\end{cases}        
\end{align*}

Let $Q_y=\prod_{p\leq y}p$, and let $\alpha\geq n_0+k+1$ be an integer. Then by Lemma~\ref{le_equalvalues} for all $1\leq j\leq k$ and $m\in \mathbb{N}$ we have
\begin{align*}
f_{\leq y}(n_0+j+mQ_y^{\alpha})=f_{\leq y}(n_0+j).    
\end{align*}
Moreover, we also have 
\begin{align*}
f_{>y}(n_0+j+mQ_y^{\alpha})=f_{>y}(n_0+j)    
\end{align*}
if $n_0+j+mQ_y^{\alpha}$ is coprime to all the primes in $\mathcal{P}_{>y}:=\mathcal{P}\setminus [2,y]$.

For any $x$ large enough in terms of $y, \alpha, k$, we have by the union bound and~\eqref{eq:inverses} the estimate
\begin{align*}
\frac{1}{x}|\{m\leq x:\,\, \exists 1\leq j\leq k, p\in \mathcal{P}_{>y}:\,\, p\mid n_0+j+mQ_y^{\alpha}\}|&\leq \frac{k}{x}\sum_{\substack{p\in \mathcal{P}_{>y}\\p\leq 2xQ_y^{\alpha}}}\left(\frac{1}{p}+\frac{1}{x}\right)\\
&=o_{y\to \infty}(1).  \end{align*}  We conclude that if $y$, $\alpha$ are suitably chosen then the asymptotic lower density of $m\in \mathbb{N}$ for which $f(n_0+j+mQ_y^{\alpha})=f(n_0+j)$ for all $1\leq j\leq k$ is $\geq 0.99$, say. The claim follows.
\end{proof}

\subsection{The case of pretentious functions that are not modified characters}

Before continuing further, we record the following result of Schur mentioned in the introduction.
\begin{lemma}[Schur]\label{lem-3pattern}
If $f: \mb{N} \ra \{-1,+1\}$ is a completely multiplicative function such that $f \neq f_3^{\pm}$ then there is some $n \in \mb{N}$ such that $(f(n+1),f(n+2),f(n+3)) = (+1,+1,+1)$.
\end{lemma}
We will use this lemma to get a first structural result for $F_{2,3}.$

\begin{proposition}\label{prop:redtoModChar}
Suppose $f:\mb{N} \ra \{-1,+1\}$ is a completely multiplicative function and $f\in F_{2,3}$.  Suppose further that $\mathbb{D}(f,\chi;\infty) <\infty$ for some real primitive\footnote{Here, we allow $\chi$ to be the trivial character, i.e., with $q = 1$.} Dirichlet character modulo $q \geq 1$. Then, for any $p \nmid 6q$ we have $f(p) = \chi(p)$. 
\end{proposition}

The proof of the proposition requires the following basic observation about the functions $f_3^{\pm}$.

\begin{lemma}\label{lem:f3pm}
If $f \in \{f_3^+,f_3^-\}$, both of the length 4 sign patterns $\mbf{\e}_1 := (+1,+1,-1,+1)$, $\mbf{\e}_2 := (+1,+1,-1,-1)$ occur in the sequence $f$ along an arithmetic progression. 

Precisely, we have
\begin{align*}
&(f_3^+(n),f_3^+(n+1),f_3^+(n+2),f_3^+(n+3)) = \begin{cases} \mbf{\e}_1 \text{ if } n \equiv 9 \pmod{27} \\ \mbf{\e}_2 \text{ if } n \equiv 3 \pmod{9}
\end{cases} \\
&(f_3^-(n),f_3^-(n+1),f_3^-(n+2),f_3^-(n+3)) = \begin{cases} \mbf{\e}_1 \text{ if } n \equiv 6 \pmod{27} \\ \mbf{\e}_2 \text{ if } n \equiv 9 \pmod{27}
\end{cases} 
\end{align*}
\end{lemma}
\begin{proof}
If $f = f_3^+$ (so $f(3) = +1$) then by taking $n = 9(3m+1)$ for $m \in \mb{N}$ we have
\begin{align*}
(f(n),f(n+1),f(n+2),f(n+3))& = (\chi_3(3m+1),\chi_3(27m+10),\chi_3(27m+11),\chi_3(9m+4))\\& = (+1,+1,-1,+1),
\end{align*}
and if $n = 3(3m+1)$ then
\begin{align*}
(f(n),f(n+1),f(n+2),f(n+3))& = (\chi_3(3m+1),\chi_3(9m+4),\chi_3(9m+5),\chi_3(3m+2))\\& = (+1,+1,-1,-1).
\end{align*}

Similarly, if $f = f_3^-$ (so $f(3) = -1$) then, again taking $n = 9(3m+1)$ we get
\begin{align*}
&(f(n),f(n+1),f(n+2),f(n+3))\\
=& (f_1(3)^2\chi_3(3m+1),\chi_3(27m+10),\chi_3(27m+11), f_1(3)\chi_3(9m+4)) \\
=& (+1,+1,-1,-1),
\end{align*}
while taking $n = 3(9m+2)$ with $3\mid m$, we get
\begin{align*}
&(f(n),f(n+1),f(n+2),f(n+3))\\
=&(f_1(3)\chi_3(9m+2),\chi_3(27m+7),\chi_3(27m+8),f_1(3)^2\chi_3(3m+1)) \\
=& (+1,+1,-1,+1).
\end{align*}
\end{proof}

\begin{proof}[Proof of Proposition~\ref{prop:redtoModChar}]
We employ a simple variant of the ``rotation trick'' for multiplicative functions, as developed in~\cite{tams-kmpt} (for a different application, see also~\cite{KMT-EDP}). 

We begin by showing that at most one prime $p > 3$ with $p \nmid q$ yields $f(p)\chi(p) \neq 1$. 

Let $S := \{p > 3 : p\nmid q, \,  f(p)\chi(p) \neq 1\}$. We necessarily have
$$
\sum_{\ss{p \leq x \\ p \in S}} \frac{1}{p} = \frac{1}{2} \sum_{p \leq x} \frac{1-f(p)\chi(p)}{p} + O(1) = \frac{1}{2}\mb{D}(f,\chi;x)^2 + O(1) \ll 1.
$$

Assume for the sake of contradiction that $|S| \geq 2$, and select $p_1,p_2 \in S$ distinct. Let $k_1,k_2 \in \mb{N}$ and let $a := (6q)^{2}$. 
We claim that there exist arbitrarily large integers $m$ with $(m,6)=1$ for which 
\begin{align}\label{eq:congruences}\begin{split}
(am+j,p)&=1,\,\forall 1\leq j\leq 4,\, \, p \in S \bk \{p_1,p_2\},\\
p_1^{k_1} &\mid \mid am+2,\\
p_2^{k_2} &\mid \mid am+3.
\end{split}
\end{align}
Indeed, if $y\geq 2$, by the Chinese remainder theorem we can find integers $1\leq b_y\leq Q_y$ such that~\eqref{eq:congruences} holds for $m\equiv b_y\pmod{Q_y}$ when the set $S$ there is replaced with $S\cap [1,y]$. Now, the proportion of $m\equiv b_y\pmod{Q_y}$ for which $(am+j,p)=1$ for some $p>y, p\in S$ is by the union bound 
$$\ll \sum_{\substack{p\in S\\ p>y}}\frac{1}{p}=o_{y\to \infty}(1),$$
so if $y$ is large enough then at least proportion $0.99$ of $m\equiv b_y\pmod{Q_y}$ satisfy~\eqref{eq:congruences}.

Fix any $m$ satisfying~\eqref{eq:congruences}. Since $4q\mid a$, 
$f(am+j) = f(j)f(\frac{a}{j}m+1)=f(j) = +1$ for $j = 1,4$. Moreover for $j\in \{2,3\}$ we have $\chi((am+j)/j) = 1$, and so
\begin{align}\label{eq:twoprimes}
f(am+j) &= f(jp_{j-1}^{k_{j-1}})f(\tfrac{am+j}{jp_{j-1}^{k_{j-1}}})=f(j)f(p_{j-1})^{k_{j-1}}\chi(\tfrac{am+j}{jp_{j-1}^{k_{j-1}}}) = f(j)(f(p_{j-1})\chi(p_{j-1}))^{k_{j-1}}.
\end{align}
Since $f\in F_{2,3}$ and $p_1,p_2\in S$, it follows that
\begin{align}\label{eq:twoprimes2}\begin{split}
0 &= (1+f(am+1))(1+f(am+2))(1+f(am+3))(1+f(am+4)) \\
&= 4(1+f(2)(f(p_1)\chi(p_1))^{k_1})(1+f(3)(f(p_2)\chi(p_2))^{k_3}).
\end{split}
\end{align}
Since $f(p_j)=-\chi(p_j)$ for $j=1,2$, choosing $k_1,k_2$ such that $f(2) = (-1)^{k_1}$ and $f(3) = (-1)^{k_2}$ the right-hand side is non-zero, which is a patent contradiction. 

Next, we show that $S = \emptyset$. Suppose instead that there is a unique prime $p_0 > 3$ with $p_0 \nmid q$ and $f(p_0) = -\chi(p_0)$, with $f \in F_{2,3}$. Thus, $f$ is a modification of $\chi$ outside of the primes dividing $6qp_0$. 

Define $h_{p_0}$ to be the completely multiplicative function defined by $h_{p_0}(p) = 1$ if $p \neq p_0$ and $h_{p_0}(p_0) = -1$. Then $g := fh_{p_0}$ satisfies $g(p) = \chi(p)$ for all $p\nmid q$. We consider two cases.

\textbf{Case 1:} $g \neq f_3^{\pm}$.
By Lemma~\ref{lem-3pattern}, we may find $n \in \mb{N}$ such that either $(g(n+1),g(n+2),g(n+3),g(n+4)) = (+1,+1,+1,-1)$ or $(+1,+1,+1,+1)$. 
Let $\alpha\geq n+5$ be any integer, 
and let $k \in \mb{Z}$ be chosen so that
$$
n+k(6q)^{\alpha}+4 \equiv p_0 - \eta \pmod{p_0^2},
$$
where $\eta := 0$ if $(+1,+1,+1,-1)$ is produced, and otherwise $\eta := -1$
if $(+1,+1,+1,+1)$ is produced. Each of these may be done since $(p_0,q)=1$. By Lemma~\ref{le_equalvalues} we have
$$
g(n+j+k(6q)^{\alpha}) = g(n+j)=+1 \text{ for } j = 1,2,3, 
$$
and $g(n+4+k(6q)^{\alpha}) = (-1)^{\eta+1}$.
On the other hand, since $p_0 \nmid n+j+k(6q)^{\alpha}$ for $j = 1,2,3$, we have
$$
h_{p_0}(n+j+k(6q)^{\alpha}) = +1 \text{ for } j = 1,2,3,
$$
and also
$$
h_{p_0}(n+4+k(6q)^{\alpha}) = h_{p_0}(p_0-\eta) = (-1)^{\eta+1}.
$$
Thus, as $f = gh_{p_0}$, setting $n'=n+k(6q)^{\alpha}$ we find $(f(n'+1),f(n'+2),f(n'+3),f(n'+4))= (+1,+1,+1,+1)$,
which is also a contradiction. 

\textbf{Case 2:} $g \in \{f_3^+, f_3^-\}$

Assume first that $g = f_3^+$. Applying Lemma~\ref{lem:f3pm}, we have
$$
(g(n),g(n+1),g(n+2),g(n+3)) = (+1,+1,-1,+1)
$$
whenever $n = 27m + 9$. If we choose $m \in \mb{N}$ such that $n+2 = 27m + 11 \equiv p_0 \pmod{p_0^2}$ (which can be done since $p_0 \neq 3$) then $p_0 \nmid n+j$ for $j = 0,1,3$, and so
$$
(h_{p_0}(n),h_{p_0}(n+1),h_{p_0}(n+2), h_{p_0}(n+3)) = (+1,+1,-1,+1).
$$
Again using $f = gh_{p_0}$ we conclude that $(f(n),f(n+1),f(n+2),f(n+3)) = (+1,+1,+1,+1)$, a contradiction. 

A similar contradiction is obtained if $g = f_3^-$, picking instead $n = 27m+6$ and using Lemma~\ref{lem:f3pm}.

Thus, in all cases we get a contradiction if $S \neq \emptyset$, so the claim follows.
\end{proof}

The case $q = 1$ of Proposition~\ref{prop:redtoModChar} may be disposed of by a direct check.
\begin{corollary}\label{cor:f23}
Let $f \in F_{2,3}$ be such that $\mathbb{D}(f,1;\infty) < \infty$. Then $f(p)= + 1$ for all $p > 2$, and $f(2) = -1$ (i.e., $f = g_3$ in Conjecture~\ref{4pattern}).
\end{corollary}
\begin{proof}
By Proposition~\ref{prop:redtoModChar} we conclude that $f(p) = +1$ for all $p > 3$. 

If $f(2) = f(3) = +1$ then $f(n) = +1$ for all $n \in \mb{N}$, and thus $f \notin F_{2,3}$. 

If $f(2) = f(3) = -1$ then we have
$$
f(28) = f(7) = +1, \, f(29) = +1, \, f(30) = f(2)f(3)f(5) = +1, \, f(31) = +1,
$$
so $f \notin F_{2,3}$ also. 

Finally, if $f(2) = +1$ but $f(3) = -1$ then
$$
f(7) = +1, \, f(8) = f(2)^3 = +1, \, f(9) = f(3)^2 = +1, \, f(10) = f(2)f(5) = +1,
$$
hence $f \notin F_{2,3}$. 

The only remaining case is $f(2) = -1$, $f(3) = +1$, which is precisely $f = g_3$ in Hudson's classification.
\end{proof}

\subsection{The case of modified characters}
 In the sequel, let 
$$
\mc{D} := \{m \in \mb{N} : m = 2^k m', \, k \in \{0,2,3\}, \, 2 \nmid m' \text{ and } \mu^2(m') = 1\} \bk \{1\}.
$$
Given $m \in \mc{D}$ we denote by $\chi_m$ a primitive real character modulo $m$; when $8 \nmid m$ there is a unique such primitive character with conductor $m$, whereas if $8|m$ then there are two of them (corresponding to the two primitive real characters modulo $8$). For example, any odd prime $p$ belongs to $\mc{D}$, and $\chi_p$ is the Legendre symbol $\chi_p = (\tfrac{\cdot}{p})$.

We say that $f: \mb{N} \ra \{-1,+1\}$ is a \emph{modification} of a character $\chi_m$, $m \in \mc{D}$, if $f(p)$ differs from $\chi_m(p)$ at only finitely many primes $p$.

Combining Proposition\ref{prop:redtoModChar} and Corollary~\ref{cor:f23}, we know that any $f\in F_{2,3}$, $f\neq g_3$ is a modification of some real primitive Dirichlet character $\chi\pmod q$ at a subset of the primes dividing $6q$, with $q\geq 3$. But then $f$ is the modification of $\chi_m$ outside the primes dividing $6m$ for some $m\in \mathcal{D}$. The main goal in this subsection is to prove the following proposition,  which limits the possibilities for $m$, showing among other things that the set $F_{2,3}$ must be finite.

\begin{proposition} \label{prop:redToBddMod}
Let $f \in F_{2,3}$ be a function for which there is an $m \in \mc{D}$ such that $f(p) = \chi_m(p)$ for all $p \nmid 6m$. Then the following hold.
\begin{enumerate}
    \item Either $m \in \{4,8\}$, or else $m$ is a prime.

    \item We have $m < 197$.
\end{enumerate} 
\end{proposition}

The proof of Proposition~\ref{prop:redToBddMod} is based on the following lemma. 

\begin{lemma}\label{le:ProductsNotF23}
Let $m_1,m_2 \in \mc{D}$ be coprime. Let $f_1,f_2: \mb{N} \ra \{-1,+1\}$ be completely multiplicative functions satisfying $f_i(p) = \chi_{m_i}(p)$ whenever $p \nmid 6m_i$ for $i = 1,2$. Then $f_1f_2 \notin F_{2,3}$. 
\end{lemma}
\begin{proof} 
Let $S := \{p \in \{2,3\} : f_1(p) \neq \chi_{m_1}(p) \text{ and } f_2(p) \neq \chi_{m_2}(p)\},$ and for $p \in S$ define the completely multiplicative function $h_p$ by 
$$
h_p(p') = \begin{cases} -1 \text{ if } p' = p \\ 1 \text{ otherwise.} \end{cases}
$$
Observe that the functions $\tilde{f}_i := f_i \prod_{p \in S} h_p$, $i = 1,2$, satisfy the following properties: 
\begin{enumerate}[(i)]
\item $\tilde{f}_1\tilde{f}_2 = f_1f_2$ 
\item $\tilde{f}_i(p) = f_i(p) = \chi_{m_i}(p)$ for all $p \nmid 6m_i$, and 
\item at most one of $\tilde{f}_1\chi_{m_1}(p)$ and $\tilde{f}_2\chi_{m_2}(p)$ is $-1$ for $p \in \{2,3\}$.
\end{enumerate}
If we let $\tilde{m}_i := \prod_{p : \tilde{f}_i\chi_{m_i}(p) = -1} p$, $i = 1,2$, then because $m_1,m_2$ are coprime, so are $\tilde{m}_1$ and $\tilde{m}_2$. 

In light of (i), we will work with the functions $\tilde{f}_i$ in place of $f_i$ and show that $f_1f_2 = \tilde{f}_1\tilde{f}_2 \notin F_{2,3}$. We consider the following two cases.

\textbf{Case 1:} $\tilde{f}_1,\tilde{f}_2 \neq f_3^{\pm}$.

Observe that for $f \in \{\tilde{f}_1,\tilde{f}_2\}$ 
there is a positive integer $n = n_f$ for which 
\begin{equation}\label{eq:SchurCons}
(f(n+1), f(n+2),f(n+3),f(n+4)) = (-1,+1,+1,+1).
\end{equation}
Indeed, $f$ is a modification of a non-principal character modulo $6m$ for some $m$ and $f \neq f_3^{\pm}$, so by Lemma~\ref{lem-3pattern} we may find $n \in \mb{N}$ such that $(f(n+1),f(n+2),f(n+3)) = (+1,+1,+1)$. 

If $f \in F_{2,3}$ then, necessarily, $f(n) = -1$, and thus the length 4 pattern $(-1,+1,+1,+1)$ must arise. 

If, on the other hand, $f \notin F_{2,3}$ then it is possible for $(f(n),f(n+1),f(n+2),f(n+3)) = (+1,+1,+1,+1)$, and in fact by Proposition~\ref{prop_pretentious_density} such patterns occur with positive asymptotic lower density, so $n$ can be arbitrarily large. But as $f$ is a modification of a non-principal character, there must exist $k \in \mb{N}$, chosen minimally with $1 \leq k < n$, such that 
$$
f(n-j+1) = +1 \quad \text{ for } \quad 0 \leq j \leq k, \, f(n-k) = -1
$$ 
(otherwise there exist arbitrarily large $n$ such that $f(1)=f(2)=\cdots =f(n)$, in which case $f\equiv 1$, contradicting the assumption that $f$ is a modification of a non-principal character modulo $6m$). In that case, $(f(n-k),f(n-k+1),f(n-k+2),f(n-k+3)) = (-1,+1,+1,+1)$. Thus, in all cases, $f$ must produce the pattern $(-1,+1,+1,+1)$. 

From~\eqref{eq:SchurCons}, for each $i = 1,2$ there is $n_i \in \mb{N}$ such that 
$$
(\tilde{f}_i(n_i+1), \tilde{f}_i(n_i+2),\tilde{f}_i(n_i+3),\tilde{f}_i(n_i+4)) = (-1,+1,+1,+1).
$$ 
We use these to show that there is in fact $n \in \mb{N}$ that yields this pattern simultaneously for both $\tilde{f}_1,\tilde{f}_2$, i.e., that 
\begin{align}\label{eq:signs3}
(\tilde{f}_i(n+1),\tilde{f}_i(n+2),\tilde{f}_i(n+3),\tilde{f}_i(n+4)) = (-1,+1,+1,+1), \, i = 1,2.
\end{align}

Let $\alpha$ be large enough in terms of $n_1,n_2$. Then, as $\tilde{m}_1^{\alpha}, \tilde{m}_2^{\alpha}$ are coprime, we can find $k_1,k_2\in \mathbb{N}$ such that
\begin{align*}
k_1\tilde{m}_1^{\alpha}-k_2\tilde{m}_2^{\alpha}=n_2-n_1.     
\end{align*}
But then by Lemma~\ref{le_equalvalues} for $n=n_1+k_1\widetilde{m}_1^{\alpha}=n_2+k_2\widetilde{m}_2^{\alpha}$ for all $1\leq j\leq 4$ we have
\begin{align*}
 \tilde{f}_1(n+j)=\tilde{f}_1(n_1+j),\quad \tilde{f}_2(n+j) = \tilde{f}_2(n_2+j),    
\end{align*}
so the claim~\eqref{eq:signs3} follows.

It follows that there is $n \in \mb{N}$ such that $\tilde{f}_1(n+j) = \tilde{f}_2(n+j)$ for $1 \leq j \leq 4$, and thus $\tilde{f}_1\tilde{f}_2(n+j) = +1$ for $1 \leq j \leq 4$. Hence, $f_1f_2 \notin F_{2,3}$, as claimed.

\textbf{Case 2:} $\tilde{f}_1= f_3^{\pm}$ or $\tilde{f}_2=f_3^{\pm}$. 

By symmetry, we may assume that $\tilde{f}_1\in \{f_3^{+},f_3^{-}\}$. Then $m_1 = 3$, and necessarily $3 \nmid m_2$ and $\tilde{f}_2 \notin \{f_3^+,f_3^-\}$. 

By Lemma~\ref{lem:f3pm}, we can find $n_1^+, n_1^-$ such that
$$
(f_1(n_1^{\pm}+1),f_1(n_1^{\pm}+2),f_1(n_1^{\pm}+3),f_1(n_1^{\pm}+4)) = (+1,+1,-1,\pm 1). 
$$
Now, as $f_2 \neq f_3^{\pm}$, Lemma~\ref{lem-3pattern} implies the existence of $n_2$ such that
$$
(f_2(n_2),f_2(n_2+1),f_2(n_2+2), f_2(n_2+3)) = (+1,+1,+1,-1)
$$ 
(since $f_2$ is a modification of a non-principal character, $f_2$ must change sign eventually). This implies that either
$$
(f_2(n_2+1),f_2(n_2+2),f_2(n_2+3),f_2(n_2+4)) = (+1,+1,-1,+1) \text{ or } (+1,+1,-1,-1).
$$ 
Hence, we find that for some $\mbf{\e} \in \{(+1,+1,-1,+1), (+1,+1,-1,-1)\}$ both $f_1$ and $f_2$ must take the  sign pattern $\mbf{\e}$. Arguing as in Case 1, we can find a common $n \in \mb{N}$ such that
$$
(f_1(n+1),f_1(n+2),f_1(n+3),f_1(n+4)) = (f_2(n+1),f_2(n+2),f_2(n+3),f_2(n+4)) = \mbf{\e},
$$
and thus $(f_1f_2(n+1),\ldots,f_1f_2(n+4)) = (+1,+1,+1,+1)$, as claimed.
\end{proof}

With the above lemma proved, we are ready to prove Proposition~\ref{prop:redToBddMod}.

\begin{proof}[Proof of Proposition~\ref{prop:redToBddMod}] (1) Suppose instead that $m$ is not a prime and $m \notin \{4,8\}$.  
Write $m = m_0d$, where $m_0$ is either prime or in $\{4,8\}$, and $d > 1$ is coprime to $m_0$. Whenever $p \nmid 6m_0d$ we have $f(p) = \chi_{m_0}(p)\chi_d(p)$, and so we can factor $f=f_{m_0}f_{d_0}$, where $f_{m_0},f_d$ are completely multiplicative functions given by
$$
f_{m_0}(p) = \begin{cases} \chi_{m_0}(p) &\textnormal{if } p \nmid [6,m_0] \\ f(p)\chi_d(p) &\textnormal{if } p \mid m_0, \, p \nmid 6 \\ f(p) &\textnormal{if } p\mid 6 \end{cases} \quad \quad f_d(p) = \begin{cases} \chi_d(p) &\textnormal{if } p \nmid [6,d] \\ f(p)\chi_{m_0}(p) &\textnormal{if } p\mid d, \, p \nmid 6 \\ 1 &\textnormal{if } p\mid 6.\end{cases}
$$
Then $f_{m_0}(p) = \chi_{m_0}(p)$ for all $p \nmid 6m_0$ and $f_d(p) = \chi_d(p)$ for all $p \nmid 6d$. But by Lemma~\ref{le:ProductsNotF23}, $f\notin F_{2,3}$, which is a contradiction. Thus, $m$ must be either prime or equal to $4$ or $8$. 

(2) If $m \in \{3,4,8\}$ then the claim $m < 197$ is trivial. Thus, assume next that $m > 3$ is prime, writing $m = p_0$. 

Suppose that $f \in F_{2,3}$, with $f(p) = \chi_{p_0}(p)$, outside, possibly, of $p = 2,3$ and $p_0$.
For $1 \leq j \leq 4$ let $1 \leq a_j \leq p_0$ be such that $36a_j+ j\equiv 0 \pmod{p_0}$, and let $a \in \{1,\ldots,p_0\} \backslash \{a_1,a_2,a_3,a_4\}$. Then
$$
0 = \prod_{1 \leq j \leq 4} (1+f(36a+j)) = \prod_{1 \leq j \leq 4} (1+f(j) \chi_{p_0}(\tfrac{36}{j}a+1)) = \prod_{1 \leq j \leq 4} (1 + f\chi_{p_0}(j) \chi_{p_0}(36a+j)).
$$
Summing over $1 \leq a \leq p_0$, upper bounding the contribution of the above product at each of $a \in \{a_1,a_2,a_3,a_4\}$ trivially by $8$, we obtain
\begin{align*}
0 &= \sum_{\substack{1 \leq a \leq p_0 \\ a \notin \{a_1,a_2,a_3,a_4\}}} \left(1+f(36a+1)\right)\left(1+f(36a+2)\right)\left(1+f(36a+3)\right) \left(1+f(36a+4)\right) \\
&\geq  \sum_{a = 1}^{p_0} \prod_{1 \leq j \leq 4} \left(1+f\chi_{p_0}(j)\chi_{p_0}(36a+j)\right) -4\cdot 8.
\end{align*}
Expanding out the product, we obtain
$$
0 \geq  p_0 + \sum_{i = 1}^4 \sum_{\substack{S \subseteq \{1,2,3,4\} \\ |S| = i}} f\chi_{p_0}\left(\prod_{j \in S} j \right) \sum_{a = 1}^{p_0} \chi_{p_0}\left(\prod_{j \in S} (36a+j)\right) -32,
$$
and so by the triangle inequality we find
\begin{align} \label{eq:charSumBd}
p_0 \leq 32 + \sum_{j = 1}^4 \sum_{\substack{S \subseteq \{1,2,3,4\} \\ |S| = j}} \left|\sum_{a = 1}^{p_0} \chi_{p_0}\left(\prod_{j \in S} (36a+j)\right)\right|.
\end{align}
For each $S \subseteq \{1,2,3,4\}$ define $\Xi_S := \sum_{a = 1}^{p_0} \chi_{p_0}\left(\prod_{j \in S} (36a+j)\right)$ (note that the polynomial $X \mapsto \prod_{j \in S}(36X+j)$ is not a square modulo $p_0$ since $p_0 > 3$). 

When $|S| = 1$ the orthogonality of Dirichlet characters (as well as $p_0 \nmid 36$) implies that $\Xi_S = 0$. 

When $|S| = 2$, say $S = \{i,j\}$, then as $p_0 > 3$ the element $j-i$ is invertible modulo $p_0$. Hence, on $\mathbb{F}_{p_0}$ we have the bijective transformations $a \mapsto 36a =: b$, $b \mapsto b+i =: c$, $c \mapsto -(j-i)^{-1}c =: d$ that give rise to
\begin{align*}
\Xi_{\{i,j\}} &= \sum_{a = 1}^{p_0} \chi_{p_0}\left((36a+i)(36a+j)\right)  = \sum_{b \pmod{p_0}}\chi_{p_0}\left((b+i)(b + j)\right)\\
&= \sum_{c \pmod{p_0}}\chi_{p_0}\left(c(c + j-i)\right) \\
&= \chi_{p_0}\left(-1\right) \sum_{d \pmod{p_0}}\chi_{p_0}\left(d(1-d)\right) = \chi_{p_0}\left(-1\right) J(p_0),
\end{align*}
where $J(p_0)$ is the Jacobi sum attached to the Legendre symbol modulo $p_0$. It is a classical result (e.g.,~\cite[(3.20)]{iwaniec-kowalski}) that $J(p_0) = -1$, and thus $|\Xi_S| = 1$ whenever $|S| = 2$. 

Finally, if $|S| \in \{3,4\}$, the Weil bound (see e.g.~\cite[Corollary 11.24]{iwaniec-kowalski}) implies that
$$
|\Xi_S| \leq (|S|-1) \sqrt{p_0}.
$$
Applied to~\eqref{eq:charSumBd}, the above remarks show that
$$
p_0 \leq 32 + 6 \cdot 1 + 4 \cdot 2\sqrt{p_0} + 3\sqrt{p_0} = 38 + 11 \sqrt{p_0}.
$$
From this, we deduce by completing the square that $(\sqrt{p_0} - 11/2)^2 \leq 273/4$, and thus $p_0 \leq 196$, and the claim follows.
\end{proof}

It remains to classify those completely multiplicative functions $f: \mb{N} \ra \{-1,+1\}$ for which there is $m < 197$, with either $m \in \{4,8\}$ or $m$ a prime, such that
\begin{equation}\label{eq:excepPrimes}
f(p) = \chi_m(p) \text{ whenever } p \nmid 6m.
\end{equation}
To deal with the case of $m \notin \{3,4,8\}$ (thus prime, say $m = p_0 > 3$, and $\chi_{p_0} = (\tfrac{\cdot}{p_0})$) we use the following simple criterion, which helps with a computer verification.
\begin{lemma} \label{lem:forCompVer}
Let $p_0  > 3$ be a prime, and let $f: \mb{N} \ra \{-1,+1\}$ be a completely multiplicative function such that $f(p)=\chi_{p_0}(p)$ whenever $p\nmid 6p_0$. Then if there is $n\in \mb{N}$ satisfying $24n + 4 < p_0^3$ such that
\begin{equation}\label{eq:prodCharsExp}
\prod_{j=1}^4 \left(1+f(j)\left(\frac{j}{p_0}\right)\left(\frac{24n+j}{p_0}\right)\right) > 0,
\end{equation}
we have $f \notin F_{2,3}$.
\end{lemma}
\begin{proof}
Suppose~\eqref{eq:prodCharsExp} holds for some $n_0 \in \mb{N}$ with $24n_0 + 4 < p_0^3$, but $f \in F_{2,3}$. For $p\neq p_0$ set $g(p) := f(p)(\tfrac{p}{p_0})$, and extend $g$ to a completely multiplicative function.

By assumption, there is some $1 \leq j \leq 4$ such that
$$
f(24n_0+j) = f(j)f(24n_0/j+1)  \neq +1.
$$ 
If $p_0 \nmid (24n_0/j+1)$ for each $1 \leq j \leq 4$ then $f(24n_0+j_0) = g(j_0)\left(\frac{24n_0+j_0}{p_0}\right) \neq +1$ for some $1 \leq j_0 \leq 4$, hence~\eqref{eq:prodCharsExp} fails, a contradiction. Thus, we must have $p_0\mid 24n_0+j_0$ for exactly one $1 \leq j_0 \leq 4$, with $f(24n_0+j) = +1$ for $j\in \{1,2,3,4\}\setminus \{j_0\}$ necessarily.

Let $b_0 \equiv n_0 \pmod{p_0}$. Then for any $n' \equiv b_0 \pmod{p}$ and $j \neq j_0$, $1 \leq j \leq 4$, we have
\begin{align}\label{eq:24n}
f(24n'+j) = g(j)\left(\frac{24n'+j}{p_0}\right) = g(j)\left(\frac{24b_0+j}{p_0}\right) = g(j)\left(\frac{24n_0+j}{p_0}\right) = +1.
\end{align}
Suppose $p_0^{\nu}\mid \mid 24n_0+j_0$; by assumption, $\nu \in \{1,2\}$. Write $24n_0 + j_0 = N_0p_0^{\nu}$. Then
$$
-1 = f(24n_0+j_0) = f(p_0)^{\nu} \left(\frac{N_0}{p_0}\right).
$$
If $\nu = 1$ then shifting $n_0$ to $n = n_0 + N_1p_0$ with $N_1$ chosen to satisfy $N_0+24N_1 \equiv p_0 \pmod{p_0^2}$, then $f(24n+j) = +1$ for $j \neq j_0$ by~\eqref{eq:24n}, $24n+j_0 \equiv p_0^2 \pmod{p_0^3}$, and thus
$$
f(24n+j_0) = f(p_0)^2\left(\frac{1}{p_0}\right) = +1.
$$ 
Suppose then that $\nu = 2$, so that $N_0$ must be a quadratic non-residue modulo $p_0$. This time, we shift $n_0$ to $n = n_0 + N_1p_0^2$, where $N_0 + 24N_1 \equiv 1 \pmod{p_0}$. Thus, $f(24n+j) = +1$ for $j \neq j_0$ by~\eqref{eq:24n} and $24n+j_0 \equiv p_0^2 \pmod{p_0^3}$, once again. Thus, in this case too, we obtain 
$$
f(24n+j_0) = f(p_0)^2\left(\frac{1}{p_0}\right) = +1.
$$ 
In either case, we find $n$ such that
$$
f(24n+j) = +1 \text{ for all } 1 \leq j \leq 4,
$$
and thus $(f(24n+1),f(24n+2),f(24n+3),f(24n+4)) = (+1,+1,+1,+1)$, contradicting $f \in F_{2,3}$. 
\end{proof}

\subsection{Exceptional examples} 
In the Mathematica notebook found alongside the arXiv submission of this paper, a verification of~\eqref{eq:prodCharsExp} for modifications of $\chi = \left(\frac{\cdot}{p_0}\right)$ at primes dividing $6p_0$ is made for all $3 \leq  p_0 < 197$.  The examples $f_p^{\pm}$ with $p \in \{5,7,11,13,53\}$ from Conjecture~\ref{4pattern} are identified there, along with the following exceptional cases \emph{not} included in Hudson's classification (bolded in the Mathematica file), but for which~\eqref{eq:prodCharsExp} does not hold for any $n$ in the range $24 n + 4 < p_0^3$ (with either choice $f(p_0) = \pm 1$):
\begin{enumerate}[(i)]
\item $p_0 = 3$, and\footnote{The cases $p_0 =3$, $(f(2),f(3)) \in \{(-1,+1), (-1,-1)\}$ correspond to $f_3^{\pm}$, which belong instead to $F_{2,2}$ by Lemma~\ref{lem-3pattern}, and may thus be ignored.} $(f(2),f(3)) = (+1,+1)$ or $(+1,-1)$ 
\item $p_0 = 5$, $(f\chi(2),f\chi(3)) =(+1,-1)$
\item $p_0 = 5$,$(f\chi(2),f\chi(3)) =(-1,+1)$
\item $p_0 = 7$, $(f\chi(2),f\chi(3)) = (-1,+1)$
\item $p_0 = 13$, $(f\chi(2),f\chi(3)) =(-1,-1)$
\item $p_0 = 17$, $(f\chi(2),f\chi(3)) =(-1,-1)$.
\end{enumerate}
For each of these cases one can, by hand, recover $n$ for which $f(n+j) = +1$ for each $0 \leq j \leq 3$:
\begin{enumerate}[(i)]
\item $n =1$ for $(+1,+1)$, and $n = 13$ for $(+1,-1)$:
\begin{align*}
&f(1) = +1, \, f(2) = +1, \, f(3) = +1, \, f(4) = f(2)^2 = +1 \\
&f(13) = \chi(13) = +1,  \, f(14) = f(2)\chi(7) = +1, \, f(15) = f(3)\chi(5) = +1, f(16) = f(4)^2 = +1.
\end{align*}
\item $n = 99$:
\begin{align*}
&f(99) = f(3)^2\chi(11) = +1, \,f(100) = f(10)^2 = +1, \, f(101) = \chi(101) = +1, \\
&f(102) = -\chi(2)\chi(3)\chi(17) = +1.
\end{align*}
\item $n = 8$ if $f(5) = +1$ and $n = 29$ if $f(5) = -1$:
\begin{align*}
&f(8) = -\chi(2) = +1, \,f(9) = f(3)^2 = +1, \, f(10) = -\chi(2)f(5) = +1, \, f(11) = \chi(1) = +1 \\
&f(29) = \chi(4) = +1, \, f(30) = -\chi(2)\chi(3)f(5) = +1, f(31) = \chi(31) = + 1,\\
&f(32)= (-\chi(2))^5 = +1.
\end{align*}
\item $n=23$:
\begin{align*}
f(23) &= \chi(2) = +1, \,f(24) = (-\chi(2))^3\chi(3) = +1, \, f(25) = f(5)^2 = +1\\ f(26) &= -\chi(2)\chi(13) = +1.
\end{align*}
\item $n = 15$:
$$
f(15) = -\chi(3)\chi(5) = +1, \,f(16) = f(4)^2 = +1, \, f(17) = \chi(4) = +1, \, f(18) = -\chi(2) = +1.
$$
\item $n = 46$: 
$$
f(46) = -\chi(2)\chi(6) = +1, \,f(47) = \chi(-4) = +1, \, f(48) = -\chi(3) = +1, \, f(49) = f(7)^2 = +1.
$$
\end{enumerate}
To complete the classification, we must verify that, of the modifications of the primitive characters modulo $4$ and $8$, i.e.\footnote{Here $\chi_8$ is the primitive \emph{odd} real character modulo $8$ with $\chi_8(3) = +1,$ $\chi_8(5) = -1$, and $\psi_8$ is the primitive \emph{even} real character modulo $8$ with $\psi_8(3) = -1 = \psi_8(5)$.} $\chi_4, \chi_8$ and $\psi_8$, any functions other than\footnote{It is straightforward to check that the functions $g_1,g_2$ mentioned in Conjecture~\ref{4pattern} coincide with $\chi_4^{\pm}$.} $\chi_4^{\pm}$ (i.e., $f(3) = -1$, $f(2) = \pm 1$) do not belong to $F_{2,3}$.
This is summarized as follows (the interested reader may verify that the pattern $(+1,+1,+1,+1)$ arises in each case):
\begin{enumerate}[(i)]
\item $m = 4$, $(f(2),f(3)) = (+1,+1):$ $n =1$.
\item $m = 4$, $(f(2),f(3)) = (-1,+1):$ $n =13$.
\item $m = 8$ ($\chi_8$ case), $(f(2),f(3)) = (+1,+1)$: $n = 1$. 
\item $m = 8$ ($\chi_8$ case), $(f(2),f(3)) = (+1,-1)$: $n = 15$.
\item $m = 8$ ($\chi_8$ case), $(f(2),f(3)) = (-1,+1)$: $n = 9$.
\item $m = 8$ ($\chi_8$ case), $(f(2),f(3)) = (-1,-1)$: $n = 14$.
\item $m = 8$ ($\psi_8$ case), $(f(2),f(3)) = (+1,+1)$: $n = 1$. 
\item $m = 8$ ($\psi_8$ case), $(f(2),f(3)) = (+1,-1)$: $n = 14$.
\item $m = 8$ ($\psi_8$ case), $(f(2),f(3)) = (-1,+1)$: $n = 25$.
\item $m = 8$ ($\psi_8$ case), $(f(2),f(3)) = (-1,-1)$: $n = 22$.
\end{enumerate}

\section{Density version of the Erd\H{o}s discrepancy theorem}\label{proof_edp}

We now turn to the proof of Theorem~\ref{thm_edp}. We consider the cases of non-pretentious and pretentious functions separately.

\subsection{Non-pretentious case}

\begin{proof}[Proof of Theorem~\ref{thm_edp}(1)] We may assume that $M$ is a large enough integer. Let 
$$S(x)=\sum_{n\leq x}f(n).$$ 
Let $\mathcal{E}_M:=\mathbb{N}\setminus \mathcal{X}_M$. Suppose that $\delta_{\log}^{+}(\mathcal{X}_M)<M^{-3}$. Then $\delta_{\log}^{-}(\mathcal{E}_M)>1-M^{-3}$. Let $H=1000M^2$ and $\mathcal{E}_M'=\mathcal{E}_M\cap (\mathcal{E}_M+H)$. Then, by the union bound, we have $\delta_{\log}^{-}(\mathcal{E}_M')>1-2M^{-3}$, and for $n\in \mathcal{E}_M'$ we have
\begin{align}\label{eq4}
|S(n+H)-S(n)|\leq 2M.    
\end{align}
Since~\eqref{eq4} holds for all $n$ in a set of logarithmic lower density $>1-2M^{-3}$, we have
\begin{align*}
 \frac{1}{\log x}\sum_{n\leq x}\frac{|S(n+H)-S(n)|^2}{n}\leq \left(1-2M^{-3}\right)\cdot (2M)^2+2M^{-3}\cdot H^2+o(1)\leq \frac{H}{10},   
\end{align*}

On the other hand, we have 
\begin{align}\label{eq11}\begin{split}
\frac{1}{\log x}\sum_{n\leq x}\frac{|S(n+H)-S(n)|^2}{n}&=\sum_{0\leq h\leq |H|}(H-|h|)\frac{1}{\log  x}\sum_{n\leq x}\frac{f(n)f(n+h)}{n}\\
&= H-o(1)+\sum_{1\leq h\leq |H|}(H-|h|)\frac{1}{\log x}\sum_{n\leq x}\frac{f(n)f(n+h)}{n}. 
\end{split}
\end{align}
By Theorem~\ref{thm_elliott_2point}, there exists a set $\mathcal{Y}_M\subset \mathbb{N}$ with $\delta_{\log}^{+}(\mathcal{Y}_M)=1$ such that for $x\in \mathcal{Y}_M$  each of the sums over $n$ is $\leq H^{-2}(\log x)$, say. We conclude that for large enough $x$ we have
\begin{align*}
 \frac{1}{\log x}\sum_{n\leq x}\frac{|S(n+H)-S(n)|^2}{n}\geq \frac{H}{2},   
\end{align*}
say. This contradicts~\eqref{eq11}, so the theorem follows. 
\end{proof}

\subsection{Pretentious case}

\begin{proof}[Proof of Theorem~\ref{thm_edp}(2)] Let $$S(x)=\sum_{n\leq x}f(n),$$
and let $M\geq 1$. By the Erd\H{o}s discrepancy theorem of Tao~\cite{tao-edp}, we know that there is an integer $H$ such that $|S(H)|\geq 2M$. Now, by Proposition~\ref{prop_pretentious_density}, there exists a set $\mathcal{Y}_M$ of positive asymptotic lower density such that for $n\in \mathcal{Y}_M$ we have $f(n+j)=f(j)$ for all $1\leq j\leq H$. Then $|S(n+H)-S(n)|=|S(H)|\geq 2M$. It follows that for every $n\in \mathcal{Y}_M$ either $|S(n)|\geq M$ or $|S(n+H)|\geq M$, so $|S(m)|\geq M$ for a positive asymptotic lower density of $m\in \mathbb{N}$. 
\end{proof}

\bibliography{refs}
\bibliographystyle{plain}

\end{document}